\documentclass[11pt]{amsart}

\usepackage{amssymb, amsmath, amscd, amsthm, graphicx, psfrag,rotating}
\usepackage{enumerate}
\usepackage[matrix,arrow,curve]{xy}
\xyoption{dvips}              
\usepackage{wasysym}
\usepackage{epstopdf}

\usepackage{amssymb, amsmath, amscd, amsthm,
color, epsfig
}
\usepackage[english]{babel}

\title[Smoothing theory delooping of disc diffeomorphism and embedding spaces]{On the smoothing theory delooping of disc diffeomorphism and embedding spaces}
\author{Paolo Salvatore}
\address{University of Rome Tor Vergata}
\email{salvator@mat.uniroma2.it}
\author{Victor Turchin}
\address{Kansas State University}
\email{turchin@ksu.edu}


\thanks{The first author acknowledges the MUR Excellence Department Project MatMod@TOV awarded to the Department of Mathematics, University of Rome Tor Vergata, CUP E83C23000330006, and the support by the INDAM group GNSAGA. The second author was partially supported by the Simons Foundation awards \#519474, \#933026. The work was partially completed at the Max Planck Institute for Mathematics, Bonn, where both authors were research visitors in the Fall 2020 -- Spring 2021.}

\newtheorem{theorem}{Theorem}[section]
\newtheorem{corollary}[theorem]{Corollary}
\newtheorem{lemma}[theorem]{Lemma}
\newtheorem{proposition}[theorem]{Proposition}
\newtheorem{conjecture}[theorem]{Conjecture}

\newtheorem{mainthm}{Theorem}

\theoremstyle{definition}

\newtheorem{remark}[theorem]{Remark}

\newcommand{\Z}{{\mathbb Z}}

\newcommand{\R}{{\mathbb R}}

\newcommand{\Ebarmn}{{\overline{\mathrm{Emb}}}_\partial(D^m,D^n)}
\newcommand{\Embmn}{\mathrm{Emb}_\partial(D^m,D^n)}
\newcommand{\Embfrmn}{\mathrm{Emb}_\partial^{fr}(D^m,D^n)}

\newcommand{\tEmbmn}{\mathrm{tEmb}_\partial(D^m,D^n)}
\newcommand{\plEmbmn}{\mathrm{plEmb}_\partial(D^m,D^n)}

\newcommand{\Ebar}{\overline{\mathrm{Emb}}}
\newcommand{\Emb}{\mathrm{Emb}}
\newcommand{\Embfr}{\mathrm{Emb}^{fr}}
\newcommand{\Imm}{\mathrm{Imm}}
\newcommand{\tImm}{\mathrm{tImm}}
\newcommand{\plImm}{\mathrm{plImm}}
\newcommand{\pdImm}{\mathrm{pdImm}}

\newcommand{\tEmb}{\mathrm{tEmb}}
\newcommand{\plEmb}{\mathrm{plEmb}}
\newcommand{\pdEmb}{\mathrm{pdEmb}}
\newcommand{\Map}{\mathrm{Map}}
\newcommand{\dMap}{\mathrm{dMap}}
\newcommand{\plMap}{\mathrm{plMap}}
\newcommand{\pdMap}{\mathrm{pdMap}}

\newcommand{\Diffn}{\mathrm{Diff}_\partial(D^n)}
\newcommand{\Dbarn}{\overline{\mathrm{Diff}}_\partial(D^n)}
\newcommand{\Homeon}{\mathrm{Homeo}_\partial(D^n)}
\newcommand{\Diff}{\mathrm{Diff}}
\newcommand{\Dbar}{\overline{\mathrm{Diff}}}
\newcommand{\Homeo}{\mathrm{Homeo}}
\newcommand{\Hombar}{\overline{\mathrm{Homeo}}}
\newcommand{\plHomeo}{\mathrm{plHomeo}}
\newcommand{\pdHomeo}{\mathrm{pdHomeo}}

\newcommand{\Operad}{\mathrm{Operad}}

\newcommand{\PL}{\mathrm{PL}}
\newcommand{\TOP}{\mathrm{TOP}}
\newcommand{\OO}{\mathrm{O}}
\newcommand{\GG}{\mathrm{G}}
\newcommand{\PD}{\mathrm{PD}}
\newcommand{\VV}{\mathrm{V}}

\newcommand{\mmod}{\,\mathrm{mod}\,}

\newcommand{\hofiber}{\mathrm{hofiber}}

\newcommand{\Assoc}{\mathrm{Assoc}}

\newcommand{\bbslash}{\backslash\!\!\backslash}

\newcommand\rth{\refstepcounter{equation}}
\newcommand\numb{\rth{\rm \theequation}}
\numberwithin{equation}{section}

\begin{document}

\sloppy

\begin{abstract}
The celebrated Morlet-Burghelea-Lashof-Kirby-Siebenmann smoothing theory theorem states that the group $\Diffn$  of  diffeomorphisms of a disc~$D^n$ relative to the boundary is equivalent to $\Omega^{n+1}\left(\PL_n/\OO_n\right)$ for any $n\geq 1$ and to $\Omega^{n+1}\left(\TOP_n/\OO_n\right)$ for $n\neq 4$. We revise smoothing theory results to show that the delooping generalizes to different versions of  disc smooth embedding spaces relative to the boundary, namely the usual embeddings, those modulo immersions, and framed embeddings.
The latter spaces deloop as 
$\Embfrmn\simeq\Omega^{m+1}\left(\OO_n\bbslash\PL_n/\PL_{n,m}\right)\simeq \Omega^{m+1}\left(\OO_n\bbslash\TOP_n/\TOP_{n,m}\right)$ for
 any $n\geq m\geq 1$ ($n\neq 4$ for the second equivalence), where the left-hand side in the case $n-m=2$ or $(n,m)=(4,3)$ should be
  replaced by the union of the path-components of $\PL$-trivial knots (framing being disregarded). Moreover, we show that for $n\neq 4$, the delooping is compatible with the Budney $E_{m+1}$-action. We use this delooping to combine the Hatcher $\OO_{m+1}$-action and the Budney $E_{m+1}$-action into a framed little discs operad $E_{m+1}^{\OO_{m+1}}$-action on  $\Embfrmn$.


\end{abstract}

\maketitle

\section{Introduction}\label{s:intro}
The groups $\Diffn$ of disc diffeomorphisms
relative to the boundary 
 and the spaces $\Embmn$ of smooth embeddings of discs are important objects in differential topology and were 
subject of active research in recent years \cite{AT1,BW_conf_cat,Budney_cubes,Budney_family,Budney_splicing,DT,FTW,KraRand,Kupers,KupRand,Randal_diff,Sakai,SakWat,Watanabe,WeissRationalP,Will,Yoshioka}. 
Some of this tremendous progress was reported at the International Congress of Mathematicians  by T.~Willwacher in 2018~\cite{Will}  and by O.~Randal-Williams in 2022~\cite{Randal_diff}.

A famous result from the 70s \cite{BL_diff, KS_essays} (see also \cite{Cerf_diff,Morlet}) asserts that one has equivalences
\[
\Diffn\simeq \Omega^{n+1}\left(\PL_n/\OO_n\right)\simeq_{n\neq 4} \Omega^{n+1}\left(\TOP_n/\OO_n\right).
\eqno(\numb)\label{eq:diff_deloop}
\]
This result is an important corollary of the smoothing theory. On the other hand, in a much more recent work~\cite{Budney_cubes} Ryan Budney defined an
$E_{n+1}$-action on $\Diffn$. We show that for $n\neq 4$ this action is compatible with the delooping.

\begin{mainthm}\label{th:a} 

There are equivalences of $E_{n+1}$-algebras 
\begin{enumerate}[(a)]
\item $\Diffn\simeq_{E_{n+1}} \Omega^{n+1}\left(\TOP_n/\OO_n\right)\simeq_{E_{n+1}} \Omega^{n+1}\left(\PL_n/\OO_n\right)$, $n\neq 4$,
\item $\Dbarn\simeq_{E_{n+1}} \Omega^{n+1} \TOP_n\simeq_{E_{n+1}} \Omega^{n+1}\PL_n$, $n\neq 4$,
\end{enumerate}
and equivalences of spaces
\begin{enumerate}[(c)]
\item $\Diff_\partial(D^4)\simeq \Omega^5\left(\PL_4/\OO_4\right)$, $\Dbar_\partial(D^4)\simeq \Omega^5\PL_4$.
\end{enumerate}
\end{mainthm}

The group $\Dbarn$ in the statement is the homotopy fiber of the derivative map 
$
\Dbarn:=\hofiber(\Diffn\to \Omega^n O_n), 
$ often called the  space of diffeomorphisms of the disc modulo immersions.\footnote{In the literature it is also called the group of framed diffeomorphisms of the disc,
see for example~\cite{KupRand}.}  
Note that $\Omega^{n+1}\TOP_n\simeq \Omega^{n+2} B\TOP_n$ is an $(n+2)$-loop space. It is an interesting question to construct an explicit $E_{n+2}$-action on $\Dbarn$. 
The statement (c) in the theorem is mostly to emphasize the fact that our methods can only produce an equivalence of spaces when $n=4$, not even of $E_1$-algebras.

The smoothing theory also allows one to deloop the  spaces $\Embmn$ of smooth embeddings relative to the boundary, together with the closely related spaces of framed embeddings $\Embfrmn$, and the spaces of embeddings modulo immersions 
\begin{multline*}
\Ebarmn:=\hofiber(\Embmn\to \Omega^m(O_n/O_{n-m}))\\
\simeq\hofiber(\Embfrmn\to\Omega^mO_n).
\end{multline*}
While hints to such deloopings appeared already in Lashof's paper~\cite{Lashof}, they were actually discovered and explicitly stated more than thirty years later by Sakai~\cite{Sakai}, see also
\cite[Proposition~5.1]{DTW}. Sakai states the delooping for $n\geq 5$, $n-m\geq 3$. However, a careful review of the literature necessary for Lashof's work  \cite{Lashof}, allowed us to conclude that the delooping holds with no restriction on $m$ and $n$ (except the partial restriction $n\neq 4$). Moreover, we were able to see that the delooping is compatible with Budney's $E_{m+1}$-action \cite{Budney_cubes}. 
Let $\VV_{n,m}=\OO_n/\OO_{n-m}$ be the standard Stiefel manifold, and denote by $\VV^t_{n,m}=\TOP_n/\TOP_{n,m}$ and 
$\VV^{pl}_{n,m} = \PL_n/\PL_{n,m}$, respectively, its topological and PL versions. More details can be found in Subsection~\ref{ss:groups}.

\begin{mainthm}\label{th:b}
For any $1\leq m\leq n\neq 4$, $n-m\neq 2$, one has the following equivalences of $E_m$- and $E_{m+1}$-algebras:
\begin{gather}
\Embmn\simeq_{E_{m}} \Omega^m\hofiber (\VV_{n,m} \to \VV^t_{n,m}) \simeq_{E_{m}} \Omega^m \hofiber(\VV_{n,m} \to \VV^{pl}_{n,m});\label{eq_th_b_1}\\
\Embfrmn\simeq_{E_{m+1}} \Omega^{m+1}(  \VV^{t}_{n,m}/\!\!/ \OO_n  )\simeq_{E_{m+1}}\Omega^{m+1}( \VV^{pl}_{n,m} /\!\!/\OO_n);\label{eq_th_b_2}\\
\Ebarmn\simeq_{E_{m+1}} \Omega^{m+1} \VV^t_{n,m}\simeq_{E_{m+1}} \Omega^{m+1}\VV^{pl}_{n,m}.\label{eq_th_b_3}
\end{gather}
Moreover, for $n=m+2\neq 4$, $m\geq 1$, equivalences \eqref{eq_th_b_1}-\eqref{eq_th_b_2}-\eqref{eq_th_b_3} hold if the left-hand sides are replaced by their subspaces  
$\Emb_\partial(D^m,D^{m+2})^{\times}$, $\Emb_\partial^{fr}(D^m,D^{m+2})^{\times}$, $\Ebar_\partial(D^m,D^{m+2})^{\times}$, which are the unions of path-components corresponding to the invertible elements in $\pi_0$.
\end{mainthm}

One also has a PL version of this theorem for $n=4$, see Subsection~\ref{ss:d4}, where we only state equivalences of spaces (not even of $E_1$-algebras). Note that 
Theorem~\ref{th:a}(a-b) is covered in  Theorem~\ref{th:b}. 
Note also that the set of connected components $\pi_0$ for all these embedding spaces (with the usual concatenation of knots operation) is a group when $n-m\neq 2$. The case $n-m\geq 3$ is due to Haefliger
\cite{Haefliger}. For the case $n-m=1$ one has a stronger statement -- that any smooth knot $S^{n-1}\hookrightarrow S^n$,  $n\neq 4$, is isotopic to a reparameterization of the trivial one. This follows from the generalized topological Schoenflies theorem~\cite{Brown,Mazur,Morse} (that a locally flat $(n-1)$-sphere in $\R^n$ bounds a disc, true in all dimensions) and the uniqueness of the smooth structure on ~$D^n$,~$n\neq 4$.\footnote{For $n=5$, one has to use additionally the fact that the boundary of the disc is the standard~$S^4$.} 

 The space $\Embfrmn$ (or rather the equivalent one $\Embfr(S^m,S^n)/\OO_{n+1}$) has a natural $\OO_{m+1}$-action by precomposition. This action is often called
 {\it Hatcher's action} as Hatcher was the first who considered it in the classical case $m=1$, $n=3$ \cite{Hatcher_knots}. The action preserves the basepoint -- the trivial (equatorial) framed knot $S^m\subset S^n$. This naturally leads to the question: Does $\Embfrmn$ have a model, which would combine Budney's
 $E_{m+1}$-action and Hatcher's $\OO_{m+1}$-action into a framed little discs operad $E_{m+1}^{\OO_{m+1}}$-action? 
 On the level of the homology, for $m=1$, this action was studied in~\cite{Sakai_BV}. 
 At least in the case $n-m\neq 2$, $n\neq 4$,
 we are able to answer this question.
 
 \begin{mainthm}\label{th:c}
 For $1\leq m\leq n\neq 4$, $n-m\neq 2$,
 \[
 \Embfrmn\simeq_{\OO_{m+1}}\Omega^{m+1}\left(\tEmb(S^m,S^n)/\!\!/\OO_{n+1}\right).
 \eqno(\numb)\label{eq_th_c}
 \]
 Moreover, in case $n=m+2\neq 4$, the equivalence \eqref{eq_th_c} holds if the left-hand side is replaced by $\Emb^{fr}_\partial(D^m,D^{m+2})^{\times}$. 
 The equivalence respects the $E_{m+1}$-action where on the left-hand side we consider Budney's action \cite{Budney_cubes}.  
 \end{mainthm}

The space $\tEmb(S^m,S^n)$ above is the space of topological locally flat embeddings $S^m\hookrightarrow S^n$ pointed in the 
equatorial inclusion $S^m\subset S^n$, see Subsection~\ref{ss:spaces} for a precise
definition. The $\OO_{m+1}$-action on the right-hand side of~\eqref{eq_th_c} is defined in Subsection~\ref{ss:O_action_loops}. 
Note that for $2\neq n-m\geq 1$, the 
space $\tEmb(S^m,S^n)$ is connected. The case $n-m\geq 3$, $n\geq 5$ is due to Stallings~\cite{Stallings} (with a simpler proof by Brakes~\cite{Brakes} covering the remaining situation $(n,m)=(4,1)$, see also~\cite{Gluck}), while the case $n-m=1$ is the famous generalized Schoenflies 
problem solved independently by Brown and Mazur  \cite{Brown,Mazur}, see also~\cite{Morse}. 

The case $n=m$ of this version of the delooping implicitly appeared in Cerf's work~\cite{Cerf_diff}. In the latter case
the space $\tEmb(S^n,S^n)$ is the group $\Homeo(S^n)$ of homeomorphisms of $S^n$ and it has two components since $\pi_0\Homeo(S^n)=\pi_0\Homeo(\R^n)$. We therefore obtain the 
following.

\begin{corollary}\label{cor:dif_fr_discs}
The group $\Diffn$ admits a  framed discs operad
$E_{n+1}^{\OO_{n+1}}$-action  by means of the equivalence
$\Diffn\simeq \Omega^{n+1}\left(\Homeo(S^n)/\OO_{n+1}\right),\,n\neq 4.$
\end{corollary}

 \begin{remark}\label{r:deloop}
For $n-m\neq 2$, $\tEmb(S^m,S^n)\simeq \Homeo(S^n)/\Homeo(S^n \mmod S^m)$, while we have the same for the component
of the trivial knot $\tEmb(S^m,S^{m+2})_*
\simeq \Homeo(S^{m+2})/\Homeo(S^{m+2}\mmod S^m)$, see Proposition~\ref{p:spher_emb}. On the other hand, $\Homeo(S^n\mmod *)=\Homeo(\R^n)\simeq\TOP_n$ and $\Homeo(S^n\mmod S^m)=\Homeo(\R^n\mmod \R^m)\simeq\TOP_{n,m}$, see Subsection~\ref{ss:groups}. Therefore,\footnote{For the
$\OO_m$-action on these spaces, see Section~\ref{s:proof_c}.}
\[
\OO_n\bbslash\TOP_n/\TOP_{n,m}\,\, \simeq_{\OO_m} \,\, \OO_{n+1}\bbslash\Homeo(S^n)/\Homeo(S^n\mmod S^m)
\]
and the $\TOP$-version of the delooping~\eqref{eq_th_b_2} $\OO_m$-equivariantly agrees with \eqref{eq_th_c}.
\end{remark}

\subsection{Organization of the paper and other results}\label{ss:org}

Section~\ref{s:examples} explains different flavors of Theorem~\ref{th:b} for the cases of low and high codimension and ambient dimension.  In particular, we state  Theorem~\ref{th:d4}, 
which is a PL version of Theorem~\ref{th:b} for the ambient dimension $n=4$. Its Corollary~\ref{cor:d2_d4} states that only the trivial smooth knot $S^2\hookrightarrow S^4$ is PL trivial.

Sections~\ref{s:conv}, \ref{s:operads}, \ref{s:loops}, and~\ref{s:imm} cover the necessary background and have no new results. Most of the time we work in the category of simplicial sets. Section~\ref{s:conv}
describes all the mapping spaces that we use as well as simplicial groups $\TOP_n$, $\PL_n$, etc. It also covers the isotopy extension theorems and the Alexander trick. 
Section~\ref{s:operads} covers the operads that we use and their action on the spaces that have already been introduced at this point. Section~\ref{s:loops} makes it explicit the construction of loop spaces,
homotopy fibers, and pullbacks as well as the operad action on them. Section~\ref{s:imm} recalls the Smale-Hirsch type immersion theory in all the three categories: smooth, topological, and PL.

In Section~\ref{s:squares}, Theorems~\ref{th:square_diff} and~\ref{th:square_emb} restate well-known smoothing theory results relating diffeomorphism, homeomorphism, embedding and (formal) immersion spaces in terms of homotopy cartesian squares, which are used  in Sections~\ref{s:fr_homeo} and~\ref{s:fr_embed}. 

Section~\ref{s:fr_homeo} introduces and studies the group $\Homeo_\partial^{\underline{O_n}}(N)$ of $\underline{O_n}$-framed homeomorphisms of a smooth $n$-manifold~$N$ relative to the boundary.
Its points are triples~$(f,F,H)$, where~$f$ is a homeomorphism of~$N$, $F\colon TN\to TN$ is an isomorphism of the tangent vector bundle lifting~$f$, $H\colon TN\times [0,1]\to TN$ is a path in isomorphisms
of the topological tangent microbundle over~$f$  between~$F$ and the one induced by~$f$.  Theorem~\ref{th:diff_compos}    states that the natural inclusion of groups $\Diff_\partial(N)
\to \Homeo_\partial^{\underline{O_n}}(N)$ is an equivalence provided $n\neq 4$.   
 The result is interesting in itself as it provides a direct relation between the 
diffeomorphism and homeomorphism groups of a manifold. We state and prove an analogous result in the PL category though with a more limited application: $N$ must be a compact codimension zero
submanifold of~$\R^n$. 

In Section~\ref{s:proof_a} we prove Theorem~\ref{th:a}. We also show that its TOP part has a refinement, Theorem~\ref{th:a_top}, stating that the equivalences in Theorem~\ref{th:a}(a-b) are equivalences 
of  $E_{n+1}^{\OO_n}$-algebras. The crucial intermediate space in the zigzag of equivalences is the space~$\Homeo_\partial^{\underline{O_n}}(D^n)$.

Section~\ref{s:fr_embed} defines and studies the space $\tEmb_{\partial_0M}^{\underline{V_{n,m}}}(M,N)$ of $\underline{V_{n,m}}$-framed topological locally flat embeddings $M\hookrightarrow N$
relative to the boundary for a pair of smooth manifolds $M\subset N$. Its points are triples~$(f,F,H)$, where~$f\colon M\hookrightarrow N$ is a topological locally flat embedding,~$F\colon TM\to TN$ is a monomorphism of tangent vector bundles lifting~$f$,
and~$H\colon TM\times [0,1]\to TN$ is a path in monomorphisms of topological microbundles over~$f$ between $F$ and the one induced by~$f$.
Theorem~\ref{th:emb_fr_temb} states that the natural inclusion $\Emb_\partial(M,N)\to \tEmb_{\partial_0M}^{\underline{V_{n,m}}}(M,N)$ is an equivalence, provided~$n\neq 4$.
 We also study a PL version of this construction but only for the spaces that we need to use
in the proof of Theorem~\ref{th:b}.

In Section~\ref{s:proof_B_TOP} we state and prove an $\OO_m\times\OO_{n-m}$-equivariant refinement of the TOP part of Theorem~\ref{th:b}. One of the motivations for us to introduce this
refinement is to emphasize that the $\OO_{m+1}$-action appearing in the delooping of $\Emb_\partial^{fr}(D^m,D^{n})$ (see Theorem~\ref{th:c})   should not be confused with its natural $\OO_{n-m}$-action, compare with 
\cite[Corollary~1.2]{Sakai}.

In Section~\ref{s:proof_B_PL} 
  we prove the PL part of Theorem~\ref{th:b} and Theorem~\ref{th:d4}. 

In Section~\ref{s:proof_c} we state and prove Theorem~\ref{th:c_refined}, which is an $\OO_{m+1}\times\OO_{n-m}$-equivariant refinement of Theorem~\ref{th:c}.

\subsection{Acknowledgement}\label{ss:ackn}
The authors thank David Auckly, Ryan Budney, Thomas Goodwillie, Allen Hatcher, John Klein, Mark Powell, Frank Quinn, Daniel Ruberman, Peter Teichner, and Michael Weiss for communication. 
The authors are especially indebted to Alexander Kupers for his competent help on numerous references in  smoothing theory and four-dimensional topology.

\section{Examples and applications}\label{s:examples}

\subsection{Codimension $n-m\leq 1$}\label{ss:cod1}

One has that $\pi_0\Diffn=\Theta_{n+1}=\pi_{n+1}(\PL/\OO)=\pi_{n+1}(\PL_n/\OO_n)$, $n\neq 4$, is Milnor's group of $h$-cobordisms of homotopy $(n+1)$-spheres (in particular $\Theta_4=0$), see \cite{KervaireMilnor, BL_diff, KS_essays}. Any smooth codimension one knot $D^{n-1}\hookrightarrow D^n$, $n\neq 4$, is isotopic to a reparameterization of the trivial knot. Moreover,
$\pi_0\Emb_\partial(D^{n-1},D^n)=\Theta_n$, $n\neq 4$. For $n=2$ and~$3$, this follows from the fact that $\pi_0 \Diff_\partial(D^1)=\pi_0\Diff_\partial(D^2)=0$ \cite{Smale_diff}. For $n\geq 6$, this is stated and proved in~\cite[Theorem~5.5]{Budney_family} using~\cite{Cerf_diff}. To see that $\pi_0\Emb_\partial(D^4,D^5)=\Theta_5=0$, one has to use the fact that the set of pseudoisotopy classes of
 diffeomorphisms of $D^4$ relative to the boundary
is trivial by~\cite[Theorem~1]{Kreck} (applied to the case of $S^4$).    
 For codimension one, the delooping formulae can be simplified since $\TOP_{n,n-1}\simeq \OO_1 \simeq \PL_{n,n-1}$ by Lemma~\ref{l:top_n_n-1}:
\[
\Emb_\partial(D^{n-1},D^n)\simeq \Embfr_\partial(D^{n-1},D^n) \simeq \Omega^n\left(\TOP_n/\OO_n\right) \simeq \Omega^n\left(\PL_n/\OO_n\right),
\,\,\, n\neq 4.
\eqno(\numb)\label{eq:cod1a}
\]
\[
\Ebar_\partial(D^{n-1},D^n)\simeq \Omega^n \TOP_n \simeq \Omega^n\PL_n,
\,\,\, n\neq 4.
\eqno(\numb)\label{eq:cod1b}
\]
It means the delooping is compatible with the Cerf lemma \cite[Appendix, Section~5, Proposition~5]{Cerf_lemme}, \cite[Proposition~5.3]{Budney_family}:
\[
\Diffn\simeq\Omega\Emb_\partial(D^{n-1},D^n);
\]
\[
\Dbarn\simeq \Omega\Ebar_\partial(D^{n-1},D^n).
\]

\subsection{Codimension $n-m=2$}\label{ss:cod2}
Any smooth  knot $D^{n-2}\hookrightarrow D^n$, $n\neq 4$, is invertible (as an element in $\pi_0$) if and only if its complement is a homotopy 
circle, if and only if it is isotopic to a (re)parameterization of the trivial knot, if and only if it is trivial as a PL knot, if and only if it is trivial as a topological locally flat knot
\cite[Theorem 28.1(i)]{Papa}, \cite[Theorem~(3)]{Levine}, \cite[Corollary~3.1]{Wall_spheres},
 \cite[Theorem~16.1]{Wall_surgery}, \cite[Theorem~1.1]{Shane}. Moreover, $\pi_0\Emb_\partial(D^{n-2},D^n)^{\times}=\Theta_{n-1}$, $n\neq 4$ \cite[Proposition~5.11]{Budney_family}. 

The groups  $\TOP_n$, $\PL_n$, $\TOP_{n,m}$, $\PL_{n,m}$ are difficult to study. The Morlet delooping and its generalization can be useful to understand these groups. Here is one immediate corollary.

\begin{theorem}\label{th:top_cod2}
For $n\neq 4$,  $\TOP_{n,n-2}\simeq\PL_{n,n-2}$. Moreover, for $n\geq 4$, some of the homotopy groups of both $\TOP_{n,n-2}$ and $\PL_{n,n-2}$ are not finitely generated implying  $\TOP_{n,n-2}\not\simeq\OO_2 \not\simeq \PL_{n,n-2}$.
\end{theorem}

 For the case  of $\PL_{4,2}$ in the proof of the second statement
below, we use the delooping from Theorem~\ref{th:d4}. The case of $\TOP_{4,2}$  does not follow from the delooping. It is  obtained by Randal-Williams in his recent note~\cite{Randal_S2_S4}, that $\pi_3 \TOP_{4,2}$ is of an infinite rank (see  footnote~\ref{foot5}). In other words, it contains as a subgroup an infinite sum of $\Z$'s.

\begin{proof}
The first statement follows from Theorem~\ref{th:b}, the equivalence
\begin{equation}\label{eq:KS}
\TOP_n/\PL_n
\simeq\begin{cases} K(\Z/2\Z,3),& n\geq 5;\\
*,& n\leq 3;
\end{cases}
\end{equation}
 see \cite[Essay~V, \S5, Subsection~5.0]{KS_essays}, and the $(n-2)$-connectedness of $\TOP_{n,n-2}/\OO_2$, $n\neq 4$, and  $\PL_{n,n-2}/\OO_2$ proved by Kirby-Siebenmann \cite[Theorem~B]{KS_codim2}
 and Wall~\cite[Lemma~5]{Wall_codim2}, respectively. 
  (This is related to and is deduced from the fact that any locally flat 
$(n-2)$-submanifold in a topological $n$-manifold has a unique up to an ambient isotopy normal bundle \cite[Theorem~A]{KS_codim2}. The latter fact in dimension four is proved
by Freedman and Quinn \cite[Section~9.3]{FQ} implying by the argument of \cite[Proof of Theorem~B]{KS_codim2} that $\TOP_{4,2}/\OO_2$ is also 2-connected. The PL version of this theory is due
to Wall~\cite{Wall_codim2}.)

Regarding the second statement, on one side, for $n\geq 6$, $\TOP_n$ and $\PL_n$ have finitely generated homotopy groups~\cite{Kupers}.\footnote{The reference excludes the case $n=7$. However, in personal communication A.~Kupers confirmed that the result holds for $n=7$ as well. The proof will appear in~\cite{KraKup}.} On the other side, some of the homotopy groups of the component  $\Emb_\partial(D^{n-2},D^n)_*$  of the trivial knot
are not finitely generated \cite{BG2,BG1}, \cite[Example~7.3]{BKK}. Thus,  $\TOP_{n,n-2}$ and $\PL_{n,n-2}$ have non-finitely generated homotopy groups.

To be precise, it is known that $\pi_1\Emb_\partial(D^{n-2},D^n)_*$, $n\geq 6$, contains as a direct summand an infinite sum of $\Z/2\Z$'s, which is a consequence of 
\cite[Corollary~5.5 (p.273), Corollary 4.6 (p.270), Note on p.230]{HW}, \cite[Corollary~3.11]{BG1} and the fiber sequence $\Diffn\to \Diff_\partial(S^1\times D^{n-1})\to \Emb_\partial(D^{n-1},S^1\times D^{n-1})$. It was also recently shown by Budney and Gabai that $\pi_{n-3}\Emb_\partial(D^{n-2},D^n)_*$, $n\geq 4$, has  infinite rank~\cite[Section~3]{BG2} (see also~\cite{Yoshioka}). These two facts imply that  $\pi_{n-1}\TOP_{n,n-2}$ and $\pi_{n-1}\PL_{n,n-2}$ contain $\oplus_{i=1}^\infty\Z/2\Z$ as a direct summand for $n\geq 6$; and that 
$\pi_{2n-5}\TOP_{n,n-2}$ and $\pi_{2n-5}\PL_{n,n-2}$ are of an infinite rank for $n\geq 5$ and $n\geq 4$, respectively. 
We use here the fact that $\pi_6(\TOP_5/\OO_5)=\pi_6(\PL_5/\OO_5)=\pi_4(\PL_4/\OO_4)=0$.\footnote{We were pointed out that in a recent short note Randal-Williams discusses a related statement that $\pi_1\Emb(S^{n-2},S^n)$ and $\pi_n\tEmb(S^{n-2},S^n)$ are not finitely generated~\cite{Randal_S2_S4}. Note that the shift of degrees is exactly due to the delooping of Theorems~\ref{th:b} and~\ref{th:c}. He proves this statement for $n=4$ by a separate more subtle argument. As a consequence of the fiber sequence \cite[Equation~(*)]{Randal_S2_S4}, $\pi_4\VV_{4,2}^t$ is not finitely generated implying the same
for $\pi_3\TOP_{4,2}$ (one has that  
 $\pi_4\TOP_4$ is finitely generated by \cite[Theorem~8.7A]{FQ}).\label{foot5}}
\end{proof}

\begin{remark}\label{r:fg}
Away the codimension $n-m=2$ and the unknown ambient dimensions $n=4$ and~$5$, all homotopy groups of $\TOP_{n,m}$ and $\PL_{n,m}$ are finitely generated, which is proved by the same argument using the fact that the homotopy groups of $\Diff_\partial(D^n)$ and $\Emb_\partial(D^m,D^n)$ are finitely generated \cite{Kupers,KraKup}. (For $n-m>2$ it is easily proved by means of the Goodwillie-Weiss manifold functor calculus, see Subsection~\ref{ss:cod3}.)
\end{remark}

%
%

\subsection{Ambient dimension $n\leq 3$}\label{ss:d3}

For $1\leq m\leq n\leq 3$, the space $\Emb_\partial(D^m,D^n)^{\times}$ is contractible. It is obvious for $n=1$; due to Smale \cite{Smale_diff}
for $n=2$; and due to Hatcher \cite{Hatcher_diff} for $n=3$. As a consequence, $\TOP_n\simeq \OO_n\simeq\PL_n$, $n\leq 3$, 
and $\TOP_{n,m}\simeq\OO_{n-m}\simeq
\PL_{n,m}$, $0\leq m\leq n\leq 3$. Indeed, $\pi_i(\TOP_n/\OO_n)=\pi_i(\PL_n/\OO_n)=0$ for $0\leq i\leq n\leq 3$ by the smoothing theory (since any topological/PL manifold of dimension $\leq 3$ has a unique smooth structure \cite[Essay~5, \S5]{KS_essays}), while the higher homotopy groups are trivial due to contractibility
of $\Diffn$, $n\leq 3$. Finally, Hatcher's result  $\Emb_\partial(D^1,D^3)^{\times}\simeq *$ \cite[Appendix]{Hatcher_diff} together with our delooping implies
$\PL_{3,1}/\OO_2\simeq *\simeq \TOP_{3,1}/\OO_2$.

\subsection{Ambient dimension $n=4$}\label{ss:d4}

For any $1\leq m\leq n\neq 4$, a smooth knot $D^m\hookrightarrow D^n$ represents an invertible element in $\pi_0\Embmn$ if and only if it is trivial
as a PL knot, if and only if it is trivial as a topological locally flat one, if and only if its complement is a homotopy $S^{n-m-1}$. Moreover, for $n-m\neq 2$, $n\neq 4$, all smooth knots satisfy these conditions, see Subsection~\ref{ss:cod1} and~\cite{Stallings,Zeeman}. For $n=4$, $m=1$ and~$4$, the same is obviously true. However, for $n=4$, $m=2$ and $3$, these  conditions may be different. Denote by $\Emb_\partial(D^m,D^4)^\star$, 
$\Emb_\partial^{fr}(D^m,D^4)^\star$, $\Ebar_\partial(D^m,D^4)^\star$ the unions of path-components of $\Emb_\partial(D^m,D^4)$, 
$\Emb_\partial^{fr}(D^m,D^4)$, $\Ebar_\partial(D^m,D^4)$, respectively, which correspond to PL trivial knots $D^m\hookrightarrow D^4$ (the framing
and the path in $\Omega^mV_{4,m}$ being disregarded). Here is the PL version of Theorem~\ref{th:b} for $n=4$. (The $\TOP$ version of the smoothing theory is not available in dimension four.)

\begin{theorem}\label{th:d4}
For any $1\leq m\leq 4$,
\begin{gather}
\Emb_\partial(D^m,D^4)^\star\simeq \Omega^m\hofiber (\VV_{4,m}\to \VV^{pl}_{4,m}),
\label{eq_th_d4_1}\\
\Emb_\partial^{fr}(D^m,D^4)^\star\simeq\Omega^{m+1}(\VV^{pl}_{4,m}/\!\!/\OO_4),
\label{eq_th_d4_2}\\
\Ebar_\partial(D^m,D^4)^\star\simeq\Omega^{m+1}\VV^{pl}_{4,m}.
\label{eq_th_d4_3}
\end{gather}
\end{theorem}
We are able to prove only an equivalence of spaces, not even of $E_1$-algebras. However, at the level of $\pi_0$ the equivalences respect the natural concatenation  product. 
 The space $\Emb_\partial(D^1,D^4)$ has only one
path-component, and therefore, $\Emb_\partial(D^1,D^4)^\star=\Emb_\partial(D^1,D^4)$, 
$\Emb_\partial^{fr}(D^1,D^4)^\star=\Emb_\partial^{fr}(D^1,D^4)$, $\Ebar_\partial(D^1,D^4)^\star=\Ebar_\partial(D^1,D^4)$. One also has $\Emb_\partial(D^4,D^4)^\star=\Emb_\partial(D^4,D^4)
=\Diff_\partial(D^4)$, etc., since the space $\plHomeo_\partial(D^4)$ of PL homeomorphisms of $D^4$ is connected (in fact, contractible), see Lemma~\ref{l:alex}, and, therefore, any diffeomorphism of $D^4$ is PL trivial.
It is not known if $\Emb_\partial(D^3,D^4)$ is connected.
Non-connectedness would imply an existence of exotic smooth structure on $D^4$ and as a consequence on $S^4$ as well,
disproving the 4-dimensional Poincar\'e conjecture.  
 However, only the trivial smooth knot $D^3\hookrightarrow D^4$ is PL trivial. Indeed,
if we close a non-trivial knot $f$ to get the  spherical knot $\tilde f\colon S^3\hookrightarrow S^4$ corresponding to
 it, then $\tilde f$ bounds a smoothly exotic
but PL trivial $D^4$. However, any PL 4-manifold admits a unique compatible smooth structure. As a consequence  
$\Emb_\partial(D^3,D^4)^\star=\Emb_\partial(D^3,D^4)_*$ is the path-component of the trivial knot $D^3\subset D^4$. 
The spaces $\Emb_\partial^{fr}(D^3,D^4)^\star$
and  $\Ebar_\partial(D^3,D^4)^\star$ are preimages of this path-component. By contrast, any locally flat knot $D^3\hookrightarrow D^4$ is topologically trivial
by the generalized Schoenfiles theorem. It was also recently shown by Budney that $\pi_0 \Emb_\partial(D^3,D^4)$ is a group
\cite[Theorem~4.1]{Budney_stab}.

The problem of existence of PL trivial smooth knots $D^2\hookrightarrow D^4$ (or $S^2\hookrightarrow S^4$) seems to be overlooked by four-dimensional topologists.\footnote{Based on personal communication with David Auckly, Ryan Budney, Mark Powell, Frank Quinn, Daniel Ruberman, and Peter Teichner.} The following is an immediate corollary of Theorem~\ref{th:d4}. 

\begin{corollary}\label{cor:d2_d4}
 Among the isotopy classes of smooth knots $D^2\hookrightarrow D^4$ (or $S^2\hookrightarrow S^4$) only the trivial one is PL trivial.
\end{corollary}

\begin{proof}
In other words, the claim is that $\pi_0\Emb_\partial(D^2,D^4)^\star=0$. 
 Indeed, one has $\pi_i(\PL_4/\OO_4)=0$, $i\leq 4$,
by the smoothing theory \cite[Essay~II]{KS_essays}, \cite{HirschMazur}, \cite[Theorem~8.3C]{FQ}. 
 Thus,
\begin{multline*}
\pi_0\Emb_\partial(D^2,D^4)^\star=\pi_0\Omega^2\hofiber\left(\OO_4/\OO_2\to \PL_4/\PL_{4,2}\right)=
\\
\pi_3\left(\PL_4/\PL_{4,2},\OO_4/\OO_2\right)=\pi_3\left(\PL_4/\OO_4,\PL_{4,2}/\OO_2
\right)=\pi_2(\PL_{4,2}/\OO_2)=0. 
\end{multline*}
The last equality is due to \cite[Lemma 5]{Wall_codim2}. 
\end{proof}

%

We finally note that any invertible element in $\pi_0\Emb_\partial(D^2,D^4)$ corresponds to a topologically trivial knot.  Indeed, the former property implies that the complement $C_f$ of such knot $f$ must be a homotopy $S^1$.\footnote{Indeed, let $g$ be the inverse of $f$. Then the cyclic cover $\tilde C_{f*g}$ of 
$C_{f*g}$ is contractible and $\tilde C_{f*g}\simeq \tilde C_f\vee\tilde C_g$. Since a retract of a contractible space is contractible, so is $\tilde C_f$. We conclude $C_f=\tilde C_f/\Z\simeq S^1$.\label{foot6}}
On the other hand, Freedman (and Quinn) showed that the complement of a topological locally flat knot $S^2\hookrightarrow S^4$ (or $D^2\hookrightarrow D^4$) being a homotopy circle implies the knot is topologically trivial
\cite[Theorem~6]{Freedman}, \cite[Theorem~11.7A]{FQ}. 
However, it is not known if topologically trivial non-trivial smooth knots $D^2\hookrightarrow D^4$ exist. Neither is it known if such knots must always be invertible.

\subsection{Codimension $n-m\geq 3$. Operadic delooping}\label{ss:cod3}

For $n-m\geq 3$, $n\geq 5$, the natural map $\VV^{pl}_{n,m}=\PL_n/\PL_{n,m}\to \VV^{t}_{n,m}= \TOP_n/\TOP_{n,m}$ is an equivalence \cite[Proposition~$(t/pl)$]{Lashof},
implying  $\TOP_{n,m}/\PL_{n,m}\simeq K(\Z/2\Z,3)$ in this range. Moreover,
for $n-m\geq 3$, the map $\PL_n/\PL_{n,m}\to \GG_n/\GG_{n-m}$ is $(2n-m-3)$-connected \cite[Theorem~5.1]{Millett},\footnote{Here and below, $\GG_n$ denotes the monoid of self-homotopy equivalences of $S^{n-1}$.} while the map $\PL_n/\PL_{n,m}\to \GG/\GG_{n-m}$
is $(n-1)$-connected \cite[Lemma~6.3]{Gauniyal}.  These results can be used to compute the low degree homotopy groups (including $\pi_0$) of
the disc embedding spaces. 

Besides the aforementioned results of the smoothing theory, the codimension $\geq 3$ embedding spaces are well-understood due to the Goodwillie-Weiss
manifold calculus \cite{GW,Weiss,WeissErr}. One of the beautiful applications of this method is the following operadic-type delooping of disc
embedding spaces.

\begin{theorem}[\cite{BW_conf_cat,DT,DTW}]\label{th:operad}
For $n-m\geq 3$,
\begin{enumerate}[(a)]
\item $\Embmn\simeq\Omega^m\hofiber\left(\OO_n/\OO_{n-m}\to \Operad^h(E_m,E_n)\right)$,
\item $\Embfrmn\simeq\Omega^{m+1}\left(\Operad^h(E_m,E_n)/\!\!/\OO_n\right)$,
\item $\Ebarmn\simeq \Omega^{m+1}\Operad^h(E_m,E_n)$.
\end{enumerate}
\end{theorem}

In the above, $\Operad^h(E_m,E_n)$ denotes the derived mapping space of the little discs operads pointed at the inclusion $E_m\subset E_n$.

It is an interesting question whether these deloopings are compatible with the natural $E_m$- and $E_{m+1}$-actions. The methods of \cite{DT,DTW} prove only an equivalence of spaces, but not of $E_m$- or $E_{m+1}$-algebras. Even though it has not been stated explicitly in~\cite{BW_conf_cat},
Boavida-Weiss' approach is compatible with the $E_m$-action, but not with Budney's $E_{m+1}$-action. Moreover, Boavida-Weiss implicitly produce a map
$\VV^{t}_{n,m}= \TOP_n/\TOP_{n,m}\to\Operad^h(E_m,E_n)$ (that can also be restricted on $\VV^{pl}_{n,m}= \PL_n/\PL_{n,m}$), which, we believe, induces an equivalence of the 
right-hand sides of Theorem~\ref{th:b} (and  Theorem~\ref{th:d4}) with those of Theorem~\ref{th:operad} for $n-m\geq 3$. This would immediately imply
that the operadic delooping is compatible with the little discs action for $n\geq 5$. 

It is natural to conjecture that the two deloopings are actually
the same.

\begin{conjecture}\label{con:pl_operad}
The Boavida-Weiss map $\VV^{pl}_{n,m} \to \Operad^h(E_m,E_n)$ is an equivalence for $n-m\geq 3$.
\end{conjecture}

Boavida and Weiss recently posted a series of papers~\cite{BW1,BW2,BW3} in which they claim a proof of this conjecture
excluding the case $(n,m)=(4,1)$.
 Note that the smoothing theory delooping, Theorem~\ref{th:b}, and the operadic one, Theorem~\ref{th:operad},  immediately 
imply 
the statement of the conjecture in this remaining case\footnote{One uses here the assumption that the Boavida-Weiss map is compatible with delooping.} since both   $ \Operad^h(E_1,E_4)$ and $\PL_4/\PL_{4,1}$ are simply connected (in fact 2-connected). For the first space, see~\cite[Proposition~10.8]{FTW}, while for the second one, one has that the map $\PL_4/\PL_{4,1}\to \GG/\GG_{3}$ is 3-connected, while $\GG/\GG_3$ can easily be checked to be simply connected (in fact 2-connected).

%

\section{Conventions and prerequisites}\label{s:conv}

\subsection{Mapping spaces: diffeomorphisms, homeomorphisms, embeddings, immersions, etc}\label{ss:spaces}
We start by recalling several definitions from \cite{BLR_aut}.
As  usual, a smooth or piecewise linear (PL) structure on a manifold is defined as a choice of a maximal smooth or PL atlas, respectively. We will mostly consider smooth manifolds. 
If such a manifold $M$ is endowed with a  PL structure, the corresponding PL manifold is denoted by $|M|$. For a PL manifold $|M|$ and
a smooth one $N$, one says that a map $F\colon |M|\to N$ is {\it piecewise smooth} (PD), 
 if for any point $x\in M$ and any PL chart $(U,\phi)$ containing $x$,
there is a triangulation (by affine simplices) of a closed neighborhood of $\phi(x)\in \phi(U)$, such that $F\circ\phi^{-1}$ is smooth on every simplex of the triangulation. A piecewise smooth map $F\colon |M|\to N$ is called a {\it piecewise smooth homeomorphism} (or PD homeomorphism) if it is a homeomorphism,
such that for any PL chart $(U,\phi)$ on $|M|$, if $F\circ\phi^{-1}$ is smooth on some affine simplex $\Delta\subset\phi(U)$, then $F\circ\phi^{-1}|_\Delta$
is a diffeomorphism on its image. A PL structure on a smooth manifold $M$ is {\it compatible} with the smooth one, if the identity map $id\colon |M|\to M$
is a PD homeomorphism. 

All the mapping spaces that we consider are defined as certain simplicial sets. We first fix the notation and then give a precise definition of the
corresponding
simplicial sets.

We denote by $\Diff(M)$, $\Homeo(M)$, $\plHomeo(M)$, $\pdHomeo(|M|,M)$ respectively the corresponding spaces of diffeomorphisms, homeomorphisms, PL homeomorphisms, PD homeomorphisms of~$M$. In case  the PL structure on~$M$ is fixed, we will also write $\pdHomeo(M)$ instead of 
$\pdHomeo(|M|,M)$. Note that the first three are groups, while the last one is not, but instead is a $\Diff(M){-}\plHomeo(|M|)$-bimodule: $\Diff(M)$ acts on it by postcomposition, while $\plHomeo(|M|)$ -- by precomposition.  Moreover, one has natural inclusions $\Diff(M)\subset \pdHomeo(|M|,M)\supset \plHomeo(|M|)$,
where the second one is an equivalence~\cite{BLR_aut}.

For a pair of smooth manifolds $M^m$, $N^n$ (we often assume $M\subset N$) with chosen PL structures (such that $|M|\subset |N|$ is a locally flat PL embedding), we denote by $\dMap(M,N)$, $\Map(M,N)$, $\plMap(|M|,|N|)$, $\pdMap(|M|,N)$, $\Emb(M,N)$, $\tEmb(M,N)$, $\plEmb(|M|,|N|)$,
$\pdEmb(|M|,N)$, $\Imm(M,N)$, $\tImm(M,N)$, $\plImm(|M|,|N|)$, $\pdImm(|M|,N)$ respectively the spaces of smooth/continuous/PL/PD maps, smooth/topological locally
flat/PL locally flat/PD embeddings/immersions $M\to N$.  In case the source and the target are the same $M=N$, we will use the simplified notation $\Map(M):=\Map(M,M)$, $\Emb(M):=\Emb(M,M)$, etc.

A topological embedding $F\colon M\hookrightarrow N$ is said to be {\it locally flat} if every point $x\in M$ has 
a neighborhood $U\ni x$, such that  $\hat F|_{U\times 0}:=F|_U$ can be extended to a homeomorphism on its image $\hat F\colon U\times \R^{n-m}\hookrightarrow N$ (which is supposed to be PL if we
work in the PL category).  It is worth mentioning at this point that in codimension $n-m\geq 3$ any PL embedding is PL locally flat~\cite{Zeeman}.  While it is
obviously false in codimension two, it is still an open question in codimension one. An example of a PL non-locally flat (but topologically locally flat) embedding could be an (iterated) suspension 
of a non-trivial PL knot $S^3\hookrightarrow S^4$, that exists if and only if the smooth Schoenflies  conjecture is false in dimension four. (In dimension four, the PL and smooth Schoenflies problems are equivalent
since any PL 4-manifold admits a unique smooth structure.)  The PL local flatness is usually called {\it local unknottedness} in the literature. For the convenience of exposition we will be using the same term {\it locally flat} for both topological and PL embeddings assuming the property of PL local flatness when it is referred to  PL embeddings or immersions.
A map $M\to N$ is a {\it locally flat immersion} if it is  locally a locally flat embedding.
Following the literature, see for example \cite{BL_diff},
we will only consider the spaces $\pdEmb(|M|,N)$, $\pdImm(|M|,N)$ of PD embeddings/immersions when the dimensions agree $n=m$.\footnote{It is
not hard to define what a locally flat PD embedding/immersion should be in case $m<n$. However, it is not obvious why the space of PD locally flat
 embeddings/immersions is equivalent
to the space of PL locally flat embeddings/immersions. Since such objects have not been used in the smoothing theory literature, we  avoid them.}
 A PD embedding is a 
PD homeomorphism on its image. A PD immersion is a PD map, which is locally a PD homeomorphism on its image.

In case $M$ (and $N$) have boundary, it is convenient and standard in the smoothing theory~\cite{BL_diff} to treat them as if they are open, which is done by attaching to the boundary $\partial M$
a ``{\it germ collar}''. Denote by $M_\epsilon:=M\cup_{\partial M} (\partial M\times [0,\epsilon))$. We say that $f\colon M_{\epsilon_1} \to N_{\epsilon_2}$ is equivalent
to $g\colon M_{\delta_1} \to N_{\delta_2}$ if they coincide on some neighborhood of $M$ in $M_{\min(\epsilon_1,\delta_1)}$. In particular it means $f|_M=g|_M$ and 
$f(M)=g(M)\subset N$. By an embedding of manifolds with boundary we mean an equivalence class of such germ-collared embeddings. This convention is applied to the spaces of embeddings and immersions including the smooth case (and including the case of manifolds with corners), so that one always has a natural map from the space of smooth to that of topological 
embeddings/immersions.  The germ collar in the smooth case should be added   to respect the smooth structure, which requires special care near the corners.
 This makes it straightforward to define a locally flat embedding/immersion between manifolds with boundary. Also, this allows one to define the topological or PL tangent 
microbundle over a manifold with boundary (which for open manifolds is defined as any neighborhood of the diagonal $\Delta_M\subset M\times M$, see~\cite{Milnor}), and the 
map of tangent microbundles induced by an embedding or immersion.  Another advantage of this construction is that it allows us to consider a  PL structure (as well as PL and PD maps) on a manifold non-compatible with the boundary. 
 For example, one can consider $M$ to be the standard disc $D^m$ with the PL structure induced
by the inclusion $D^m\subset \R^m$ ,\footnote{The point is that in this case the boundary $\partial D^m$ is not a PL submanifold of $\R^m$.} for which we can still define the space $\plEmb(D^m,\R^n)$. In practice for a positive codimension $n-m>0$, we only consider proper embeddings/immersions fixed in 
a neighborhood of the boundary. Only when $n-m=0$ we will allow embeddings/immersions to be free on a part of its boundary $\partial_1M\subset M$. Because of Lemma~\ref{l:top_n_n-1} and the Smale-Hirsch type theory, see Section~\ref{s:imm}, this additional {\it germ collar} information about embeddings/immersions does not change the homotopy type 
of embedding/immersion space.

All these mapping spaces are defined as simplicial sets, whose $k$-simplices are  parameterized by  $\Delta^k$ families of
such maps. For example, the $k$-simplices of $\Diff(M)$ are diffeomorphisms
\[
M\times\Delta^k \to M\times\Delta^k,
\eqno(\numb)\label{eq:diff_simpl}
\]
commuting with the projection on $\Delta^k$. Similarly, the $k$-simplices of $\Homeo(M)$, $\plHomeo(|M|)$, $\pdHomeo(M)$ are respectively homeomorphisms,
PL homeomorphisms, PD homeomorphisms  \eqref{eq:diff_simpl}, commuting with the projection on $\Delta^k$.
The $k$-simplices of   $\dMap(M,N)$, $\Map(M,N)$, $\plMap(|M|,|N|)$, $\pdMap(|M|,N)$, $\Emb(M,N)$, $\tEmb(M,N)$, $\plEmb(|M|,|N|)$,
$\pdEmb(|M|,N)$, $\Imm(M,N)$, $\tImm(M,N)$, $\plImm(|M|,|N|)$, $\pdImm(|M|,N)$ are respectively
smooth/continuous/PL/PD maps, smooth/topological locally
flat/PL locally flat/PD embeddings/immersions, 
\[
M\times\Delta^k \xrightarrow{F_k} N\times\Delta^k,
\eqno(\numb)\label{eq:map_simpl}
\]
commuting with the projection on $\Delta^k$. In case of $\tEmb(M,N)$, $\plEmb(|M|,|N|)$,  $\tImm(M,N)$, $\plImm(|M|,|N|)$, the local flatness of~\eqref{eq:map_simpl}
is supposed to be of a {\it parameterized} nature described as follows. For any $x\in M$ and $t\in \Delta^k$ there must exist open neighborhoods $U\subset M$ of $x$ and $W\subset\Delta^k$ of $t$ and a homeomorphism
on its image $\hat F_k\colon U\times \R^{n-m}\times W\hookrightarrow N\times W$, commuting with the projection on $W$ and such that $\hat F_k|_{U\times 0\times W}=F_k|_{U\times W}$.
{\it Parameterized} PL {\it local flatness} is equivalent to  {\it parameterized local unknottedness} (as defined in~\cite[p.~652]{Hudson}) and  always holds in codimension $n-m\geq 3$~\cite[Corollary~3.2,
p.~653]{Hudson}.

Often in the smoothing theory literature, as for example in our main references~\cite{BL_diff, KS_essays}, these objects are referred to as  (complete) semi-simplicial complexes due to the
term {\it simplicial set} not being fully accepted yet at the time when the theory was developing.  However,   all  objects that we study here are honest simplicial sets. In particular, unlike semi-simplicial sets, they have well-defined degeneracies. 
Given
an order preserving map $\gamma\colon\{0,\ldots,\ell\}\to\{0,\ldots,k\}$ inducing $\gamma_*\colon \Delta^\ell\to\Delta^k$, one defines $\gamma^*F_k$ as the pullback of $F_k$ along 
$id_N\times \gamma_*$: 
\begin{equation}\label{eq:gamma_star}
\xymatrix{
M\times\Delta^\ell\ar[d]_{id_M\times\gamma_*}\ar[r]^{\gamma^*F_k}&N\times\Delta^\ell\ar[d]^{id_N\times\gamma_*}\\
M\times\Delta^k\ar[r]^{F_k}& N\times\Delta^k.
}
\end{equation}
It is easy to check that in the case  $\gamma$ is a degeneracy ($\ell=k+1$), $\gamma^*F_k$ is a $(k+1)$-simplex satisfying the same properties as $F_k$ (being a smooth embedding or a topological locally flat one, or an immersion, or a homeomorphism, etc.).

For $A,B\subset M$, we denote by $\Map(M,N\mmod A)$, $\Map_B(M,N)$, $\Map_B(M,N\mmod A)$ the spaces of continuous maps $M\to N$
(i.e., simplicial subsets of $\Map(M,N)$), which coincide with the inclusion $i\colon M\subset N$ on $A$ or/and near $B$, respectively. We use the same conventions for all the other types of maps that we consider: homeomorphisms, diffeomorphisms, embeddings, immersions, etc. For example, 
$\Homeo_B(M)$ denotes the simplicial set, whose $k$-simplices are homeomorphisms \eqref{eq:diff_simpl}, which are identity near $B\times\Delta^k$
and commute with the projection on $\Delta^k$. In case $B=\partial M$, for shortness we write $\Homeo_\partial (M)$, $\Emb_\partial(M,N)$, etc. 
When we consider mapping spaces with additional structure, for example framed embeddings $\Emb^{fr}_B(M,N\mmod A)$, in our conventions the structure must be standard near $B$ with no restriction on the structure on or near $A$.

By $\Map(M,N)_*$, $\Emb_\partial(M,N)_*$, $\Emb_\partial(M)_*$, etc., we mean the path-component of the inclusion $M\subset N$ and the identity map $id_M\colon M\to M$, respectively.
By $\Map(M)^\times$, $\Emb_\partial(M)^\times$, $\Emb_\partial(D^{n-2},D^n)^\times$, etc., we understand the union of those path components that correspond to invertible elements in~$\pi_0$ with respect to the composition or concatenation  product.

\subsubsection{} Some of these simplicial sets can be replaced by topological mapping spaces, though by far not all of them. This is the case of $\Diff(M)$, $\Homeo(M)$,
$\dMap(M,N)\simeq \Map(M,N)$, $\Emb(M,N)$, $\Imm(M,N)$; 
and $\tEmb(M,N)$ provided $n\geq 5$ and $n-m\geq 3$. 
For topological embeddings it is actually a consequence of a highly non-trivial theorem, which states that for $n\geq 5$ and $n-m\geq 3$ a family of
 embeddings \eqref{eq:map_simpl}
is locally flat if and only if for every point $x\in\Delta^k$ the corresponding embedding $M\hookrightarrow N$ is locally flat~\cite{Chernav}. This implies that the simplicial set 
$\tEmb(M,N)$ defined above coincides with the singular chains simplicial set of the subspace $E(M,N)$ (that consists of locally flat embeddings) of the space $C(M,N)$ of continuous maps $M\to N$ endowed with the usual compact-open topology~\cite{Lashof}. In the other cases, in particular, when $n-m=2$, a simplicial set
is the only tool 
 to describe the space of topological locally flat embeddings with the desired properties.\footnote{A natural condition for such space would be that 
its $k$-th homotopy group is the set of equivalence classes of locally flat families parametrized by $S^k$ embeddings, where the equivalence is induced by  locally flat families of embeddings parametrized by $S^k\times [0,1]$.} In the case of PL and PD embeddings and immersions,  simplicial sets are also the only way 
to deal with such spaces.

\subsection{Alexander trick and spaces of topological and PL homeomorphisms of a disc}\label{ss:alex}
Throughout the paper by  $D^n$ we mean the standard unit $n$-disc  $D^n:=\{x\in\R^n\,|\, |x|\leq 1\}$. %
By its PL structure, we mean the structure induced from $\R^n$. 
For simplicity, we write $D^n$ for $|D^n|$.  
 For $0\leq m\leq n$, one has a natural inclusion $D^m\subset D^n$ induced by the inclusion $\R^m=\R^m\times 0^{n-m}\subset \R^n$.

The following well-known fact is often attributed to J.~D.~Alexander~\cite{Alex}.

\begin{lemma}[Alexander trick]\label{l:alex}
For any $0\leq m\leq n$, the groups $\Homeo_\partial(D^n)$, $\plHomeo_\partial(D^n)$, 
 $\Homeo_\partial(D^n\mmod D^m)$, 
$\plHomeo_\partial(D^n\mmod D^m)$ 
 are contractible.
\end{lemma} 

\begin{proof}[Idea of the proof]
Given a homeomorphism $f\colon D^n\xrightarrow{\cong}D^n$, we extend it as identity on $\R^n\setminus
D^n$. Then the homotopy $H\colon D^n\times [0,1]\to D^n$ between $H|_{t=0}=id$ and $H|_{t=1}=f$ is defined as follows:
\begin{equation}\label{eq:alex_h}
H(x,t)=
\begin{cases}
tf\left(\frac xt\right),& t\in (0,1];\\
x,& t=0.
\end{cases}
\end{equation}
In case $f$ is PL, so is the homotopy $H$.
Note that if the support of $f$ is a polyhedron $K\subset \overset{\circ}{D}{}^n$ in the open unit disc, then the support
of $H$  is the cone on $K$.  
This homotopy can then be used to construct an explicit simplicial set homotopy
\[
H\colon \mathrm{(pl)Homeo}_\partial(D^n (\mmod D^m))\times\Delta^1\to  \mathrm{(pl)Homeo}_\partial(D^n (\mmod D^m)).
\]
\end{proof}

\begin{remark}\label{r:alex}
The same argument proves that the space of compactly supported homeomorphisms of $\R^n$ (preserving $\R^m$ pointwise) is contractible.
\end{remark}

\begin{remark}\label{s:alex_pd}
The Alexander trick does not work  for the spaces $\pdHomeo_\partial(D^n (\mmod D^m))$ of PD homeomorphisms of a disc. The problem is that the 
formula~\eqref{eq:alex_h} of~$H$, for any non-PL PD map~$f$,
produces a non-PD map near $(x,t)=(0,0)$.\footnote{For example, if $n=1$ and $f(x)=x+x^2$ on $[a,b]\subset D^1$, then $H(x,t)=x+\frac{x^2}t$ is continuous, but not (piecewise) smooth on the triangle with vertices $(a,1)$, $(b,1)$ and $(0,0)$. In general, this problem appears whenever $f$ has a non-linear part in its local Tailor expansion.} 
To prove contractibility of $\pdHomeo_\partial(D^n (\mmod D^m))$, one needs to change PD homeomorphisms by  small PD isotopies to  PL ones
(as it is done in~\cite{Whitehead}, see also~\cite{BLR_aut}) and only then one  applies the 
Alexander trick.
\end{remark}

\subsection{Isotopy extension theorems and the spaces of (topological/PL) spherical and disc embeddings}\label{ss:isot_ext}
The isotopy extension theorem holds in all the three categories: smooth~ \cite{Lima,Palais}, topological  \cite[Corollary~1.4]{EdwKirby}, \cite[Theorem~6.17]{Sieb_deform}, \cite{Wright}, and 
PL~\cite[Theorem~3, Corollary~3.2]{Hudson}. Given a $k$-simplex
\[
f\colon M\times\Delta^k\to N\times\Delta^k,\,\,\,\,
f(x,t)=(f_1(x,t),t)),
\]
 of $\Emb_\partial(M,N)$,  $\tEmb_\partial(M,N)$, or  $\plEmb_\partial(|M|,|N|)$, and any $t_0\in\Delta^k$, where $M$ is compact, the theorem
says that there always exists a $k$-simplex $g: N\times\Delta^k\to N\times \Delta^k$ of $\Diff_\partial(N)$, $\Homeo_\partial(N)$, 
$\plHomeo_\partial(|N|)$, respectively, such that $f$ coincides with the composition
\[
M\times\Delta^k\to N\times\Delta^k\xrightarrow{g} N\times\Delta^k,
\]
where the first map is $(x,t)\mapsto(f_1(x,t_0),t)$.

The following is an immediate corollary of the isotopy extension theorem. 

\begin{proposition}\label{p:disc_emb}
For any $n\geq m\geq 1$, one has isomorphisms of simplicial sets:
\begin{enumerate}[(a)]
\item $\Embmn_*=\Diffn/\Diff_\partial(D^n\mmod D^m)$, $n>m$;
\item $\tEmbmn=_{n-m\neq 2}\tEmbmn_*=\Homeo_\partial(D^n)/\Homeo_\partial(D^n\mmod D^m);$
\item $\plEmbmn=_{n-m\neq 2}\plEmbmn_*=\plHomeo_\partial(D^n)/\plHomeo_\partial(D^n\mmod D^m)$, in  the first identity additionally $(n,m)\neq (4,3)$.
\end{enumerate}
\end{proposition}

An important corollary of this proposition  and of Lemma~\ref{l:alex} is that the spaces $\tEmbmn_*$ and $\plEmbmn_*$ are contractible.

\begin{proof}
(a) follows from the smooth version of the isotopy extension theorem and the fact that any diffeomorphism of a disc $D^n$ is isotopic to a diffeomorphism
with support disjoint from~$D^m$,~$m<n$.

The first identity in (b) for $n-m=1$ follows from the generalized Schoenflies theorem, which holds in any ambient dimension \cite{Brown,Mazur,Morse}. For $n-m>2$, this follows
from Stalling's unknotting theorem \cite{Stallings,Brakes}\footnote{The theorem states that for $n-m>2$, any locally flat sphere $K^{m}\subset S^n$ is unknotted, i.e. the pair $(S^n,K^m)$ is homeomorphic to the standard pair  $(S^n,S^m)$. The same statement in the PL category was proved by Zeeman~\cite{Zeeman}.} and the fact that the groups $\Homeo_\partial(D^n)$ and 
$\Homeo_\partial(D^m)$ are connected.

The first identity in (c) similarly follows from the PL version of the generalized Schoenflies theorem (true in all dimensions except possibly $n=4$, which is a consequence of the fact that a topological $n$-disc
admits a unique PL structure\footnote{In case $n=5$, one must additionally assume that the boundary is the standard $S^4$.}) and the
PL unknotting theorem for codimension $>2$ spheres \cite{Zeeman}.

 The second identities in (b) and (c) follow from the topological and PL versions, respectively, of the isotopy extension theorem.
\end{proof}

One has a similar result for the spaces of spherical embeddings $S^m\hookrightarrow S^n$. We will need only its TOP version.

\begin{proposition}\label{p:spher_emb}
For any $0\leq m\leq n$, $n-m\neq 2$,
 \[
 \tEmb(S^m,S^n)=\Homeo(S^n)/\Homeo(S^n\mmod S^m).
 \]
  Moreover, for any $n\geq 2$, 
$\tEmb(S^{n-2},S^n)_*=\Homeo(S^n)/\Homeo(S^n \mmod S^{n-2}).$
\end{proposition}

\subsection{Simplicial groups/sets $\OO_n$, $O_n$, $O_{n,m}$, $\TOP_n$, $\TOP_{n,m}$, $\PL_n$, $\PL_{n,m}$, $\PD_n$, $\PD_{n,m}$, $\VV_{n,m}$, $V_{n,m}$, $\VV_{n,m}^t$,
$\VV_{n,m}^{pl}$}\label{ss:groups}

${}$

$\bullet$ Abusing notation we denote by $\OO_n$ both the topological orthogonal group and the simplicial groups of its either continuous or smooth singular chains. Often in the smoothing theory little attention is given to the difference between the orthogonal and linear groups. We partially follow this tradition and denote by $O_n$ the general linear group and the simplicial groups of its continuous and smooth chains. This is also done to avoid writing obvious maps in zigzags of equivalences, which correspond to replacement of $O_n$ by $\OO_n$.
 Often it will be clear from the context which version of $\OO_n$ or $O_n$
is used. To emphasize which type of chains are used (continuous or smooth), we write $tO_n$ and $dO_n$ for the continuous and smooth chains, respectively. 
We will also need to use the piecewise smooth version of singular chains on $O_n$ that we denote by $pdO_n$. Since the sum and product of piecewise
smooth maps into $\R$ are piecewise smooth, $pdO_n$ is also a simplicial group, for which the inclusions 
\[
dO_n\xrightarrow{\,\simeq\,}pdO_n\xrightarrow{\,\simeq\,}tO_n
\]
 are equivalences by the standard approximation argument. Finally,  define $O_{n,m}\subset O_n$ as the subgroup of linear isomorphisms of $\R^n$ preserving $\R^m$ pointwise. Obviously, $O_{n,m}\simeq  \OO_{n-m}$.

$\bullet$ The groups $\TOP_n$ and $\PL_n$ are defined as the groups of germs of homeomorphisms of $\R^n$ near zero, which are PL in the case of $\PL_n$. 
Their $k$-simplices are equivalence classes  of homeomorphisms on its image
\[
f_k\colon U\times\Delta^k\to \R^n\times\Delta^k,
\eqno(\numb)\label{eq:TOP_germ}
\]
commuting with the projection on $\Delta^k$ and identity on $0\times\Delta^k$. Here, $U\subset \R^n$ is an open set containing~0. One says $f_k\sim f'_k$ (where $f'_k\colon U'\times \Delta^k
\to\R^n\times\Delta^k$) if there exists an open $V\subset U\cap U'$, such that $0\in V$ and $f_k|_{V\times\Delta^k}=f'_k|_{V\times\Delta_k}$.
In case of $\PL_n$ the local homeomorphisms \eqref{eq:TOP_germ} are required to be PL. We call these equivalence classes {\it germs} of homeomorphisms
of $\R^n\times\Delta^k$ near $0\times\Delta^k$.

The group $\TOP_n$ can also be defined as 
\[
\TOP_n=\Homeo(\R^n\mmod 0)/\Homeo_0(\R^n)
\]
the quotient of the simplicial group of homeomorphisms of $\R^n$ that preserve the origin by its normal subgroup of homeomorphisms that are identity near the origin.
Here one uses the fact that any germ of a homeomorphism on its image~\eqref{eq:TOP_germ} can be represented by an actual homeomorphism $\R^n\times\Delta^k\to \R^n\times\Delta^k$ \cite[Lemmas 3 and 3'(PL)]{KuipLash1}, see also   \cite[Theorem~1]{Kister}.   The latter subgroup $\Homeo_0(\R^n)$ is contractible being isomorphic (by means of the inversion of $\R^n$) to the contractible group of
homeomorphisms of $\R^n$ with compact support and preserving~0, see Remark~\ref{r:alex}. In particular, the groups $\TOP_n$ and $\Homeo(\R^n)$
are equivalent by means of the zigzag
\[
\TOP_n\xleftarrow{\,\simeq\,}\Homeo(\R^n\mmod 0)\xrightarrow{\,\simeq\,}\Homeo(\R^n).
\eqno(\numb)\label{eq:top_n}
\]
A similar thing can be done for $\PL_n$ \cite[Lemma 1.6(f)]{KuipLash2}, though it is slightly harder to show that $\plHomeo_0(\R^n)$ is contractible. 
By means of a PL self-homeomorphism of $\R^n\setminus\{0\}$ which is a PL approximation of the inversion, the simplicial group 
$\plHomeo_0(\R^n)$ is isomorphic to the group of self-homeomorphisms of $\R^n$ with compact support, which preserve~0 and are PL everywhere except possibly
the origin. But the latter group is also contractible by the same Alexander trick.

$\bullet$ One defines $\TOP_{n,m}\subset\TOP_n$ and $\PL_{n,m}\subset\PL_n$ as subgroups of germs of homeomorphisms of $\R^n$ near~$0$ that preserve 
$\R^m\subset\R^n$ pointwise. One similarly has
\[
\TOP_{n,m}=\Homeo(\R^n\mmod \R^m)/\Homeo_0(\R^n\mmod\R^m),
\]
\[
\PL_{n,m}=\plHomeo(\R^n\mmod \R^m)/\plHomeo_0(\R^n\mmod\R^m),
\]
and, therefore, the natural maps
\[
\Homeo(\R^n\mmod \R^m)\xrightarrow{\,\simeq\,}\TOP_{n,m}
\eqno(\numb)\label{eq:top_nm}
\]
\[
\plHomeo(\R^n\mmod \R^m)\xrightarrow{\,\simeq\,}\PL_{n,m}
\]
are equivalences by Remark~\ref{r:alex}.

We will need to use the following well-known fact. Since it is always stated without proof, we prove it for completeness of exposition.

\begin{lemma}\label{l:top_n_n-1}
For any $n\geq 1$, $\TOP_{n,n-1}\simeq\OO_1\simeq\PL_{n,n-1}$.
\end{lemma}

\begin{proof}
By \eqref{eq:top_nm},
\begin{multline*}
\TOP_{n,n-1}\simeq \Homeo(\R^n\mmod\R^{n-1})=\Homeo(S^n\mmod S^{n-1})=\{\pm 1\}\times\Homeo_\partial(D^n)^{\times 2}\simeq\{\pm 1\},
\end{multline*}
where the last equivalence is a consequence of Lemma~\ref{l:alex}.

To see that $\PL_{n,n-1}$ is a union of two contractible components, we produce a deformation retraction of its subgroup $\mathrm{SPL}_{n,n-1}$ (of  orientation preserving  germs) to a point. Let $f\in \mathrm{SPL}_{n,n-1}$ be a germ of orientation preserving PL homeomorphisms
of $\R^n$ near~$0$ and let $\tilde f\in\plHomeo(\R^n\mmod \R^{n-1})$ be its representative. Define $\tilde H\colon \R^n\times [0,1]\to \R^n\times [0,1]$ as 
follows
\[
\tilde H(x_1,\ldots,x_n;t)=
\begin{cases}
(x_1,\ldots,x_n;t),& \text{if $|x_n|\leq t$,}\\
(\tilde f(x_1,\ldots,x_{n-1},x_n-t)+(0,\ldots,0,t),t),&\text{if $x_n\geq t$,}\\
(\tilde f(x_1,\ldots,x_{n-1},x_n+t)-(0,\ldots,0,t),t),&\text{if $x_n\leq t$.}
\end{cases}
\]
Even though $\tilde H$ depends on the choice of~$\tilde f$, 
 the  germ $H$ of homeomorphisms of $\R^n\times[0,1]$ near $0\times [0,1]$ corresponding to $\tilde H$ does not depend on this choice. One has $H|_{t=0}= f$, while $H|_{t=1}=id$. 
 We then use a parametrized version of this construction to define an explicit deformation retraction to a point $H\colon \mathrm{SPL}_{n,n-1}\times\Delta^1\to \mathrm{SPL}_{n,n-1}$.\footnote{We were pointed out
that a slight modification of this arguments recently appeared in~\cite[Lemma~8.17]{KraKupDisc}. Note that this second proof works in the TOP case as well.}
\end{proof}

$\bullet$ One defines similarly the simplicial sets $\PD_n$ and $\PD_{n,m}$ of germs of PD homeomorphisms of $\R^n$ near 0 (preserving $\R^m$ pointwise
in case of $\PD_{n,m}$).  One has $pdO_n\subset \PD_n\supset \PL_n$, $pdO_{n,m}
\subset \PD_{n,m} \supset\PL_{n,m}$, where the inclusions $\PL_n\subset\PD_n$ and $\PL_{n,m}\subset\PD_{n,m}$ are equivalences.

\begin{remark}\label{r:pd}
The simplicial sets $\PD_n$ and $\PD_{n,m}$ are not groups but $pdO_n$-$\PL_n$- and $pdO_{n,m}$-$\PL_{n,m}$-bimodules, respectively, where $pdO_n$ and $pdO_{n,m}$
act  by postcomposing, while $\PL_n$ and $\PL_{n,m}$ act by precomposing. In general, a composition of PD maps may not be PD. However, a composition $g_k\circ f_k$ of any $k$-simplices $f_k\colon U\times \Delta^k\to\R^n\times \Delta^k$
and $g_k\colon\R^n\times\Delta_k\to\R^n\times\Delta^k$ of $\PD_n$ and $pdO_n$, respectively, is of class PD. Indeed, if $f_k$ is smooth on a simplex $\Xi\subset U\times\Delta^k$ and $g_k$ is smooth on $\R^n\times\Upsilon$, where $\Upsilon$ is a simplex in $\Delta^k$, then $g_k\circ f_k$ is smooth
on $\Xi\cap(\R^n\times\Upsilon)$. This is because $p_2\circ f_k=p_2$ 
 implying $f_k(U\times\Upsilon)\subset \R^n\times\Upsilon$
and hence  $f_k\left(\Xi\cap(\R^n\times\Upsilon)\right)\subset  \R^n\times\Upsilon$.
\end{remark}

$\bullet$ One can show that $tO_n\cap\PD_n=pdO_n$ and $tO_{n,m}\cap\PD_{n,m}=pdO_{n,m}$. 
 By contrast the simplicial group  $tO_n\cap \PL_n$ has as $k$-simplices only constant linear maps $\Delta^k\to O_n$ and as a consequence is not equivalent
to $O_n$. One can define the simplicial set of PL chains on $O_n$, but the composition of such chains (induced by the product in $O_n$) is not PL in general. (Because the product of PL maps into 
 $\R$ may not be PL.) In particular, this means that one does not have any natural map $\OO_n\to \PL_n$ at least for the models of $\OO_n$ and $\PL_n$ that 
we consider above.\footnote{A remedy to it  is to define a map $B\OO_n\to B\PL_n$, see~\cite[Lemma~3.3]{LR} or~\cite[Equation~(6.1)]{HirschMazur}. Let $E\PL_n$ be a contractible space with a free
right $\PL_n$-action. Then the space $(\PD_n \times E\PL_n)/\PL_n$ is contractible (the action of $\PL_n$ on $\PD_n\times E\PL_n$ is diagonal) and has a free
left $\OO_n$-action. Thus, it can be considered as a model for $E\OO_n$, while $\left(\left(_{\OO_n}\!\!{\backslash}\PD_n\right)\times E\PL_n\right)/\PL_n$
is equivalent to $B\OO_n$ and has a natural projection onto $E\PL_n/\PL_n=B\PL_n$.\label{foot10}} The non-existence of a direct map $\OO_n\to\PL_n$
makes the PL version of the delooping more difficult to establish. In particular, this is the reason Theorem~\ref{th:c} and the $\OO_m\times\OO_{n-m}$-equivariant
version of the delooping (Theorems~\ref{th:a_top} and~\ref{th:b_top_refine})  are stated and proved only in the TOP version.

$\bullet$ Define $\VV_{n,m}:=\OO_n/\OO_{n-m}$, $\VV_{n,m}^{t}:=\TOP_n/\TOP_{n,m}$, $\VV_{n,m}^{pl}:=\PL_n/\PL_{n,m}$ the so called
{\it Stiefel manifold}, {\it topological Stiefel manifold}, and {\it PL Stiefel manifold}, respectively. We will also use $V_{n,m}:=O_n/O_{n,m}.$
\begin{itemize}
\item[-] $\VV_{n,m}$ is the space of isometric linear inclusions $\R^m\to\R^n$;
\item[-] $V_{n,m}$ is the space of injective linear maps $\R^m\to\R^n$;
\item[-] $\VV_{n,m}^t$ is the space of germs (near the origin) of topological locally flat embeddings $\R^m\to\R^n$ sending $0\mapsto 0$;
\item[-]  $\VV_{n,m}^{pl}$ is the space of germs (near the origin) of PL locally flat embeddings $\R^m\to\R^n$ sending $0\mapsto 0$.
\end{itemize}
As for $\OO_n$ and $O_n$, we denote by $\VV_{n,m}$ and $V_{n,m}$ both the topological spaces and the singular chains on them. Since 
$\OO_{n-m}\to \OO_n\to\VV_{n,m}$ and $O_{n,m}\to O_n\to V_{n,m}$ are fiber bundles, the simplicial sets of singular chains on $\VV_{n,m}$ and
$V_{n,m}$ are isomorphic to quotients of simplicial groups $t\VV_{n,m}=t\OO_n/t\OO_{n-m}$, $tV_{n,m}=tO_n/tO_{n,m}$. The same holds for the smooth 
singular chains.

 The $k$-simplices of $\TOP_n/\TOP_{n,m}$ and $\PL_n/\PL_{n,m}$
can be described as equivalence classes (germs)  of  locally flat embeddings $f\colon U\times \Delta^k\to\R^n\times\Delta^k$
(where $0\in U$ and $U\subset\R^m$ is open), which commute with the projection on $\Delta^k$ and send $0\in \R^m$ to $0\in \R^n$ (i.e., $(0,t)\mapsto (0,t)$ for all $t\in\Delta^k$). We call these equivalence classes {\it germs} of locally flat embeddings $\R^m\times\Delta^k$ into $\R^n\times\Delta^k$ near $0\times\Delta^k$. Note that we require that the local flatness is of parameterized nature, see Subsection~\ref{ss:spaces}. This is the reason why any $k$-simplex of germs of embeddings lifts to a $k$-simplex in $\TOP_n$ and $\PL_{n,m}$, respectively, and, therefore the quotient group description and the germs of embeddings one produce the same simplicial sets.  An important consequence of the latter description is that these quotient simplicial sets are Kan. This will be useful 
when we define  models for  iterated loop spaces on $\VV_{n,m}^t$ and $\VV_{n,m}^{pl}$, see Subsection~\ref{ss:loops}.

\section{Operads}\label{s:operads}

\subsection{Little discs operad $E_m$, overlapping discs operad $R_{m+1}$, Budney's $E_{m+1}$-action}\label{ss:little_discs}
We denote by $E_m$ the topological operad of little $m$-discs. Its $k$-th component $E_m(k)$ consists of configurations of $k$ discs with disjoint interiors
in the unit disc $D^m$. Each disc can be described by an affine linear map $L_i\colon\R^m\to\R^m$, $i=1\ldots k$, which is a composition of a translation and a positive rescaling and such that $L_i(D^m)$ is the $i$-th disc in the configuration. Abusing notation, we also denote by $E_m$ the simplicial operad of its singular chains
as well as its smooth variant. If we need to emphasize which of the two is used we write $tE_m$ and $dE_m$, respectively.

The smooth one $dE_m$ acts on $\Diff_\partial(D^m)$ and $\Emb_\partial(D^m,D^n)$, while the bigger one $tE_m$ acts on $\Homeo_\partial(D^m)$ 
and $\tEmb_\partial(D^m,D^n)$. The idea of this $E_m$-action is as follows. Given $f_1,\ldots,f_k\in\Diff_\partial(D^m)$ (or $\Homeo_\partial(D^m)$,
$\Emb_\partial(D^m,D^n)$, etc.) and $g\in E_m(k)$, where $g=(L_1,\ldots,L_k)$, we define 
\[
g(f_1,\ldots,f_k)(x):=
\begin{cases}
L_i\circ f_i\circ L_i^{-1}(x),& x\in L_i(D^m),\,\, i=1\ldots k;\\
x,& \text{otherwise.}
\end{cases}
\]

Besides being $E_m$-algebras, $\Diff_\partial(D^m)$ and $\Homeo_\partial(D^m)$ are groups, and therefore, associative monoids. Moreover,
they are $E_m$-algebras in the category of associative monoids (or, equivalently, associative monoids in the category of $E_m$-algebras). The 
two compatible structures determine the operad, which is the Boardman-Vogt tensor product (see \cite{BoardVogt,Dunn}) $R_{m+1}:=E_m\otimes Assoc$ 
 of $E_m$ and the associative operad $Assoc$. This operad implicitly appeared in Ryan Budney's work~\cite{Budney_cubes}.\footnote{The letter ``$R$'' is in honor of Ryan Budney.}
We call it {\it operad of overlapping $m$-discs}. Explicitly, its points are configurations of $k$ possibly overlapping discs $\vec{L}=(L_1,\ldots,L_k)\in E_m(1)^{\times k}$ with additional information:
when two discs overlap (share a point in their interiors), we know which one is above and which is below (similarly to cards lying on a table). 
Let $\Sigma_k$ denote the $k$-th symmetric group. The space $R_{m+1}(k)$ is defined as the quotient space of $E_m(1)^{\times k}\times\Sigma_k$, where the equivalence
relation is generated by $(\vec{L},\sigma)\sim (\vec{L},\sigma')$, where $\sigma'$ is obtained from $\sigma$ by transposition of two neighboring
elements $\sigma_i$ and $\sigma_{i+1}$ under the condition that the discs $L_{\sigma_i}(D^m)$ and $L_{\sigma_{i+1}}(D^m)$ have disjoint interiors.
Here, when we write $\sigma=(\sigma_1,\ldots,\sigma_k)$, we mean $\sigma_i=\sigma^{-1}(i)$.

Given $g=\left[(\vec{L},\sigma)\right]\in R_{m+1}(k)$ and  $f_1,\ldots,f_k\in\Diff_\partial(D^m)$ (or $\Homeo_\partial(D^m)$), the result of the action
$g(f_1,\ldots,f_k)$ is 
\[
g(f_1,\ldots,f_k)=\hat L_{\sigma_1}(f_{\sigma_1})\circ \hat L_{\sigma_2}(f_{\sigma_2})\circ \ldots \circ\hat L_{\sigma_k}(f_{\sigma_k}),
\eqno(\numb)\label{eq:overl_action}
\]
where $\hat L_{j}(f_{j})\colon D^m\to D^m$ is defined as follows
\[
\hat L_j(f)(x)=
\begin{cases}
L_j\circ f\circ L_j^{-1}(x),& x\in L_j(D^m);\\
x,& \text{otherwise.}
\end{cases}
\]
In order to avoid heavy notation, we follow the same conventions as for $E_m$  to denote the simplicial set operad of singular (continous and smooth) chains
on $R_{m+1}$.

One has a natural map of operads $E_{m+1}\to R_{m+1}$  induced by the projection $\R^{m+1}\to\R^m$. A permutation $\sigma$ assigned to
a configuration of $k$ disjoint open $(m+1)$-disks must satisfy the following property: if the projection of the $i$-th disc overlaps with that of the $j$-th one, then 
$\sigma(i)<\sigma(j)$ if and only if the $(m+1)$-st coordinates of the points of the $i$-th disc are less than the $(m+1)$-st coordinates of the points
with the same projection lying in the $j$-th disc. In the latter case we say that the $i$-th disc is below the $j$-th one. By applying~\cite[Main Theorem]{Smale_vietoris} one can  show that this map
is an equivalence of operads.\footnote{One would need to use an intermediate operad $E_m\otimes E_1$  whose components are configuration spaces of cylinders $D^m\times D^1$ of arbitrary radius and height in the unit cylinder. The latter operad can more easily be seen to be equivalent to $E_{m+1}$, compare with~\cite{Dunn}.} That is the reason we denote this operad by $R_{m+1}$ rather than by $R_m$. 

As an important corollary, $\Diff_\partial(D^m)$ and $\Homeo_\partial(D^m)$ are $E_{m+1}$-algebras.

Following R.~Budney \cite{Budney_cubes}, the framed embedding space $\Embfrmn$ is equivalent to the space $\Emb_{\partial D^m\times D^{n-m}}(D^m\times D^{n-m})$. 
The latter space is also an $E_m$-algebra in associative monoids (the associative product being the composition), and therefore, is also an $R_{m+1}$-algebra
and an $E_{m+1}$-algebra by restriction. The $R_{m+1}$-action is given by the same formula~\eqref{eq:overl_action}, where 
$\hat L_j(f)\in \Emb_{\partial D^m\times D^{n-m}}(D^m\times D^{n-m})$ is obtained from $f\in \Emb_{\partial D^m\times D^{n-m}}(D^m\times D^{n-m})$ as follows
\[
\hat L_j(f)(x)=
\begin{cases}
(L_j\times id_{D^{n-m}})\circ f\circ (L_j^{-1}\times id_{D^{n-m}})(x),& x\in L_j(D^m)\times D^{n-m};\\
x,& \text{otherwise.}
\end{cases}
\]
To be  precise, the simplicial set $\Emb_{\partial D^m\times D^{n-m}}(D^m\times D^n)$ is a $dR_{m+1}$-algebra, while its friend
$\tEmb_{\partial D^m\times D^{n-m}}(D^m\times D^n)$ is a $tR_{m+1}$-algebra, that we can also view as  a $dR_{m+1}$-algebra (and $dE_{m+1}$-algebra)
by restriction. We refer to~\cite{Budney_cubes} for beautiful illustrations of this action.

\subsection{No natural action on $\plHomeo_\partial(D^n)$, $\plEmbmn$, $\pdHomeo_\partial(D^m)$.}
 The PL and PD cousins $\plHomeo_\partial(D^n)$, $\plEmbmn$, $\pdHomeo_\partial(D^m)$, etc.
do not seem to be endowed with a natural action of any $E_m$ operad. One of the issues is that PL singular chains on $E_m$ cannot be composed,
and therefore, do not form an operad. The problem is that the product of PL functions is usually not PL. For example, $E_m(1)$ is a submonoid of the group of affine maps $x\mapsto \lambda x+b$, $\lambda\in (0,1]$, $b\in D^m$. But note that if $\lambda_1(t)$ and $\lambda_2(t)$ are (piecewise) linear maps $\Delta^i\to (0,1]$, then their product $\lambda_1(t)\lambda_2(t)$ may not be so. There is, probably, a way to go around it at least for for the spaces of PL  homeomorphisms (and embeddings) of discs if we replace discs by PL ones, for example, by taking a cube $I^m$ instead of $D^m$. Then the operad 
of PL endomorphisms of $I^m$ in the category of PL $m$-manifolds with codimension zero embeddings as morphisms i.e., with components 
$\{\plEmb(\coprod_k I^m,I^m),\, k\geq 0\}$, acts on $\plHomeo_\partial(I^m)$. Nonetheless, the latter operad does not seem to contain a suboperad equivalent to $E_m$. Moreover, this operad does not act on $\pdHomeo_\partial(I^m)$. So, it is not useful. On a positive side, the PD chains on $E_m$ form an operad $pdE_m$ equivalent to $E_m$,\footnote{This is similar to the fact that $pdO_m$ is a group.} but it does not act on any of the spaces in question.

\subsection{Framed operads and their action}\label{ss:fr_operads}
Given a group or a monoid $G$ acting on an operad $P$ by endomorphisms, one can define the {\it $G$-framed} operad $P^G$, whose $k$-th
component is $P^G(k)=P(k)\times G^{\times k}$. The composition map $\circ_i\colon P^G(k)\times P^G(\ell)\to P^G(k+\ell-1)$ is defined as
follows
\[
(p_1;g_1\ldots g_k)\circ_i (p_2;g_1'\ldots g_\ell')=
(p_1\circ_i(g_i\cdot p_2);g_1,\ldots,g_{i-1},g_i g_1',\ldots,g_i g_\ell',g_{i+1},\ldots,g_k).
\]

The most common example is the framed little discs operad $E_m^{\OO_m}$, where $g\in\OO_m$ acts on $(L_1,\ldots,L_k)\in E_m(k)$ as follows
\[
g\cdot (L_1,\ldots,L_k)=(g\circ L_1\circ g^{-1}, \ldots,g\circ L_k\circ g^{-1}).
\]
This action obviously generalizes to the Ryan Budney overlapping discs operad $R_{m+1}$ producing its framed version $R_{m+1}^{\OO_m}$. 
We will also consider the operads $E_{m+1}^{\OO_m}$, $E_m^{\OO_m\times\OO_{n-m}}$, $R_{m+1}^{\OO_m\times \OO_{n-m}}$,
 $E_{m+1}^{\OO_m\times \OO_{n-m}}$. To obtain $E_{m+1}^{\OO_m}$, we just restrict the $\OO_{m+1}$-action on $E_{m+1}$ to the subgroup 
 $\OO_m$. Note that the map $E_{m+1}\to R_{m+1}$ extends to the framed version $E_{m+1}^{\OO_m}\to R_{m+1}^{\OO_m}$. The factor $\OO_{n-m}$ 
 is defined to act trivially on $E_m$, $R_{m+1}$, $E_{m+1}$. For example, an $E_m^{\OO_m\times\OO_{n-m}}$-algebra is the same as an 
 $E_m^{\OO_m}$-algebra in the category of $\OO_{n-m}$-spaces. As before we use the same notation for the topological operads and their (smooth or continuous) singular chains 
 operads in simplicial sets, and we write in front $d$ or $t$ to emphasize that the smooth or continuous version is considered. 
 
 The spaces $\Embmn$, $\Embfrmn$, $\tEmbmn$ are algebras over $E_m^{\OO_m\times\OO_{n-m}}$, while $\Emb_{\partial D^m\times D^{n-m}}(D^m\times D^{n-m})$ and $\tEmb_{\partial D^m\times D^{n-m}}(D^m\times D^{n-m})$ are $R_{m+1}^{\OO_m\times\OO_{n-m}}$-algebras and, therefore
$E_{m+1}^{\OO_m\times\OO_{n-m}}$-algebras by restriction. Indeed, $\OO_m$ acts on $\Embmn$ (and $\tEmbmn$) by conjugation: for $f\in\Embmn$, $g\in\OO_m$,
\[
\left(g\cdot f=g\circ f\circ g^{-1}\right)\colon D^m\xrightarrow{g^{-1}}D^m\xrightarrow{f} D^n\xrightarrow{(g\times id_{\R^{n-m}})|_{D^n}} D^n,
\]
while $\OO_{n-m}$ acts by postcomposition. The group $\OO_m$ acts on $\Embfrmn$ in the same way as on $\Embmn$, while $\OO_{n-m}$ acts 
by postcomposition and simultaneous rotation of the framing in the opposite direction. On $\Emb_{\partial D^m\times D^{n-m}}(D^m\times D^{n-m})$ and $\tEmb_{\partial D^m\times D^{n-m}}(D^m\times D^{n-m})$, the group $\OO_m\times\OO_{n-m}$ acts by conjugation. Note that in all the  considered cases
the $\OO_m\times\OO_{n-m}$-action preserves the basepoint -- the trivial knot $D^m\subset D^n$ and the identity map $id\colon D^m\times D^{n-m}\to
D^m\times D^{n-m}$, respectively. This property is part of the compatibility condition (applied to the arity zero element $*=E_m(0)=R_{m+1}(0)$) between the operad $E_m$- or $R_{m+1}$-action and the
group $\OO_m\times\OO_{n-m}$-action.

\section{Loop spaces and homotopy pullbacks}\label{s:loops}

\subsection{Loop spaces}\label{ss:loops}
For a pointed topological space $(Y,*)$, we define its iterated $m$-loop space $\Omega^m Y$ as the space of continuous maps $D^m\to Y$ sending a neighborhood of the
boundary to the basepoint. Obviously, $\Omega^m Y$ is an algebra over the topological $E_m$ operad, while its singular chains simplicial set $S_*\Omega^m Y$
is an algebra over the simplicial set  operad $E_m$.

Given a pointed simplicial set $X$, there are different combinatorial models whose realizations are equivalent to the iterated loop space $\Omega^m|X|$ of
its realization $|X|$, see for example~\cite{Kan,Berger}. However, such models are not $E_m$-algebras, but rather algebras over a different combinatorial model of $E_m$. We do not explore
this combinatorial path, but instead define wherever possible our iterated loop spaces ad hoc so that (in most of the 
cases) they are acted upon by~$E_m$. 

$\bullet$ For example, we define $\Omega^j\TOP_n$ as a simplicial group whose $k$-simplices are germs of homeomorphisms of $\R^n \times D^j\times \Delta^k$
near $0\times D^j\times\Delta^k$ commuting with the projection on $D^j\times\Delta^k$ and identity near $\R^n\times\partial D^j\times\Delta^k$. It is immediate that $\Omega^j\TOP_n$ is an $R_{j+1}$-algebra (to be precise, a $tR_{j+1}$-one). Moreover, one has a natural evaluation map
\[
S_*\Omega^j|\TOP_n|\xrightarrow{\,\,\simeq\,\,} \Omega^j\TOP_n,
\eqno(\numb)\label{eq:eval_top}
\]
which is an equivalence of $R_{j+1}$-algebras. The reason it is an equivalence, is that both $\TOP_n$ and $\Omega^j\TOP_n$ are Kan complexes. Therefore,
the $i$-th homotopy group $\pi_i|\TOP_n|$ is described as the set of equivalence classes (up to ``homotopy'') of germs of homeomorphisms of $\R^n\times\Delta^i$ near $0\times\Delta^i$,
which are the identity on $\R^n\times\partial\Delta^i$. Similarly, $\pi_i|\Omega^j\TOP_n|$ is described as the set of equivalence classes of
germs of homeomorphisms of $\R^n\times D^j\times\Delta^i$ near $0\times D^j\times\Delta^i$, which are the identity on $\R^m\times\partial(D^j\times \Delta^i)$.
The latter implies $\pi_i|\Omega^j\TOP_n|=\pi_{i+j}|\TOP_n|$. Thus, the left- and right-hand sides have exactly the same homotopy groups. 

$\bullet$ We define $\Omega^j\TOP_{n,m}$ and $\Omega^j\left(\TOP_n/\TOP_{n,m}=\VV_{n,m}^t\right)$ in exactly the same way. For example,
the $k$-simplices of $\Omega^j\VV_{n,m}^t$ are the germs of locally flat embeddings 
\[
\R^m\times D^j\times\Delta^k\hookrightarrow \R^n\times D^j\times\Delta^k
\eqno(\numb)\label{eq:germ_Dj_V}
\]
near $0\times D^j\times\Delta^k$, which commute with the projection onto $D^j\times\Delta^k$ and are the identity inclusion near $\R^m\times \partial D^j\times \Delta^k$. One gets similar evaluation maps
\[
S_*\Omega^j|\TOP_{n,m}|\xrightarrow{\,\,\simeq\,\,} \Omega^j\TOP_{n,m},\quad
S_*\Omega^j|\VV_{n,m}^t|\xrightarrow{\,\,\simeq\,\,} \Omega^j\VV_{n,m}^t,
\]
which are equivalences of $R_{j+1}$- and $E_j$-algebras, respectively.

$\bullet$ We define $\Omega^j\PL_n$, $\Omega^j\PL_{n,m}$, $\Omega^j\left(\PL_n/\PL_{n,m}=\VV_{n,m}^{pl}\right)$, $\Omega^j\PD_n$,
$\Omega^j\PD_{n,m}$ as simplicial sets whose $k$-simplices are germs of homeomorphisms of $\R^n\times D^j\times\Delta^k$ or germs of 
embeddings~\eqref{eq:germ_Dj_V}, all as above, which are additionally PL or PD, whatever applies. The first two are simplicial groups, but none is an $R_{j+1}$-
or $E_j$-algebra (no matter we choose continuous or smooth versions of these operads). However, these simplicial sets are equivalent to the
$j$-th loop spaces of their realizations. For example, for $\PL_n$, consider the simplicial subset $S_*^{pl}\Omega^j_{pl}|\PL_n|\subset S_*\Omega^j|\PL_n|$,
whose $k$-simplices are PL maps $\Delta^k\times D^j\to|\PL_n|$, which send a neighborhood of $\Delta^k\times\partial D^j$ to the basepoint. One gets a zigzag
of equivalences of simplicial groups
\[
S_*\Omega^j|\PL_n|\xleftarrow{\,\,\simeq\,\,}S_*^{pl}\Omega^j_{pl}|\PL_n|\xrightarrow{\,\,\simeq\,\,} \Omega^j\PL_n,
\]
where the first map (inclusion) is an equivalence by the approximation argument, while the second (evaluation) map is an equivalence by the same reason
why \eqref{eq:eval_top} is one. In exactly the same way one gets the zigzags
\[
S_*\Omega^j|\PL_{n,m}|\xleftarrow{\,\,\simeq\,\,}S_*^{pl}\Omega^j_{pl}|\PL_{n,m}|\xrightarrow{\,\,\simeq\,\,} \Omega^j\PL_{n,m},
\quad \quad
S_*\Omega^j|\VV_{n,m}^{pl}|\xleftarrow{\,\,\simeq\,\,}S_*^{pl}\Omega^j_{pl}|\VV_{n,m}^{pl}|\xrightarrow{\,\,\simeq\,\,} \Omega^j\VV_{n,m}^{pl},
\]
\[
S_*\Omega^j|\PD_n|\xleftarrow{\,\,\simeq\,\,}S_*^{pl}\Omega^j_{pl}|\PD_n|\xrightarrow{\,\,\simeq\,\,} \Omega^j\PD_n,
\quad\quad
S_*\Omega^j|\PD_{n,m}|\xleftarrow{\,\,\simeq\,\,}S_*^{pl}\Omega^j_{pl}|\PD_{n,m}|\xrightarrow{\,\,\simeq\,\,} \Omega^j\PD_{n,m}.
\]

$\bullet$ For $\OO_n$, $O_n$, $O_{n,m}$, $\VV_{n,m}=\OO_n/\OO_{n-m}$, $V_{n,m}=O_n/O_{n,m}$, the simplicial iterated loop spaces
 $\Omega^j\OO_n$, $\Omega^j O_n$, $\Omega^j O_{n,m}$, $\Omega^j\VV_{n,m}$, $\Omega^j V_{n,m}$ are defined as the singular chains of the
 corresponding topological loop spaces. Thus, for example, $\Omega^j\OO_n$ denotes both the topological $j$-loop space of the topological group $\OO_n$ and
 its simplicial set of singular chains. When we need to emphasize that the simplicial version is considered, we write $S_*\Omega^j\OO_n$ or $\Omega^j t\OO_n$.
 We will also consider their smooth and piecewise smooth versions. For example, $\Omega^j dO_n$ and $\Omega^j pdO_n$ denote the simplicial subsets of $\Omega^j tO_n$, whose $k$-simplices are smooth and piecewise smooth maps $D^j\times\Delta^k\to O_n$, respectively, which send a neighborhood of 
 $\partial D^j\times \Delta^k$ to the basepoint. Note that $\Omega^j dO_n$ is a $dR_{j+1}$-algebra, while $\Omega^j pdO_n$  is only a simplicial group.
 

$\bullet$ In some cases we cannot produce an ad hoc loop space. For example, if $G$ is a simplicial group and $X$ a simplicial $G$-set, then even if we have
simplicial models $\Omega^j G$ and $\Omega^j X$ for $j$-loop spaces, compatible with the action, we may not be able to get an immediate model for
 $\Omega^j(X/G)$. The problem is that $\Omega^j X/\Omega^j G$ may not be equivalent to $\Omega^j(X/G)$. (They are equivalent if and only if  $\pi_j X\to\pi_j (X/G)$ is surjective or, equivalently, the connecting
 homomorphism $\pi_j(X/G) \to \pi_{j-1} G$ is zero.) To deal with such situation, we pass to the topological category and as a model for this space we consider
 $\Omega^j|X/G|$ or its simplicial set of singular chains $S_*\Omega^j|X/G|$. For example, this will be the case for the iterated loop space of $\TOP_n/\OO_n$ appearing
 in Theorem~\ref{th:a}, or the loop spaces of homotopy quotients by $\OO_n$ appearing in Theorems~\ref{th:b} and~\ref{th:c}.

\subsection{Homotopy fibers and pullbacks}\label{ss:hofib}
For a map of pointed topological spaces $f\colon X\to Y$, the homotopy fiber is the space of pairs 
\[
\hofiber(X\to Y):=\{(x,\alpha)\,|\, x\in X,\, \alpha\colon [0,1]\to Y, \text{ such that } \alpha(0)=f(x),\, \alpha(1)=*\}.
\]
More generally the homotopy pullback of 
$X\xrightarrow{f} Y\xleftarrow{g} Z$ is the space of triples
\[
X\times_Y^h Z:=\{(x,\alpha,z)\,|\, x\in X,\, z\in Z,\,   \alpha\colon [0,1]\to Y, \text{ such that } \alpha(0)=f(x),\, \alpha(1)=g(z)\}.
\]

In our arguments we often deal with homotopy fibers and pullbacks, that are also defined ad hoc. We consider several examples below. Other homotopy fibers and pullbacks are defined similarly.

$\bullet$ For example, by $\hofiber(\OO_n\to \TOP_n)$ we mean a simplicial group, whose $k$-simplices are pairs $(f_k,g_k)$, where $f_k$ is a $k$-simplex of $t\OO_n$,
and $g_k$ is a germ of homeomorphism of $\R^n\times \Delta^k\times [0,1]$ near $0\times \Delta^k\times [0,1]$ commuting with the projection
on $\Delta^k\times [0,1]$ and such that, $g_k|_{t=0}=f_k$ and $g_k|_{t=1}=id_{\R^n\times\Delta^k}$. One has a natural evaluation map
\[
S_*\hofiber\left(|\OO_n|\to |\TOP_n|\right)\xrightarrow{\,\,\simeq\,\,} \hofiber(\OO_n\to\TOP_n),
\]
which is an equivalence of simplicial groups.

$\bullet$ The simplicial set $\hofiber(\VV_{n,m}\to \VV_{n,m}^t)$ that appears in Theorem~\ref{th:b} is defined to have as $k$-simplices pairs $(f_k,g_k)$,
where $f_k$ is a $k$-simplex of $t\VV_{n,m}$, while $g_k$ is a germ of a locally flat embedding
\[
\R^m\times \Delta^k\times [0,1]\to \R^n\times \Delta^k\times [0,1]
\]
near $0\times\Delta^k\times [0,1]$ commuting with the projection on $\Delta^k\times [0,1]$ and such that $g_k|_{t=0}=f_k$ and $g_k|_{t=1}$ is the standard 
inclusion $\R^m\times \Delta^k\subset \R^n\times \Delta^k$. One also has the evaluation equivalence
\[
S_*\hofiber\left(|\VV_{n,m}|\to |\VV_{n,m}^t|\right)\xrightarrow{\,\,\simeq\,\,} \hofiber(\VV_{n,m}\to\VV_{n,m}^t),
\]

$\bullet$ The simplicial set $\hofiber(O_n\to\PD_n)$ is defined similarly as $\hofiber(\OO_n\to \TOP_n)$, except that    $f_k$ is a $k$-simplex of $pdO_n$, while $g_k$ is
a germ of a PD homeomorphism of $\R^n\times [0,1]\times\Delta^k$ near $0\times \Delta^k\times [0,1]$. One has a zigzag of equivalences
\[
S_*\hofiber\left(|O_n|\to |\PD_n|\right)\xleftarrow{\,\,\simeq\,\,} S^{pl}_*\hofiber^{pl}\left(|O_n|\to |\PD_n|\right)\xrightarrow{\,\,\simeq\,\,} \hofiber(O_n\to\PD_n),
\eqno(\numb)\label{eq:hof_On_PDn}
\]
where $S^{pl}_*\hofiber^{pl}\left(|O_n|\to |\PD_n|\right)$ is a simplicial subset of $S_*\hofiber\left(|O_n|\to |\PD_n|\right)$ whose $k$-simplices
are pairs $(f_k,g_k)$, where $f_k\colon \Delta^k\to |pdO_n|$ and $g_k\colon \Delta^k\times[0,1]\to |PD_n|$ are PL maps, such that $g_k|_{t=0}=f_k$
and $g_k|_{t=1}=*$ is the constant basepoint.

$\bullet$ The homotopy pullbacks are defined in the same manner. For example, the $k$-simplices of $O_n\times_{\TOP_n}^h\TOP_{n,m}$ and
$O_n\times_{\PD_n}^h \PD_{n,m}$ are triples $(f_k,\alpha_k,g_k)$, where $f_k$ is a $k$-simplex of $tO_n$ and $pdO_n$, respectively;
$g_k$ is a $k$-simplex of $\TOP_{n,m}$ and $\PD_{n,m}$, respectively; $\alpha_k$ is a germ of homeomorphism of $\R^n\times \Delta^k\times [0,1]$
near $0\times \Delta^k\times [0,1]$, commuting with the projection on $\Delta^k\times [0,1]$ and such that $\alpha_k|_{t=0}=f_k$, $\alpha_k|_{t=1}=g_k$.
Moreover, $g_k$ is supposed to be PD in case of $O_n\times_{\PD_n}^h \PD_{n,m}$. One has equivalences
\[
S_*\left(|O_n|\times^h_{|\TOP_n|}|\TOP_{n,m}|\right)\xrightarrow{\,\,\simeq\,\,} O_n\times_{\TOP_n}^h\TOP_{n,m},
\]
\[
S_*\left(|O_n|\times^h_{|\PD_n|}|\PD_{n,m}|\right)
\xleftarrow{\,\,\simeq\,\,}S_*^{pl}\left(|O_n|\times^{h,pl}_{|\PD_n|}|\PD_{n,m}|\right)\xrightarrow{\,\,\simeq\,\,} O_n\times_{\TOP_n}^h\TOP_{n,m},
\]
where $S_*^{pl}\left(|O_n|\times^{h,pl}_{|\PD_n|}|\PD_{n,m}|\right)$ is a simplicial subset of $S_*\left(|O_n|\times^h_{|\PD_n|}|\PD_{n,m}|\right)$ defined in the same way as $S^{pl}_*\hofiber^{pl}\left(|O_n|\to |\PD_n|\right)$ above.

\subsubsection{Loop spaces of  homotopy fibers and pullbacks}\label{sss:loops}
The loop spaces of homotopy fibers and more generally of homotopy pullbacks are defined by combining the  constructions from Subsections~\ref{ss:loops}
and~\ref{ss:hofib}.

$\bullet$ For example, the simplicial set $\Omega^j\left(O_n\times_{\PD_n}^h \PD_{n,m}\right)$ has as $k$-simplices triples $(f_k,\alpha_k,g_k)$,
where $f_k$ is a $k$-simplex of $\Omega^jpdO_n$, $g_k$ is a $k$-simplex of $\Omega^j\PD_{n,m}$, while $\alpha_k$ is a germ of a PD homeomorphism
of  $\R^n\times D^j\times \Delta^k\times [0,1]$ near $0\times D^j\times \Delta^k\times [0,1]$, commuting with the projection on $ D^j\times \Delta^k\times [0,1]$, identity near $\R^n\times \partial D^j\times \Delta^k\times [0,1]$, and such that $\alpha_k|_{t=0}=f_k$ and $\alpha_k|_{t=1}=g_k$. 
It is
equivalent to the chains on $\Omega^j\left(|O_n|\times_{|PD_n|}^h|PD_{n,m}|\right)$ by means of the following zigzag
\[
S_*\Omega^j\left(|O_n|\times_{|PD_n|}^h|PD_{n,m}|\right)\xleftarrow{\,\,\simeq\,\,}
S^{pl}_*\Omega^j_{pl}\left(|O_n|\times_{|PD_n|}^{h,pl}|PD_{n,m}|\right)\xrightarrow{\,\,\simeq\,\,}
\Omega^j\left(O_n\times_{\PD_n}^h \PD_{n,m}\right)
\]
with the easily guessable term in the middle. 

$\bullet$ Unlike $\Omega^j\left(O_n\times_{\PD_n}^h \PD_{n,m}\right)$, which is just a simplicial set, not even a monoid or an $E_1$-algebra,
$\Omega^j\left(O_n\times_{\TOP_n}^h \TOP_{n,m}\right)$ is an $R_{j+1}$-algebra being an $E_j$-algebra in the category of associative monoids. 
In the next Subsection we describe its natural $\OO_m\times\OO_{n-m}$-action, which turns it into an $R_{m+1}^{\OO_m\times\OO_{n-m}}$-algebra.

\subsection{$\OO_m\times\OO_{n-m}$-action on loop spaces}\label{ss:O_action_loops}
In Subsection~\ref{ss:fr_operads} we explained how $\OO_m\times\OO_{n-m}$ acts on $\Embmn$, $\Embfrmn$, $\tEmbmn$, 
$\Emb_{\partial D^m\times D^{n-m}}(D^m\times D^{n-m})$ and how this action comes as  part of $E_{m}^{\OO_m\times\OO_{n-m}}$- or $R_{m+1}^{\OO_m\times\OO_{n-m}}$-structure. 
We now explain how the corresponding structure appears in the loop spaces that we consider.

The May-Boardman-Vogt theorem says that an $E_m$-space $X$ is $E_m$-equivalent to an $m$-loop space if and only if $\pi_0 X$ is a group~\cite{BoardVogt,May}.
One has a similar result for $E_m^{\OO_m}$-spaces \cite{SalvWahl}. Given a pointed space $Y$ with an $\OO_m$-action preserving the basepoint, the
loop space $\Omega^mY$ is naturally an $E_m^{\OO_m}$-algebra, for which the $\OO_m$-action is defined as follows
\[
g\star f:=g\cdot f\circ g^{-1}: D^m\xrightarrow{g^{-1}} D^m\xrightarrow{f} Y\xrightarrow{g\cdot(-)} Y,
\]
$f\in\Omega^m Y$, $g\in \OO_m$. The first named author and Wahl proved that an $E_m^{\OO_m}$-algebra $X$ is $E_m^{\OO_m}$-equivalent to $\Omega^m Y$ 
for some pointed $\OO_m$-space $Y$ if and only if $\pi_0X$ is a group \cite{SalvWahl}.

In case $Y$ is a pointed $\OO_m\times\OO_{n-m}$-space (the action is required to preserve the basepoint), then $\Omega^m Y$ is an 
$E_{m}^{\OO_m\times\OO_{n-m}}$-algebra. The action of $\OO_{n-m}$ on $\Omega^m Y$ is defined pointwise:
\[
g\star f:=g\cdot f\colon D^m\xrightarrow{f} Y\xrightarrow{g\cdot(-)} Y,
\]
$f\in\Omega^m Y$, $g\in \OO_{n-m}$.

$\bullet$ For example, $\Omega^m\hofiber(\VV_{n,m}\to \VV_{n,m}^t)$ is an $E_{m}^{\OO_m\times\OO_{n-m}}$-algebra, where $\OO_m\times\OO_{n-m}$
acts on $\VV_{n,m}=\OO_n/\OO_{n-m}$ and on $\VV_{n,m}^t=\TOP_n/\TOP_{n,m}$ by conjugation, which is well-defined since 
$\OO_m\times\OO_{n-m}$ lies in the normalizer $N(\TOP_{n,m})\subset \TOP_n$. Since $\OO_{n-m}\subset\TOP_{n,m}$, one can equivalently define the
$\OO_m$-action by conjugation, and the $\OO_{n-m}$-action by multiplication from the left. Note that this action preserves the basepoint -- the constant inclusion 
$\R^m\subset\R^n$.

$\bullet$ In case $Y$ is a group or an associative unital monoid with $\OO_m\times\OO_{n-m}$ acting on it by endomorphisms, the loop space $\Omega^m Y$ becomes an 
 $R_{m+1}^{\OO_m\times\OO_{n-m}}$-algebra. For example, that is the case of $\Omega^m \left( \OO_n\times_{\TOP_n}^h\TOP_{n,m}\right)$,
 where $\OO_m\times\OO_{n-m}$ acts on $\OO_n$, $\TOP_n$, $\TOP_{n,m}$ by conjugation.

\subsection{Embeddings modulo immersions spaces}\label{ss:modulo} 
In this section we explain precisely what simplicial sets are meant by 
\[\Dbarn:=\hofiber\left(\Diffn\xrightarrow{D}\Omega^n O_n\right),
\]
\[
 \Ebarmn:=\hofiber\left(\Embmn\xrightarrow{D}\Omega^mV_{n,m}\right),\]
 \begin{multline*}
 \Ebar_{\partial D^m\times D^{n-m}}(D^m\times D^{n-m}):=\\
 \hofiber\left(\Emb_{\partial D^m\times D^{n-m}}(D^m\times D^{n-m})\xrightarrow{D} \dMap_{\partial D^m\times D^{n-m}}(D^m\times D^{n-m},O_n)\right),
 \end{multline*}
and how they are acted upon by $R_{m+1}^{\OO_m\times\OO_{n-m}}$ and $E_{m}^{\OO_m\times\OO_{n-m}}$. The map $D$ in all the three 
cases is the differential. 

The $k$-simplices of $\Ebarmn$ are pairs $(f_k,h_k)$, where $f_k$ is a $k$-simplex of $\Embmn$, i.e. $f_k$ is a smooth embedding 
\[
f_k\colon D^m\times\Delta^k\hookrightarrow D^n\times\Delta^k,
\]
commuting with the projection on $\Delta^k$ and equal to the identity inclusion near $\partial D^m\times\Delta^k$; and $h_k$ is a smooth map
$h_k\colon D^m\times\Delta^k\times[0,1]\to V_{n,m}$ sending a neighborhood of $\partial D^m\times\Delta^k\times[0,1]$ to the basepoint~$*$ and such that $h_k|_{t=0}=D(f_k)$, $h_k|_{t=1}=*$.

Similarly, the $k$-simplices of $\Ebar_{\partial D^m\times D^{n-m}}(D^m\times D^{n-m})$ are pairs $(f_k,h_k)$, where $f_k$  is a $k$-simplex
of $\Emb_{\partial D^m\times D^{n-m}}(D^m\times D^{n-m})$ and $h_k\colon D^m\times D^{n-m}\times\Delta^k\times [0,1]\to O_n$ is smooth, such that $h_k|_{t=0}=D(f_k)$ and
$h_k\equiv *$ at $t=1$ and near $\partial D^m\times D^{n-m}\times\Delta^k\times[0,1]$.

 The spaces are acted upon by $\OO_m\times\OO_{n-m}$ (or rather by its smooth variant $d\OO_m\times d\OO_{n-m}$). For simplicity we describe this action at the level of vertices (0-simplices). For $(g_1,g_2)\in\OO_m\times\OO_{n-m}$ and $(f,h)\in\Ebarmn$, the result of the action $(g_1,g_2)\star (f,h)$ is a pair in which the first element is the embedding
 \[
 D^m\xrightarrow{g_1^{-1}}D^m\xrightarrow{f}D^n\xrightarrow{(g_1,g_2)} D^n,
 \]
 while the second element is the composition
 \[
 D^m\times [0,1]\xrightarrow{g_1^{-1}\times id_{[0,1]}}D^m\times [0,1]\xrightarrow{h} V_{n,m}\xrightarrow{(g_1,g_2)\cdot(-)}V_{n,m},
 \]
 where $\OO_m\times\OO_{n-m}$ acts on $V_{n,m}=O_n/O_{n,m}$ by conjugation (equivalently, $\OO_m$ acts by conjugation, while $\OO_{n-m}$
 acts by the left multiplication).

 For $g\in\OO_m\times\OO_{n-m}$ and $(f,h)\in \Ebar_{\partial D^m\times D^{n-m}}(D^m\times D^{n-m})$, the result of the action 
 $g\star (f,h)$ is a pair in which the first element is the embedding $g\circ f\circ g^{-1}$, while the second element is the composition
 \[
 D^m\times D^{n-m}\times[0,1]\xrightarrow{g^{-1}\times id_{[0,1]}}D^m\times D^{n-m}\times[0,1]\xrightarrow{h}O_n\xrightarrow{g\cdot(-)\cdot g^{-1}} O_n.
 \]
 
 It is obvious that the spaces are $E_m$-algebras. (Though one should be careful with the $E_m$-action on the second component $h$ since the conjugation by a positive rescaling of the first
  factor $\R^m$ in $\R^m\times\R^{n-m}=\R^{n}$ acts non-trivially on $O_n$.) Moreover, the $E_m$-action is compatible with the  $\OO_m\times\OO_{n-m}$-action
 and so the spaces are $E_m^{\OO_m\times\OO_{n-m}}$-algebras (or rather $dE_m^{d\OO_m\times d\OO_{n-m}}$-algebras). 
 
 This structure extends to an $R_{m+1}^{\OO_m\times\OO_{n-m}}$-action for $\Ebar_{\partial D^m\times D^{n-m}}(D^m\times D^{n-m})$.
 In the particular case $n=m$, we get an $R_{n+1}^{\OO_n}$-action on $\Dbarn$. In other words, we claim that the aforementioned spaces are
 $E_m^{\OO_m\times\OO_{n-m}}$-algebras in the category of associative unital monoids (in fact, groups for $\Dbarn$).  However, one should be a little bit careful
 with this structure as the (differential) map $D$ appearing in the definition of these spaces is not a monoid homomorphism. 
 The associative product on $\Ebar_{\partial D^m\times D^{n-m}}(D^m\times D^{n-m})\ni (f,h),(f',h')$ is defined as
 \[
 (f,h)*(f',h')=(f\circ f',h(f'(x),t)\cdot h'(x,t)),
 \]
 while the inverse in $\Dbarn$ is given by
 \[
(f,h)^{-1}=(f^{-1},h(f^{-1}(x),t)^{-1}).
 \]

\section{Smale-Hirsch type theorems for smooth, topological, and PL immersions}\label{s:imm}
Let $(M,\partial_0 M,\partial_1M)$ be a manifold triad, which means that the boundary of $M^m$ is split into 
two submanifolds 
 $\partial M=\partial_0 M\cup_{\partial_{01}M}\partial_1 M$ by a submanifold $\partial_{01}M=\partial_0 M\cap\partial_1M$.
In the smooth category one can assume that $M$ is a manifold with  corners of codimension $\leq 2$, the codimension two corner part being $\partial_{01}M$. We will be assuming $(M_0,\partial_0M)\subset
(N,\partial N)$, where $M_0$ is an open neighborhood of $\partial_0M$ in $M$, the inclusion being denoted by~$i$.
The Smale-Hirsch theorem says that in case $m<n$, or $m=n$ and $\partial_1 M$ intersecting each path-component of~$M$,  
the space $\Imm_{\partial_0M}(M,N)$ of immersions $M\looparrowright N$ is equivalent to the space of formal immersions \cite{Smale_imm,Hirsch}. The latter
consists of pairs $(f,F)$, where $f\colon M\to N$ is any map, while $F\colon TM\to TN$ is a monomorphism of tangent bundles lifting~$f$.
The results hold in all the three categories: smooth~\cite{Smale_imm,Hirsch}, topological~\cite{Lees,Kurata}, and PL~\cite{HaeflPoen_imm}. It also has a PD variant 
for the case $n=m$ \cite{BL_diff}. In our paper we refer to these formal immersions as $V_{n,m}$-framed, $\VV_{n,m}^t$-framed, $\VV_{n,m}^{pl}$-framed,
and $\PD_n$-framed maps. The spaces of such maps will be denoted by $\dMap_{\partial_0M}^{V_{n,m}}(M,N)\simeq\Map_{\partial_0M}^{V_{n,m}}(M,N)$,
$\Map_{\partial_0M}^{\VV_{n,m}^t}(M,N)$, $\plMap_{\partial_0M}^{\VV_{n,m}^{pl}}(|M|,|N|)$, $\pdMap_{\partial_0M}^{\PD_n}(|M|,N)$.
(For the last space one requires $n=m$.) The $k$-simplices of these spaces are pairs $(f_k,F_k)$, where $f_k$ is a $k$-simplex of the corresponding 
mapping  space  $\dMap_{\partial_0M}(M,N)$,
$\Map_{\partial_0M}(M,N)$, $\plMap_{\partial_0M}(|M|,|N|)$, $\pdMap_{\partial_0M}(|M|,N)$, while $F_k$ is a monomorphism of tangent bundles
lifting~$f_k$:
\begin{equation}\label{eq:tangent_mono}
\xymatrix{
TM\times\Delta^k\ar[d]_{\pi\times id_{\Delta^k}}\ar[r]^{F_k}&TN\times\Delta^k\ar[d]^{\pi\times id_{\Delta^k}}\\
M\times\Delta^k\ar[r]^{f_k}& N\times\Delta^k.
}
\end{equation}
For  $\dMap_{\partial_0M}^{V_{n,m}}(M,N)$ and $\Map_{\partial_0M}^{V_{n,m}}(M,N)$, by $TM$ and $TN$ we mean the usual tangent vector bundle,
while $F_k$ must be a fiberwise linear monomorphism, coinciding with $D(i)\times id_{\Delta^k}$ near $\partial M\times\Delta^k$ (where $i\colon M\subset N$), which is additionally required to be smooth in case of $\dMap_{\partial_0M}^{V_{n,m}}(M,N)$. 
For $\Map_{\partial_0M}^{\VV_{n,m}^t}(M,N)$, $\plMap_{\partial_0M}^{\VV_{n,m}^{pl}}(|M|,|N|)$, $\pdMap_{\partial_0M}^{\PD_n}(|M|,N)$,
by $TM$ and $TN$ we mean the topological or PL tangent microbundles. The latter are defined as  neighborhoods of the diagonals $\Delta_M$ and $\Delta_N$ in
$M\times M$ and  $N\times N$, respectively, with the projection on the first factor. Therefore, $F_k$ is in fact a germ of a {\it parameterized  locally flat fiberwise embedding} $M^{\times 2}\times
\Delta^k\to N^{\times 2}
\times\Delta^k$ near $\Delta_M\times \Delta^k$ that extends
 $\Delta_M\times \Delta^k\xrightarrow{\Delta_{f_k}\times id_{\Delta^k}} \Delta_N\times\Delta^k$, makes the square~\eqref{eq:tangent_mono} commute,
where $\pi\colon M\times M\to M$ is the projection onto the first factor, and such that $F_k=i\times i\times id_{\Delta^k}$ near $\Delta_{\partial M}\times \Delta^k$. The map $f_k$ does not have to be injective.  
To make the required property of {\it parameterized  local flatness of a fiberwise embedding} precise, we replace $F_k$ by 
$$\left(G_k:=(p_1,p_2\circ F_k,p_3)\right)\colon M\times M\times\Delta^k\to
M\times N\times\Delta^k.$$ 
 Then the $k$ simplices are described as germs of parameterized locally flat\footnote{For {\it parameterized local flatness}, see Subsection~\ref{ss:spaces}. 
Here, the parameterization is taken with respect to $M\times\Delta^k$.} embeddings $G_k$ near $\Delta_M\times\Delta^k$ (extending
$\left(id_M\times (f_k\circ p_{23})\right)|_{\Delta_M\times\Delta^k}$) commuting with the projection on $M\times\Delta^k$ and coinciding with $(p_1,i\circ p_2,p_3)$ near $\partial M\times M\times\Delta^k$.\footnote{In other words, we replace monomorphisms of tangent microbundles $F\colon TM\to TN$ lifting $f\colon M\to N$,
by monomorphisms of microbundles $G\colon TM\to f^*TN$ over $M$.\label{footnote_imm}} The map $G_k$ (or $F_k$)  is additionally required to be PL and PD in case of the spaces $\plMap_{\partial_0M}^{\VV_{n,m}^{pl}}(|M|,|N|)$
and  $\pdMap_{\partial_0M}^{\PD_n}(|M|,N)$, respectively. 

One has natural inclusions, which are equivalences (provided the handle dimension of $M$ relative to $\partial_0M$ is less than~$n$) by the corresponding Smale-Hirsch type theory~\cite{Smale_imm,Hirsch,Lees,Kurata,HaeflPoen_imm}:
\[
\Imm_{\partial_0 M}(M,N)\xrightarrow{\,\,\simeq\,\,} \dMap_{\partial_0M}^{V_{n,m}}(M,N)\xrightarrow{\,\,\simeq\,\,} 
\Map_{\partial_0M}^{V_{n,m}}(M,N);
\]
\[
\tImm_{\partial_0 M}(M,N) \xrightarrow{\,\,\simeq\,\,} \Map_{\partial_0M}^{\VV_{n,m}^t}(M,N);
\]
\[
\plImm_{\partial_0 M}(|M|,|N|) \xrightarrow{\,\,\simeq\,\,} \plMap_{\partial_0M}^{\VV_{n,m}^{pl}}(|M|,|N|);
\]
\[
\pdImm_{\partial_0 M}(|M|,N) \xrightarrow{\,\,\simeq\,\,} \pdMap_{\partial_0M}^{\PD_n}(|M|,N).
\]

For a smooth immersion $f\colon M\looparrowright N$, we denote by $D(f)\colon TM\to TN$ the induced
monomorphism of tangent vector bundles. For a topological or PL immersion $f\colon M\looparrowright N$, we denote by $\Delta(f):=(f,f)\colon M\times M\to N\times N$ the induced monomorphism of
 tangent microbundles. 

\subsection{Identifying  tangent vector bundle with  topological tangent microbundle}\label{ss:tangent_id}

In the sequel we  distinguish the two tangent bundles and 
denote by $TM$ only the tangent vector bundle of a smooth manifold $M$. To identify 
the two, we choose a Riemannian metric on $M$ and  define a map $exp\colon TM\to M\times M$. For any $x\in M$, we define $exp_x\colon T_xM\to  M$
by sending $v\in T_xM$ to the endpoint of the geodesic which starts from $x$ with velocity $v$ and goes  with the constant speed for one unit of time. For any $u\in TM$, we assign $exp(u):=(\pi(u),exp_{\pi(u)}(u))$. The obtained map is an isomorphism of the tangent vector 
bundle (viewed as a topological microbundle) with the topological tangent microbundle. Slightly abusing notation, we will be writing $exp\colon TM\to M\times M$
and $exp^{-1}\colon M\times M\to TM$ meaning this isomorphism and its inverse.

In case we have a pair of manifolds $i\colon M\subset N$, we choose compatibly metrics on $M$ and $N$ so that the square
\begin{equation}\label{eq:exp_MN}
\xymatrix{
TM\ar[d]_{exp_M}\ar[r]^{D(i)}&TN\ar[d]^{exp_N}\\
M\times M\ar[r]^{\Delta(i)}& N\times N
}
\end{equation}
commutes. The commutativity of the square is equivalent to the requirement that the inclusion $i$ preserves geodesics. Such metric  can be easily defined in a tubular neighborhood of $M$ in $N$ and then extended globally using a partition of unity.

\begin{remark}\label{r:exp} 
While  the identification of the tangent vector bundle with the topological microbundle is straightforward, it is not easy to be done in the
PL case. The problem is that the fibers of the tangent vector bundle have the induced PL structure being vector spaces. At the same time the 
map $exp\colon TM\to M\times M=|M|\times |M|$ in most of the cases does not respect this PL structure. However, for our purposes it is enough to 
do it when $M$ is a codimension zero submanifold of a Euclidean space, such as $D^m\subset\R^m$ or $D^m\times D^{n-m}\subset\R^n$. In case of $\R^n$ (or any $M^n\subset \R^n$) the map
$exp:T\R^n=\R^n\times\R^n\to \R^n\times\R^n$, $exp\colon (x,y)\mapsto (x,x+y)$, is linear, and therefore, preserves the PL structure of the fibers.
\end{remark}

We end this section with the following.

\begin{proposition}\label{p:top_nm}
For any $n>m\geq 1$, the  scanning maps
\begin{gather}
\tImm_{\partial D^m\times D^{n-m}}(D^m\times D^{n-m}\mmod D^m\times 0)\xrightarrow{\,\,\simeq\,\,} \Omega^m\TOP_{n,m},\label{eq:top_nm2}\\
\plImm_{\partial D^m\times D^{n-m}}(D^m\times D^{n-m}\mmod D^m\times 0)\xrightarrow{\,\,\simeq\,\,} \Omega^m\PL_{n,m},\label{eq:pl_nm}\\
\pdImm_{\partial D^m\times D^{n-m}}(D^m\times D^{n-m}\mmod D^m\times 0)\xrightarrow{\,\,\simeq\,\,} \Omega^m\PD_{n,m}\label{eq:pd_nm}
\end{gather}
are equivalences.
\end{proposition}

The maps \eqref{eq:top_nm2}-\eqref{eq:pl_nm}-\eqref{eq:pd_nm} are defined as obvious restrictions of maps $D^m\times D^{n-m}\times\Delta^k\to D^m\times
 D^{n-m}\times\Delta^k$ on a
germ-neighborhood of $D^m\times 0\times\Delta^k$, see Subsections~\ref{ss:spaces} and~\ref{ss:loops} where these
spaces are defined.
%

\begin{proof}
To prove \eqref{eq:top_nm2}, consider the commutative diagram whose columns are fiber sequences:
\[
\xymatrix{
\tImm_{\partial D^m\times D^{n-m}}(D^m\times D^{n-m}\mmod D^m\times 0)\ar[d]\ar[r]&\Omega^m\TOP_{n,m}\ar[d]\\
\tImm_{\partial D^m\times D^{n-m}}(D^m\times D^{n-m})\ar[r]^-{\simeq}\ar[d]&\Omega^m\TOP_{n}\ar[d]\\
\tImm_\partial(D^m\times 0,D^m\times D^{n-m})\ar[r]^-\simeq&\Omega^m\VV_{n,m}^t.
}
\]
The middle and bottom horizontal maps are equivalences by the topological version of the Smale-Hirsch theorem \cite{Lees,Kurata}, which implies that so is the top
one. The proof of  \eqref{eq:pl_nm} is identical. To prove \eqref{eq:pd_nm}, consider the commutative square
\[
\xymatrix{
\plImm_{\partial D^m\times D^{n-m}}(D^m\times D^{n-m}\mmod D^m\times 0)\ar[d]^-\simeq\ar[r]^-\simeq&\Omega^m\PL_{n,m}\ar[d]^-\simeq\\
\pdImm_{\partial D^m\times D^{n-m}}(D^m\times D^{n-m}\mmod D^m\times 0)\ar[r]&\Omega^m\PD_{n,m}.
}
\]
Since three out of four arrows are equivalences, so is the fourth one.
\end{proof}

\section{Homotopy cartesian squares of smoothing theory}\label{s:squares}
In this section we recollect some smoothing theory results relating homeomorphism, diffeomorphism, embedding  and (formal) immersion spaces and restate them  in terms of homotopy pullback squares, see Theorems~\ref{th:square_diff} and~\ref{th:square_emb}.  Note that the notion of homotopy pullback square (or, equivalently, homotopy cartesian square) appeared about the same time when the smoothing theory was developing, but it 
was not widely used back then. To compare,
Mather's paper~\cite{Mather}, where the term was formally introduced and studied, appeared one year after Burghelea-Lashof's work~\cite{BL_diff}, which is one of the main references for this section. 

Recall that a commutative square of spaces
\begin{equation}\label{eq:square}
\xymatrix{
X\ar[d]^-f\ar[r]^e&Y\ar[d]\\
Z\ar[r]^g&W
}
\end{equation}
is {\it homotopy cartesian} if the induced map $X\to Z\times^h_W Y$ is a weak equivalence. This property holds if and only if for any choice of $z\in Z$, the induced map
\begin{equation}\label{eq:map_of_fibers}
\hofiber_z(X\to Z)\to \hofiber_{g(z)}(Y\to W)
\end{equation}
is a weak equivalence. 

We will need two  technical lemmas.

\begin{lemma}\label{l:cartesian}
A commutative square~\eqref{eq:square} is homotopy cartesian if and only if for every~$z\in Z$, the map~\eqref{eq:map_of_fibers} is bijective on~$\pi_0$, and for every~$x\in X$, the map
$$
\hofiber_{f(x)}(X\to Z)\to \hofiber_{g(f(x))}(Y\to W)
$$
is a weak equivalence when restricted on the component of $x$ in the domain and on the component of $e(x)$ in the codomain.
\end{lemma}
In the above, abusing notation, $x$ denotes the pair $(x,\gamma_{f(x)})$, where $\gamma_{f(x)}(t)\equiv f(x)$ is the constant path in~$Z$, and same for $e(x)$.

\begin{proof}
The direct implication is straightforward. So, we check the converse one. It is easy to see that if~\eqref{eq:map_of_fibers} is surjective (respectively, injective)
on~$\pi_0$ for every~$z$, then  the induced map $\pi_0(X)\to \pi_0\left(Z\times_W^h Y\right)$ is also surjective (respectively, injective).  Thus, it is enough
to check that $\pi_i(X,x)\to \pi_i\left(Z\times_W^h Y, x\right)$ is an isomorphism  for every $x\in X$ and every $i\geq 1$. The latter follows from the following morphism of fiber sequences and 
the five lemma (with some care taken when $i=1$):
\[
\xymatrix{
\hofiber_{f(x)}(X\to Z)\ar[d]\ar[r]&X\ar[d]\ar[r]&Z\ar@{=}[d]\\
\hofiber_{g(f(x))}(Y\to W)\ar[r]&Z\times_W^hY\ar[r]& Z.
}
\]

\end{proof}

\begin{lemma}\label{l:cartesian2}
Given maps of spaces $X\xrightarrow{f}Y\xrightarrow{g}Z$ and $z\in Z$, the square
\[
\xymatrix{
\hofiber_z(X\to Z)\ar[d]\ar[r]&\hofiber_z(Y\to Z)\ar[d]\\
X\ar[r]&Y
}
\]
is homotopy cartesian.
\end{lemma}

\begin{proof}
The right map is a fibration, while the square is a pullback. Hence, it is a homotopy pullback.
\end{proof}

%
%

\subsection{Diffeomorphisms, homeomorphisms and formal immersions}\label{ss:square1}
As it is explained below, the following is an immediate consequence of results from~\cite{KS_essays,BL_diff}.

\begin{theorem}\label{th:square_diff}
(a) Let $N$ be a smooth manifold of dimension $n\neq 4$, then the following square is homotopy cartesian:
\begin{equation}\label{eq:square_top}
\xymatrix{
\Diff_\partial(N)\ar[d]\ar[r]&\dMap_\partial^{O_n}(N)\ar[d]\\
\Homeo_\partial(N)\ar[r]&\Map_\partial^{\TOP_n}(N).
}
\end{equation}
(b) Let $N$ be a smooth compact manifold of any dimension $n$, then the following square is homotopy cartesian:
\begin{equation}\label{eq:square_pd}
\xymatrix{
\Diff_\partial(N)\ar[d]\ar[r]&\dMap_\partial^{O_n}(N)\ar[d]\\
\pdHomeo_\partial(N)\ar[r]&\pdMap_\partial^{\PD_n}(N).
}
\end{equation}

\end{theorem}

\begin{remark}\label{r:sq_non-comm}
Strictly speaking, the way we (following Burghelea-Lashof~\cite{BL_diff}) define $\dMap_\partial^{O_n}(N)$ makes squares~\eqref{eq:square_top} and~\eqref{eq:square_pd}
commute only up to homotopy. A remedy to it would be to consider smooth tangent microbundles instead of tangent vector bundles. Another way to make~\eqref{eq:square_top}
commute on the nose is to replace  $\dMap_\partial^{O_n}(N)$ by an equivalent space $\Map_\partial^{O_n\times^h_{\TOP_n}\TOP_n}(N)$, see Section~\ref{s:fr_homeo}.
\end{remark}

\begin{proof}[Proof of Theorem~\ref{th:square_diff}]
\cite[Theorem~4.2]{BL_diff} states that for any smooth compact manifold $N$, one has natural maps 
\begin{equation}\label{eq:sections}
\Homeo_\partial(N)/\Diff_\partial(N)\to \Gamma_\partial(P^t,N) \,\,\text{ and }\,\, \pdHomeo_\partial(N)/\Diff_\partial(N)\to \Gamma_\partial(P^{pd},N),
\end{equation}
which are injective on $\pi_0$ and equivalences componentwise. Here, $\Gamma_\partial(-,N)$ stands for the space of sections over $N$ relative to the boundary, $P^t\to N$ is a fibration with fiber $\TOP_n/\OO_n$ encoding the space of vector bundle structures on the
topological tangent microbundle, and $P^{pd}\to N$ has fibers $\PD_n/\OO_n$ and it encodes the space of vector bundle structures on the PL tangent microbundles on~$|N|$. Both 
spaces of sections have a basepoint~$s_0$ -- the  vector bundle structure on $TN$. Moreover, $\dMap_\partial^{O_n}(N)$ and $\Map_\partial^{\TOP_n}(N)$ are monoids with the former acting
on $\pdMap_\partial^{\PD_n}(N)$. The quotient spaces $\Map_\partial^{\TOP_n}(N)/\dMap_\partial^{O_n}(N)$ and $\pdMap_\partial^{\PD_n}(N)/\dMap_\partial^{O_n}(N)$
can be identified with a disjoint union of components of $\Gamma_\partial(P^t,N)$ and $ \Gamma_\partial(P^{pd},N)$, respectively, with the maps~\eqref{eq:sections} factoring through them
\cite[Proposition~1.4]{BL_diff}. This implies that $\Diff_\partial(N)$ is equivalent to the homotopy fiber over $s_0$ of the maps
\begin{equation}\label{eq:sections2}
\Homeo_\partial(N)\to \Gamma_\partial(P^t,N) \,\,\text{ and }\,\, \pdHomeo_\partial(N)\to \Gamma_\partial(P^{pd},N),
\end{equation}
 and so is 
$\dMap_\partial^{O_n}(N)$ for the maps
\begin{equation}\label{eq:sections3}
\Map_\partial^{\TOP_n}(N)\to \Gamma_\partial(P^t,N) \,\,\text{ and }\,\, \pdMap_\partial^{\PD_n}(N)\to \Gamma_\partial(P^{pd},N).
\end{equation}
By applying Lemma~\ref{l:cartesian2}, we  conclude that the squares \eqref{eq:square_top} and \eqref{eq:square_pd} are homotopy cartesian. 

To see that \eqref{eq:square_top} is homotopy cartesian for a non-compact $N$, one can recall \cite[Essay~V]{KS_essays} in which  a simplicial space ${\mathrm{Smooth}}_\partial(N)$ of smooth structures
on $N$, coinciding with the given one near $\partial N$, is defined. Its $k$-simplices are smooth structures on $N\times\Delta^k$ for which the projection to $\Delta^k$ is a submersion. 
It is then proved in \cite[Theorem~2.2]{KS_essays} that $\mathrm{Smooth}_\partial(N)$ is equivalent to $\Gamma_\partial(P^t,M)$, while
\[
\Diff_\partial (N)\to \Homeo_\partial(N)\to \mathrm{Smooth}_\partial(N)
\]
is a fiber sequence, which similarly implies (a) for any smooth $N$.
\end{proof}

\begin{remark}\label{r:rourke}
Kirby-Siebenmann~\cite{KS_essays} consider only spaces of smooth and PL structures on topological manifolds. 
In \cite[Essay~5, $\S$0, p.~219]{KS_essays} they refer to upcoming lecture notes by C.~P.~Rourke where similar spaces of smooth structures on PL (not necessarily compact) manifolds are considered and are proved to be equivalent to~$\Gamma_\partial(P^{pd},N)$. However, we were only able to find a preprint by Rourke~\cite{Rourke}, where this result
 \cite[Theorem~1]{Rourke} is stated with  a sketch of  a proof.
\end{remark}

\subsection{Embeddings and immersions}\label{ss:square2}
Let $(M,\partial_0M,\partial_1M)$ and $N$ be as in Section~\ref{s:imm}. We additionally assume that $M$ and $N$ are smooth and $M$ is compact.

\begin{theorem}\label{th:square_emb}
(a) For $m\leq n\neq 4$,  the following square is homotopy cartesian:
\begin{equation}\label{eq:square_kupers}
\xymatrix{
\Emb_{\partial_0 M}(M,N)\ar[r]\ar[d]& \Imm_{\partial_0 M}(M,N)\ar[d]\\
\tEmb_{\partial_0 M}(M,N)\ar[r]& \tImm_{\partial_0 M}(M,N).
}
\end{equation}

(b) For $m=n$,  the following square is  homotopy cartesian:
\begin{equation}\label{eq:square_kupers3}
\xymatrix{
\Emb_{\partial_0 M}(M,N)\ar[r]\ar[d]& \Imm_{\partial_0 M}(M,N)\ar[d]\\
\pdEmb_{\partial_0 M}(|M|,N)\ar[r]& \pdImm_{\partial_0 M}(|M|,N).
}
\end{equation}
\end{theorem}

\begin{proof}
The case $n=m$ of this theorem is due to Burghelea and Lashof~\cite{BL_diff}.  
 To be precise, 
\cite[Theorem~3.1(a)]{BL_diff} states that for any $x\in \Imm_{\partial_0 M}(M,N)$, the induced map
\begin{equation}\label{eq:BL_bijection}
\pi_i (\Imm_{\partial_0 M}(M,N), \Emb_{\partial_0 M}(M,N);x)\to \pi_i (\Imm_{\partial_0 M}(M,N), \Emb_{\partial_0 M}(M,N);x)
\end{equation}
is bijective for any $i\geq 1$, assuming $x\in\Emb_{\partial_0 M}(M,N)$ in case $i\geq 2$.  (\cite[Theorem~3.1(b)]{BL_diff} states the same in the PD situation.) It is not hard to see that this property is
equivalent to the condition of Lemma~\ref{l:cartesian}
implying
homotopy cartesianity of the squares~\eqref{eq:square_kupers} and~\eqref{eq:square_kupers3}.  

The case $n>m$ of (a) was considered by Lashof in~\cite{Lashof}. However, the statement of its \cite[Theorem~A~$(t/d)$]{Lashof} is slightly
weaker. Namely, it claims bijectivity of~\eqref{eq:BL_bijection} only provided $x\in\Emb_{\partial}(M,N)$, which is not enough. However, his argument can be easily
adjusted. The idea is to prove the homotopy cartesianity of the case $m<n$ from the case $m=n$. We must show that for any 
$x\in \Imm_{\partial}(M,N)$, the induced map of homotopy fibers is an equivalence:
\[
\hofiber\bigl( \Emb_{\partial}(M,N)\to  \Imm_{\partial}(M,N)\bigr)_x \xrightarrow{\simeq} \hofiber\bigl( \tEmb_{\partial}(M,N) \to  \tImm_{\partial}(M,N)\bigr)_x.
\]
For the smooth immersion $x\colon M\looparrowright N$, let $T$ denote the normal bundle of $M$ in $N$ immersed as a tubular neighborhood by means of
$\tilde x\colon T\looparrowright N$ (with $\partial_0 T\subset \partial T$ being a tubular neighborhood of $\partial M$ in $\partial N$).
The result follows from the following commutative square of weak equivalences:
\[
\xymatrix{
 \hofiber\bigl( \Emb_{\partial_0 T}(T,N)\to  \Imm_{\partial_0 T}(T,N)\bigr)_{\tilde x} \ar[d]^\simeq_{\text{Theorem~\ref{th:square_diff}(a) for }m=n}\ar[r]^\simeq&
  \hofiber\bigl( \Emb_{\partial}(M,N) \to  \Imm_{\partial}(M,N)\bigr)_x  \ar[d]\\
\hofiber\bigl( \tEmb_{\partial_0 T}(T,N)\to  \tImm_{\partial_0 T}(T,N)\bigr)_{\tilde x} \ar[r]^\simeq&   
 \hofiber\bigl( \tEmb_{\partial}(M,N) \to  \tImm_{\partial}(M,N)\bigr)_x.
}
\]
\end{proof}

%

\section{Spaces of $\underline{O_n}$-framed homeomorphisms}\label{s:fr_homeo} 

In order to prove  Theorem~\ref{th:a}(a), we need to introduce an intermediate space $\Homeo_\partial^{\underline{O_n}}(D^n)$
of $\underline{O_n}$-framed homeomorphisms of $D^n$. We will define such a space  $\Homeo_\partial^{\underline{O_n}}(N)$ for any smooth $n$-manifold $N$
possibly with a boundary. The construction is interesting in its own right. The vertices of this simplicial set are triples $(f,F,H)$, where $f\in\Homeo_\partial(N)$,
$(f,F)\in\Map_\partial^{O_n}(N)$, while $H$ is a homotopy in topological isomorphisms of the tangent bundle $TN\xrightarrow{\pi}N$ viewed as a microbundle
\[
\xymatrix{
TN\times [0,1]\ar[r]^-H\ar[d]_{\pi\times id_{[0,1]}}&TN\times [0,1]\ar[d]^{\pi\times id_{[0,1]}}\\
N\times[0,1]\ar[r]^{f\times id_{[0,1]}}&N\times[0,1],
}
\]
such that $H|_{t=0}=F$ and $$H|_{t=1}=exp^{-1}\circ \Delta(f)\circ exp\colon TN\xrightarrow{exp}N\times N\xrightarrow{(f,f)}N\times N
\xrightarrow{exp^{-1}}TN.$$ More generally, a $k$-simplex of $\Homeo_\partial^{\underline{O_n}}(N)$ is a triple $(f_k,F_k,H_k)$ of  $\underline{O_n}$-framed homeomorphisms of~$N$ parametrized 
by $\Delta^k$: $f_k$ is a $k$-simplex of $\Homeo_\partial(N)$, $(f_k,F_k)$ is a $k$-simplex of 
$\Map_\partial^{O_n}(N)$, and $H_k$ is a germ of homeomorphisms of $TN\times \Delta^k\times [0,1]$ near $s(N)\times \Delta^k\times [0,1]$ (where $s\colon N\to TN$ is the zero section) extending $f_k\times id_{[0,1]}\colon s(N)\times\Delta^k\times [0,1]\to s(N)\times\Delta^k\times [0,1]$ and making the following square commute:
\[
\xymatrix{
TN\times\Delta^k \times[0,1]\ar[r]^-{H_k}\ar[d]_{\pi\times id_{\Delta^k\times [0,1]}}&TN\times\Delta^k\times [0,1]\ar[d]^{\pi\times id_{\Delta^k\times [0,1]}}\\
N\times\Delta^k\times[0,1]\ar[r]^{f_k\times id_{[0,1]}}&N\times\Delta^k\times[0,1].
}
\]
Moreover, $H_k$ is required to be the identity near $\pi^{-1}(\partial N)\times\Delta^k\times[0,1]$, and additionally $H_k|_{t=0}=F_k$, $H_k|_{t=1}=
exp^{-1}\circ \Delta(f_k)\circ exp$.

$\Homeo_\partial^{\underline{O_n}}(N)$ is a simplicial group with the product
\[
(f,F,H)*(f',F',H'):=(f\circ f',F\circ F',H\circ H').
\]
Moreover, one has a natural inclusion-homomorphism of simplicial groups:
\begin{equation}\label{eq:dif_fr_homeo}
\Diff_\partial(N)\to \Homeo_\partial^{\underline{O_n}}(N).
\end{equation}
We describe this inclusion at the level of vertices. Given a diffeomorphism $f\in\Diff_\partial(N)$, the induced map $exp^{-1}\circ \Delta(f)\circ exp
\colon TN\to TN$ of tangent bundles is usually different from the linear map $D(f)\colon TN\to TN$.
 (Note, however, that both are identity near $\pi^{-1}(\partial N)$.) Fortunately, there is a preferred path between $D(f)$ and $exp^{-1}\circ \Delta(f)\circ exp$.
 Denote by $K_x^f\colon T_xN\to T_{f(x)}N$ the map $K_x^f(y)=exp^{-1}_{f(x)}\left(f(exp_x y)\right)$. Then define $H^f\colon TN\times [0,1]\to
 TN\times[0,1]$ as follows:
 \[
 H^f(y,t):=
 \begin{cases}
 \left(\frac 1t K_{\pi(y)}^f(ty),t\right),& t\in(0,1];\\
 \left(D(f)(y),0\right),& t=0.
 \end{cases}
 \]

\begin{lemma}\label{l:dif_fr_homeo}
For any smooth manifold $N$ (with Riemannian metric), the map \eqref{eq:dif_fr_homeo} is a homomorphism of simplicial groups.
\end{lemma}

\begin{proof}
Again, we check it only at the level of vertices, i.e., 0-simplices. For $k$-simplices, $k>0$, the argument is just a $\Delta^k$-parametrized version of 
what we present below.

Let $f_1,f_2\in \Diff_\partial(N)$, we must show that 
\begin{equation}\label{eq:diff_compos}
H^{f_1}\left(H^{f_2}(y,t) \right)=H^{f_1\circ f_2}(y,t).
\end{equation}
Let $x:=\pi(y)$. One has $H^{f_2}(y,t)=\left(\frac 1t K_x^{f_2}(ty),t\right).$ Thus, the left-hand side of~\eqref{eq:diff_compos} is
\begin{multline*}
\left(\frac 1tK^{f_1}_{f_2(x)}\left(t\cdot\frac 1t K_x^{f_2}(ty)\right),t\right)=
\left(\frac 1tK^{f_1}_{f_2(x)}\left( K_x^{f_2}(ty)\right),t\right)=\\
\left(\frac 1t exp^{-1}_{f_1(f_2(x))}\left(f_1\left(exp_{f_2(x)}\left( exp^{-1}_{f_2(x)}\left(f_2(exp_xy)\right)\right)\right)\right),t\right)=\\
\left(\frac 1t exp^{-1}_{f_1(f_2(x))}\left(f_1\left(f_2(exp_xy)\right)\right),t\right)=
\left(\frac 1t K_x^{f_1\circ f_2}(ty),t\right)=
H^{f_1\circ f_2}(y,t).
\end{multline*}
\end{proof}

\begin{theorem}\label{th:diff_compos}
For any smooth manifold $N$ of dimension $n\neq 4$, the map $\Diff_\partial(N)\to \Homeo_\partial^{\underline{O_n}}(N)$  is an equivalence of simplicial groups.
\end{theorem}

\begin{proof}
We define an auxiliary space $\Map^{O_n\times^h_{\TOP_n}\TOP_n}_\partial(N)$ as a simplicial set whose vertices are pairs $(f,\alpha)$, where $f\in\Map_\partial(N)$ and $\alpha\colon TN\times[0,1]\to TN$ is a
 path of topological fiberwise isomorphisms lifting $f$ and such that $\alpha|_{t=0}$ is  linear,  higher simplices being defined similarly. This space is  equivalent to $\dMap^{O_n}_\partial(N)$. 
One gets the following commutative diagram, where the right vertical map is  obtained by the restriction $\alpha\mapsto\alpha|_{t=1}$.
\begin{equation}\label{eq:diff_xxx}
\xymatrix{
\Diff_\partial(N)\ar[d]\ar[r]&\Homeo_\partial^{\underline{O_n}}(N)\ar@{->>}[d]\ar[r]&\Map^{O_n\times^h_{\TOP_n}\TOP_n}_\partial(N)\ar@{->>}[d]\\
\Homeo_\partial(N)\ar@{=}[r]&\Homeo_\partial(N)\ar[r]&\Map_\partial^{\TOP_n}(N).
}
\end{equation}
 The right square in it is a pullback and, hence, a homotopy pullback since the right vertical map is a Kan fibration. The outer square is a homotopy pullback by Theorem~\ref{th:square_diff},
 see Remark~\ref{r:sq_non-comm}. As a consequence, the left top horizontal arrow is an equivalence.
\end{proof}

\subsection{PL version}\label{ss:pl_fr_homeo}
The PL version of this construction that we present below has a much more limited generality. One of the reasons for
this limited generality is outlined in Remark~\ref{r:exp}. Another  problem is
that one does not have an inclusion of $\OO_n$ in $\PL_n$. Moreover, unlike $\OO_n$, the group $\PL_n$ is not determined by its action on $\R^n$ in the sense 
that a principal $\PL_n$ bundle cannot be derived from some kind of $\R^n$-bundle.

However, for the case when $N\subset\R^n$ is a compact codimension zero submanifold of $\R^n$ with smooth boundary, the construction goes through with little change.  Define the space  $\pdHomeo_\partial^{\underline{O_n}}(N)$ of ${\underline{O_n}}$-framed relative
to the boundary
PD homeomorphisms of $N$ as the simplicial set of triples $(f_k,F_k,H_k)$, where $f_k$ is a $k$-simplex of $\pdHomeo_\partial(N)$,
$(f_k,F_k)$ is a $k$-simplex of $\pdMap_\partial^{O_n}(N)$, while $H_k$ is a germ of PD homeomorphisms 
$$
H_k\colon \left(TN=N\times\R^n\right)\times\Delta^k\times[0,1]
\to N\times \R^n\times\Delta^k\times[0,1]
$$
 near $N\times 0\times\Delta^k\times[0,1]$
 (on which they are identity), that commute with the projection on $N\times\Delta^k\times[0,1]$, are the identity near $N\times\R^n\times \Delta^k\times [0,1]$ and such that $H_k|_{t=0}=F_k$ while $H_k|_{t=1}=exp^{-1}\circ \Delta(f_k)\circ exp$. Since $exp\colon (x,y)\mapsto (x,x+y)$ is both smooth and PL,
 the boundary condition for $H_k|_{t=1}$ is well-defined.

The formulae that define the map \eqref{eq:dif_fr_homeo} work as well to get a map 
\begin{equation}\label{eq:diff_pd_homeo}
\Diff_\partial(N)\to\pdHomeo_{\partial}^{\underline{O_n}}(N).
\end{equation}
For  $f\in\Diff_\partial(N)$,  the path $H\colon N\times\R^n\times[0,1]\to N\times\R^n\times[0,1]$ is as follows:
\[
H(x,y,t)=
\begin{cases}
(f(x),\frac{f(x+ty)-f(x)}{t},t),&t\in(0,1],\\
(f(x),D(f)_x(y),t),&t=0.
\end{cases}
\]
Since $H$ is smooth, it is PD, and the map~\eqref{eq:diff_pd_homeo}  is well-defined.

\begin{proposition}\label{p:dif_homeo_pd}
For any $n\geq 0$ and any codimension zero compact submanifold $N\subset\R^n$ with smooth boundary $\partial N$, the map $\Diff_\partial(N)\to\pdHomeo_{\partial}^{\underline{O_n}}(N)$ is an equivalence of simplicial sets for any $n\geq 1$.
\end{proposition}

\begin{proof}
Same as for Theorem~\ref{th:diff_compos}.
\end{proof}

\section{Proof of Theorem~\ref{th:a}}\label{s:proof_a}
\subsection{TOP version}\label{ss:a_top}
We will prove a refinement of Theorem~\ref{th:a} that takes into account the $\OO_n$-action. 


\begin{theorem}\label{th:a_top}
For $n\neq 4$, one has an equivalence of $E_{n+1}^{\OO_n}$-algebras
\begin{equation}\label{eq:a_top_1}
\Diffn\simeq_{E_{n+1}^{\OO_n}}\Omega^{n+1}\left(\TOP_n/\OO_n\right);
\end{equation}
\begin{equation}\label{eq:a_top_2}
\Dbarn\simeq_{E_{n+1}^{\OO_n}}\Omega^{n+1}\TOP_n.
\end{equation}
\end{theorem}

\begin{proof}[Proof of~\eqref{eq:a_top_1}]
Consider the following zigzag
\[
\Diffn\xrightarrow[I]{\,\simeq\,}\Homeo_\partial^{\underline{O_n}}(D^n)\xleftarrow[J]{\,\simeq\,}\Omega^n\hofiber(O_n\to\TOP_n).
\]
The first map is an equivalence by Theorem~\ref{th:diff_compos}. One can see that it respects the $E_n$-action. By Lemma~\ref{l:dif_fr_homeo} it  also respects the associative product, implying that $I$ is an equivalence of 
$R_{n+1}$-algebras. It is also not hard to see that $I$ respects the $\OO_{n}$-action, where the $\OO_{n}$-action on $\Diffn$ is described in Subsection~\ref{ss:fr_operads}, while the $\OO_n$-action on 
$\Homeo_\partial^{\underline{O_n}}(D^n)$ is defined in the same manner as for $\Dbarn$,
see Subsection~\ref{ss:modulo}. The map $J$ is the fiber inclusion  for the projection
$\Homeo_\partial^{\underline{O_n}}(D^n)\to \Homeo_\partial(D^n)$, the fiber 
 being taken over the
identity homeomorphism $id\in \Homeo_\partial(D^n)$. Since the base space $\Homeo_\partial(D^n)$ is contractible by Lemma~\ref{l:alex}, the fiber inclusion~$J$ is an equivalence. On the other hand, $\{id\}$
is an $R_{n+1}^{\OO_n}$-subalgebra of $\Homeo_\partial(D^n)$ and the projection is a morphism
of $R_{n+1}^{\OO_n}$-algebras. Thus, $J$ is also an inclusion of $R_{n+1}^{\OO_n}$-algebras.

To get the last step, i.e., equivalence to $\Omega^{n+1}\left(\TOP_n/ \OO_n\right)$, we pass from
simplicial sets to topological spaces. 
 We therefore first replace $\Omega^n\hofiber\left(O_n\to\TOP_n\right)$
by $\Omega^n\hofiber\left(|\OO_n|\to|\TOP_n|\right)$, see Subsection~\ref{sss:loops}. The new space
is an $R_{n+1}^{|\OO_n|}$-algebra ($|\OO_n|$ acts on $R_{n+1}$ by restriction along the group homomorphism $|\OO_n|\to \OO_n$). On the other hand, $\hofiber\left(|\OO_n|\to|\TOP_n|\right)$ is an $\Assoc^{|\OO_n|}$-algebra
$E_1^{|\OO_n|}$-equivalent to 
$\Omega\left(|\TOP_n|/|\OO_n|\right)=\Omega\left|\TOP_n/\OO_n\right|$ by Proposition~\ref{l:E1equiv} (where $|\OO_n|$ acts trivially on $\Assoc$ and $E_1$). We conclude that 
 $\Omega^n\hofiber\left(|\OO_n|\to|\TOP_n|\right)$ is $E_{n+1}^{|\OO_n|}$-equivalent to 
$\Omega^{n+1}\left|\TOP_n/\OO_n\right|$.
\end{proof}

\begin{proof}[Proof of~\eqref{eq:a_top_2}]
It follows from Theorem~\ref{th:square_diff}(a) that for any smooth $n$-manifold $N$, $n\neq 4$, and any $f\in\Diff_\partial(N)$, the homotopy fibers over the image of $f$ are equivalent:
\[
\hofiber_f\left(\Diff_\partial(N)\to\Map_\partial^{O_n}(N)\right)
\simeq
\hofiber_f\left(\Homeo_\partial(N)\to\Map_\partial^{\TOP_n}(N)\right).
\eqno(\numb)\label{eq:dbar_hombar}
\]
Define $\overline{\Homeo}_\partial(D^n):=
\hofiber\left(\Homeon\to\Omega^n\TOP_n\right)$.\footnote{As a simplicial set (and as an $R^{\OO_n}_{n+1}$-algebra)   it is defined similarly to $\Dbarn$, see Section~\ref{ss:modulo}.}
We get a zigzag of $R_{n+1}^{\OO_n}$-equivalences:
\[
\Dbarn\xrightarrow[I]{\simeq}\overline{\Homeo}_\partial(D^n)\xleftarrow[J]{\simeq}\Omega^n (\Omega\TOP_n).
\]
Here $I$ is a particular case of the equivalence \eqref{eq:dbar_hombar}. On the level of vertices it sends $(f,H)\in\Dbarn$ to $(f,\hat H^f)\in
\overline{\Homeo}_\partial(D^n)$,
where 
\[
\hat H^f(t)=\begin{cases}
H^f(1-2t),& t\in[0,\frac 12];\\
H(2t-1),&t\in[\frac 12,1].
\end{cases}
\]
where $H^f$ is defined in Section~\ref{s:fr_homeo}. The map $J$ is the inclusion of the fiber over the identity with respect to the projection $\overline{\Homeo}_\partial(D^n)\to\Homeon$. 
\end{proof}

\subsection{PL version}\label{ss:a_pl}
\begin{proof}[Proof of $\Diffn\simeq\Omega^{n+1}(\PL_n/\OO_n)$]
We show first that we have an equivalence of spaces. 
Recall that $\PL_n \simeq \PD_n$. 
Consider the zigzag
\[
\Diffn\xrightarrow{\,\,\simeq\,\,}\pdHomeo_\partial^{\underline{O_n}}(D^n)\xleftarrow{\,\,\simeq\,\,}\Omega^n\hofiber\left(O_n\to\PD_n\right).
\]
The first map is an equivalence by Proposition~\ref{p:dif_homeo_pd}. The second map is the inclusion of the fiber over $id\in\pdHomeo_\partial(D^n)$
and is an equivalence by Lemma~\ref{l:alex} and the fact that the inclusion $\plHomeo_\partial(D^n) \hookrightarrow\pdHomeo_\partial(D^n)$
is an equivalence~\cite{BLR_aut}. 

We then pass to the topological category and replace $\Omega^n\hofiber\left(\OO_n\to\PD_n\right)$ by $\Omega^n\hofiber\left(|\OO_n|\to|\PD_n|\right)$,
see Subsection~\ref{sss:loops}. The last step is to apply the equivalence
\[
\hofiber\left(|\OO_n|\to|\PD_n|\right)\xrightarrow{\,\,\simeq\,\,} \Omega(|\PD_n|/|\OO_n|)=\Omega|\PD_n/\OO_n|.
\]

Note, however, that most of the spaces in our zigzag are not even $E_1$-algebras, see Subsections~\ref{ss:little_discs} and~\ref{ss:loops}.
To see that $\Diffn\simeq_{E_{n+1}}\Omega^{n+1}(\PD_n/\OO_n)$ for $n\neq 4$, we notice that the natural map $\PD_n\to\TOP_n$
induces an equivalence 
\begin{equation}\label{eq:pd_to_top}
\Omega^{n+1}|\PD_n/\OO_n|\to \Omega^{n+1}|\TOP_n/\OO_n|.
\end{equation}
 Indeed, one has a natural morphism from the 
zigzag of this subsection that proves $\Diffn\simeq\Omega^{n+1}(\PD_n/\OO_n)$ to the zigzag of Subsection~\ref{ss:a_top} which proves
$\Diffn\simeq\Omega^{n+1}(\TOP_n/\OO_n)$. Arguing inductively, we obtain that the natural maps between the terms of the zigzags are all equivalences. In particular, it concerns the last terms, implying that the map~\eqref{eq:pd_to_top} is also an equivalence. On the other hand, $\Diffn\simeq_{E_{n+1}}\Omega^{n+1}(\TOP_n/\OO_n)$,
which implies the same in the PD version.
\end{proof}

\begin{proof}[Proof of $\Dbarn\simeq\Omega^{n+1}\PL_n$]
By Theorem~\ref{th:square_diff}(b),
 for any smooth compact $n$-manifold $N$ and for any $f\in\Diff_\partial(N)$, the homotopy fibers over the images of~$f$ are equivalent
\[
\hofiber_f\left(\Diff_\partial(N)\to\dMap_\partial^{O_n}(N)\right)
\xrightarrow{\,\,\simeq\,\,}
\hofiber_f\left(\pdHomeo_\partial(N)\to\pdMap_\partial^{\PD_n}(N)\right).
\eqno(\numb)\label{eq:dbar_pd_hombar}
\]

Let $\mathrm{pd}\Hombar_\partial(D^n)$ be the homotopy fiber
\[
\mathrm{pd}\Hombar_\partial(D^n):=\hofiber\left(\pdHomeo_\partial(D^n)\to\Omega^n\PD_n\right)
\]
over the constant loop at the basepoint. 
 One gets the following zigzag of equivalences:
\[
\Dbarn\xrightarrow{\,\,\simeq\,\,}\mathrm{pd}\Hombar_\partial(D^n)\xleftarrow{\,\,\simeq\,\,}\Omega^n\Omega\PD_n\xleftarrow{\,\,\simeq\,\,}\Omega^{n+1}\PL_n.
\]
We use here \eqref{eq:dbar_pd_hombar}, the fact that $\dMap_\partial(D^n)\simeq\pdMap_\partial(D^n)\simeq\Map_\partial(D^n)$, by the approximation argument, and that $\Map_\partial(D^n)\simeq *$, by the Alexander trick Lemma~\ref{l:alex}.
Once again we remind that the second and third spaces in this zigzag are just simplicial sets, not even $E_1$-algebras, while the last one is only a simplicial group. To see that $\Dbarn\simeq_{E_{n+1}}
\Omega^{n+1}|\PL_n|$, $n\neq 4$, we  argue in the same way as for the equivalence $\Diffn\simeq_{E_{n+1}}\Omega^{n+1}|\PD_n/\OO_n|$, by comparing the PD zigzag to the analogous topological one.
\end{proof}

\section{$\underline{V_{n,m}}$- and $\underline{O_n}$-framed topological embeddings and immersions}\label{s:fr_embed}
In this section we extend the previous machinery to study embeddings. 

\subsection{$\underline{V_{n,m}}$-framed topological embeddings and immersions}\label{ss:Vnm}
Let $(M,\partial_0M,\partial_1M)$ and $N$ be as in Section~\ref{ss:square2}. 
Define the space $\tEmb_{\partial_0M}^{\underline{V_{n,m}}}(M,N)$ of $\underline{V_{n,m}}$-framed locally flat topological embeddings as the space of triples
$(f,F,H)$, where $f$ is a topological locally flat embedding $f\in \tEmb_{\partial_0M}(M,N)$, $F$ is a monomorphism of tangent vector bundles $F\colon TM\to TN$
lifting $f$ (in other words, $(f,F)\in \Map_{\partial_0M}^{V_{n,m}}(M,N)$), and $H$ is an isotopy of topological monomorphisms of tangent bundles $H\colon TM\times[0,1]\to TN$, such that $H|_{t=0}=F$ and $H|_{t=1}=exp^{-1}_N\circ\Delta(f)\circ exp_M$. Here, $H$ can also be seen as a germ of topological locally flat embeddings 
$H\colon TM\times [0,1]\hookrightarrow TN\times [0,1]$ near $s(M)\times [0,1]$ and extending $f\times id_{[0,1]}$, commuting with the projection on $[0,1]$ 
and coinciding with the inclusion $D(i)\times id_{[0,1]}$ near $\pi^{-1}\partial_0 M$.
A $k$-simplex of $\tEmb_{\partial_0M}^{\underline{V_{n,m}}}(M,N)$ is a triple $(f_k,F_k,H_k)$, where $f_k$ is a $k$-simplex of $\tEmb_{\partial_0M}(M,N)$,
$(f_k,F_k)$ is a $k$-simplex of $\Map_{\partial_0M}^{V_{n,m}}(M,N)$, and $H_k$ is a germ of locally flat embeddings $H_k\colon TM\times\Delta^k\times [0,1]
\to TN\times\Delta^k\times [0,1]$ near $s(M)\times\Delta^k\times[0,1]$, making the square
\[
\xymatrix{
 TM\times\Delta^k\times [0,1]\ar[rr]^-{H_k}\ar[d]_{\pi\times id_{\Delta^k\times[0,1]}} & & TN\times\Delta^k\times [0,1]\ar[d]_{\pi\times id_{\Delta^k\times[0,1]}}\\ 
M\times\Delta^k\times[0,1]\ar[rr]_-{f_k\times id_{[0,1]}}&&N\times\Delta^k\times[0,1].
}
\]
commute and coinciding with $D(i)\times id_{\Delta^k\times [0,1]}$ near $\pi^{-1}(\partial_0M)\times \Delta^k\times [0,1]$.

We similarly define the space $\tImm_{\partial_0M}^{\underline{V_{n,m}}}(M,N)$ of $\underline{V_{n,m}}$-framed locally flat topological immersions as the space of similar triples
$(f,F,H)$ with the difference that $f$ is a topological locally flat immersion $f\in \tImm_{\partial_0M}(M,N)$.

One has natural maps 
\begin{gather}
 \Emb_{\partial_0M}(M,N)\to
 \tEmb_{\partial_0M}^{\underline{V_{n,m}}}(M,N); \label{eq:emb_fr_temb}\\
 \Imm_{\partial_0M}(M,N)\to
 \tImm_{\partial_0M}^{\underline{V_{n,m}}}(M,N) \label{eq:imm_fr_imm}
\end{gather}
defined in the same way as the map \eqref{eq:dif_fr_homeo}.  As an immediate consequence of the Smale-Hirsch theory, see Section~\ref{s:imm}, 
the second map is an equivalence. 
 The same can be said about the first map.

\begin{theorem}\label{th:emb_fr_temb}
For $M$ and $N$ as above, with $m\leq n\neq 4$, the map~\eqref{eq:emb_fr_temb} is an equivalence.
\end{theorem}

\begin{proof}
One has a natural morphism from square~\eqref{eq:square_kupers} to a similar square
\begin{equation}\label{eq:square_kupers2}
\xymatrix{
 \tEmb_{\partial_0M}^{\underline{V_{n,m}}}(M,N)\ar[r]\ar@{->>}[d]&  \tImm_{\partial_0M}^{\underline{V_{n,m}}}(M,N)\ar@{->>}[d]\\
\tEmb_{\partial_0 M}(M,N)\ar[r]& \tImm_{\partial_0 M}(M,N).
}
\end{equation}
The former one is homotopy cartesian by Theorem~\ref{th:square_emb}(a), while the latter
is cartesian and, therefore, also homotopy cartesian,  its right vertical map being a Kan fibration. Since~\eqref{eq:imm_fr_imm} is an equivalence, we conclude 
that so is~\eqref{eq:emb_fr_temb}. 
\end{proof}

\subsection{$\underline{O_n}$-framed topological embeddings}\label{ss:On_fr}
From now on assume that $M\subset N$, $m<n$, and $\partial M=\partial_0 M$. The space $\Emb_{\partial }^{fr}(M,N)$ of  framed embeddings relative to the boundary is the space of pairs $(f,F)$, where $f\in\Emb_{\partial }(M,N)$ and $F$ is a (smooth) fiberwise isomorphism of vector bundles $F\colon TN|_{M}\to TN$  lifting $f$:
\begin{equation}\label{eq:fr_emb}
\xymatrix{
TN|_{M}\ar[d]\ar[r]^{F}&TN\ar[d]\\
M\ar@{^{(}->}[r]^{f}& N,
}
\end{equation}
and such that  $F|_{TM}=D(f)$ and $F$ is the identity inclusion near the preimage of $\partial M$. 
The $k$-simplices $(f_k,F_k)$ of $\Emb_{\partial}^{fr}(M,N)$  are defined similarly: $f_k$ must be a $k$-simplex of  
$\Emb_{\partial}(M,N)$, while $F_k\colon TN|_{M}\times\Delta^k\to TN\times\Delta^k$ must be a fiberwise isomorphism of vector bundles
lifting $f_k\colon M\times\Delta^k\hookrightarrow N\times\Delta^k$ and coinciding with $D(f_k)$ on $TM\times \Delta^k$ and with the identity inclusion near the preimage of
 $\partial M\times\Delta^k$.

Define the space $\tEmb_{\partial}^{\underline{O_n}}(M,N)$ of $\underline{O_n}$-framed relative to $\partial M$ topological locally flat embeddings
$M\hookrightarrow N$ as a simplicial set whose vertices are triples $(f,F,H)$, where $f\in\tEmb_{\partial}(M,N)$, $F$ is a continuous fiberwise isomorphism of vector bundles~\eqref{eq:fr_emb} lifting~$f$, and $H$ is an isotopy of locally flat monomorphisms of topological microbundles
\begin{equation}\label{eq:isot_tmono}
\xymatrix{
TM\times [0,1]\ar[d]_{\pi_M\times id_{[0,1]}}\ar[rr]^{H}&&TN\times[0,1]\ar[d]^{\pi_N\times id_{[0,1]}}\\
M\times [0,1]\ar@{^{(}->}[rr]^{f\times id_{[0,1]}}&& N\times [0,1],
}
\end{equation}
equal to $D(i)\times id_{[0,1]}$ near $\pi_M^{-1}(\partial M)\times [0,1]$ and such that $H|_{t=0}=F|_{TM}$ and $H|_{t=1}=
exp_N^{-1}\circ \Delta(f)\circ exp_M$. (In other words, $(f,F|_{TM},H)\in \tEmb^{\underline{V_{n,m}}}_{\partial M}(M,N)$.) Its $k$-simplices are defined similarly.

Note that $\Emb_{\partial}^{fr}(M,N)$ and $\tEmb_{\partial}^{\underline{O_n}}(M,N)$ fit into the pullback squares:
\begin{equation}\label{eq:pb_squares}
\xymatrix{
\Emb_{\partial }^{fr}(M,N)\ar[r] \ar@{->>}[d]& \Map_{\partial }^{O_n}(M,N)\ar@{->>}[d]\\
\Emb_{\partial }(M,N)\ar[r]&  \Map_{\partial }^{V_{n,m}}(M,N);
}
\quad\quad\quad
\xymatrix{
\tEmb_{\partial }^{\underline{O_n}}(M,N)\ar[r] \ar@{->>}[d]& \Map_{\partial }^{O_n}(M,N)\ar@{->>}[d]\\
\tEmb_{\partial }^{\underline{V_{n,m}}}(M,N)\ar[r]&  \Map_{\partial }^{V_{n,m}}(M,N).
}
\end{equation}


\begin{theorem}\label{th:emb_fr_temb2}
For $m<n\neq 4$,  the natural inclusion
\[
\Emb_{\partial }^{fr}(M,N)\xrightarrow{\,\,\simeq\,\,} \tEmb_{\partial }^{\underline{O_{n}}}(M,N)
\]
is an equivalence.
\end{theorem}
\begin{proof}
Follows from Theorem~\ref{th:emb_fr_temb} and the fact that the squares~\eqref{eq:pb_squares} are homotopy pullback squares.
\end{proof}

\subsection{Examples}\label{ss:examples}
\subsubsection{$\tEmb_\partial^{\underline{V_{n,m}}}(D^m,D^n)^{\times}$}\label{sss:embmn_tembmn}
Let $1\leq m< n\neq 4$. Consider the subspace $\Emb_\partial(D^m,D^n)^\times\subset \Embmn$ consisting of path-components 
of invertible elements in $\pi_0$. In case $n-m\neq 2$, $\Emb_\partial(D^m,D^n)^{\times}= \Embmn$, see the Introduction.
These path-components consist of smooth knots which are topologically trivial. Let $\tEmb_\partial^{\underline{V_{n,m}}}(D^m,D^n)^{\times} \subset \tEmb_\partial^{\underline{V_{n,m}}}(D^m,D^n)$ 
be the similarly defined  union of components. 
This subspace can be identified as the preimage of the component $\tEmbmn_*$ of the trivial knot.
\footnote{We use here the fact $\tEmb_\partial(D^m,D^n)^\times =\tEmb_\partial(D^m,D^n)_*$:  if a codimension two knot ($n=m+2$) is invertible then its complement is a homotopy $S^1$, see footnote~\ref{foot6}, while a codimension two topological knot 
 is trivial if and only if its complement is a homotopy $S^1$ \cite{Stallings}.}  One has again that for $n-m\neq 2$, $\tEmb_\partial^{\underline{V_{n,m}}}(D^m,D^n)^\times = \tEmb_\partial^{\underline{V_{n,m}}}(D^m,D^n)$. Applying Theorem~\ref{th:emb_fr_temb},  
 we get that the natural map
\begin{equation}\label{eq:embmn_tembmn}
\Emb_\partial(D^m,D^n)^{\times} \xrightarrow{\,\,\simeq\,\,} \tEmb_\partial^{\underline{V_{n,m}}}(D^m,D^n)^{\times},\,\, n\neq 4,
\end{equation}
is an equivalence.

\subsubsection{$\tEmb_{\partial D^m\times D^{n-m}}^{\underline{O_n}}(D^m\times D^{n-m})^{\times}$}\label{sss:emb_dm_dn}
Let again $1\leq m<n\neq 4$.
Consider the subspaces 
$$\Emb_{\partial D^m\times D^{n-m}}(D^m\times D^{n-m})^{\times}\subset \Emb_{\partial D^m\times D^{n-m}}(D^m\times D^{n-m})$$  and  
$$\tEmb_{\partial D^m\times D^{n-m}}^{\underline{O_n}}(D^m\times D^{n-m})^{\times}
\subset \tEmb_{\partial D^m\times D^{n-m}}^{\underline{O_n}}(D^m\times D^{n-m})$$
 of path-components of homotopy invertible elements. For $n-m\neq 2$,
these inclusions are identities. Equivalently these subspaces can be described as preimages of the component
$\tEmb_\partial(D^m\times 0,D^m\times D^{n-m})_*$ of the trivial knot $*: D^m\times 0\subset D^m\times D^{n-m}$, see~\ref{sss:embmn_tembmn}. 
 Theorem~\ref{th:emb_fr_temb} implies that  the natural map
\begin{equation}\label{eq:emb_dm_dn}
\Emb_{\partial D^m\times D^{n-m}}(D^m\times D^{n-m})^{\times}\xrightarrow{\,\,\simeq\,\,} 
\tEmb_{\partial D^m\times D^{n-m}}^{\underline{O_n}}(D^m\times D^{n-m})^{\times},\,\, n\neq 4,
\end{equation}
is an equivalence.

%
%

\subsubsection{$\tEmb^{\underline{O_n}}(S^m,S^n)^\times$}\label{sss:spherical_fr_emb}
For $1\leq m\leq n\neq 4$,
define $\Emb^{fr}(S^m,S^n)^\times$ as the union of those path-components in $\Emb^{fr}(S^m,S^n)$ that lie in the preimage of $\Emb^{fr}_\partial(D^m,D^n)^\times$ under the identification $\Emb_\partial^{fr}(D^m,D^n)\simeq \Emb^{fr}(S^m,S^n)/\OO_{n+1}$. The subspace $\tEmb^{\underline{O_n}}(S^m,S^n)^\times \subset \tEmb^{\underline{O_n}}(S^m,S^n)$ is defined similarly as the union of path-components sent to $\tEmb_\partial^{\underline{O_n}}(D^m,D^n)^\times\subset
\tEmb_\partial^{\underline{O_n}}(D^m,D^n)$ under the identification $\tEmb_\partial^{\underline{O_n}}(D^m,D^n)\simeq \tEmb^{\underline{O_n}}(S^m,S^n)/\OO_{n+1}$.\footnote{The space $\tEmb^{\underline{O_n}}(S^m,S^n)/\OO_{n+1}$
fibers over $\VV_{n,m}^t{\times}^h_{\VV_{n,m}^t}\!\!\!\{*\}$ (which is contractible) with fiber $\tEmb_\partial^{\underline{O_n}}(D^m,D^n)$ implying the aforementioned equivalence.}
 For $n-m\neq 2$, $n\neq 4$, these inclusions are identities. 

The natural inclusion
\begin{equation}\label{eq:spherical_fr_emb}
\Emb^{fr}(S^m,S^n)^\times\xrightarrow{\,\,\simeq\,\,}\tEmb^{\underline{O_n}}(S^m,S^n)^\times,\,\, n\neq 4,
\end{equation}
is an equivalence by Theorems~\ref{th:diff_compos} ($n=m$) and~\ref{th:emb_fr_temb2} ($n>m$).

\subsection{PL version}\label{ss:PL_fr_temb}
The PL version of this construction works much less generally. (Or at least we can not see how it can be done in full generality.) Here are a few reasons for it.

\begin{itemize}
\item To relate smooth and PL embeddings, one has to use the intermediate spaces of PD embeddings, which are used in the literature only 
in the case of codimension zero \cite{BL_diff}.

\item One has to identify the vector tangent bundles of the source and target manifolds $M$ and $N$ with the PL tangent microbundles of $|M|$ and $|N|$, respectively. Moreover, such identification should respect  the PL structure of the fibers induced by their vector space structure, see Remark~\ref{r:exp}.
\end{itemize}

By Theorem~\ref{th:square_emb} the square~\eqref{eq:square_kupers3} is homotopy cartesian when $n=m$. So, the problem is really to define the appropriate  
$\underline{O_n}$-framing for codimension zero PD embeddings.
We do not explore the full generality of this construction, but instead, consider the following only example that we need. 

Let $n>m\geq 1$. (In particular, we allow $n=4$.) Define the space $\pdEmb_{\partial D^m\times D^{n-m}}^{\underline{O_n}}(D^m\times
D^{n-m})$ of $\underline{O_n}$-framed PD self-embeddings of $D^m\times D^{n-m}$ as a simplicial set whose $k$-simplices
are triples $(f_k,F_k,H_k)$, where $f_k$ is a $k$-simplex of $\pdEmb_{\partial D^m\times D^{n-m}}(D^m\times
D^{n-m})$, $F_k$ is such that $(f_k,F_k)$ is a $k$-simplex of $\pdMap_{\partial D^m\times D^{n-m}}^{O_n}(D^m\times
D^{n-m})$ ($F_k$ can also be seen as a PD map $D^m\times D^{n-m}\times \Delta^k\to O_n$, which is equal to the constant identity near $\partial D^m\times D^{n-m}\times \Delta^k$), and $H_k$ is a germ of PD homeomorphisms $H_k\colon D^m\times D^{n-m}\times \Delta^k\times [0,1]\times \R^n
\to D^m\times D^{n-m}\times \Delta^k\times [0,1]\times \R^n$ near $D^m\times D^{n-m}\times \Delta^k\times [0,1]\times 0$ (on which it is  $f_k\times 0$), which is identity near 
$\partial D^m\times D^{n-m}\times \Delta^k\times [0,1]\times \R^n$ and such that $H_k|_{t=0}=F_k$, $H_k|_{t=1}=exp^{-1}\circ \Delta(f_k)\circ exp$ and the following square commutes:
\[
\xymatrix{
D^m\times D^{n-m}\times\Delta^k\times [0,1]\times\R^n\ar[r]^-{H_k}\ar[d]&D^m\times D^{n-m}\times\Delta^k\times [0,1]\times\R^n\ar[d]\\
D^m\times D^{n-m}\times\Delta^k\times [0,1]\ar[r]^{f_k\times id_{[0,1]}}&D^m\times D^{n-m}\times\Delta^k\times [0,1].
}
\]
To recall the exponential map $exp\colon T\R^n=\R^n\times\R^n\to\R^n\times\R^n$ sends $(x,y)\mapsto (x,x+y)$. (We apply it to $T\left(D^m\times D^{n-m}\right)\subset T\R^n$.)

We consider the subspace 
$$
\pdEmb_{\partial D^m\times D^{n-m}}^{\underline{O_n}}(D^m\times
D^{n-m})^\star \subset\pdEmb_{\partial D^m\times D^{n-m}}^{\underline{O_n}}(D^m\times
D^{n-m}),
$$
 which consists of those path-components that arise from the knots
 $
\pdEmb_{\partial D^m\times D^{n-m}}(D^m\times
D^{n-m})\simeq \plEmb_{\partial D^m\times D^{n-m}}(D^m\times
D^{n-m})$, that restrict to the trivial PL knot $D^m\times 0\hookrightarrow D^m\times D^{n-m}$. When $n-m\neq 2$ and $(n,m)\neq (4,3)$,
$$
\pdEmb_{\partial D^m\times D^{n-m}}^{\underline{O_n}}(D^m\times
D^{n-m})^\star =\pdEmb_{\partial D^m\times D^{n-m}}^{\underline{O_n}}(D^m\times
D^{n-m}).
$$
Moreover, $\pdEmb_{\partial D^m\times D^{2}}(D^m\times D^{2})^\star=\pdEmb_{\partial D^m\times D^{2}}(D^m\times D^{2})^\times$, $m\neq 2$. (Recall that
a smooth codimension two knot $D^m\hookrightarrow D^{m+2}$, except possibly in the ambient dimension $m+2=4$,  has a homotopy $S^1$ complement if and only if it is a reparameterization of the trivial one, see Subsection~\ref{ss:cod2}. The same is true in the PL category. I.e., a PL locally flat knot $D^m\hookrightarrow D^{m+2}$ is trivial if and only if its complement is a homotopy $S^1$, $m\neq 2$, \cite{Levine,Papa,Shane,Wall_spheres,Wall_surgery}, \cite[Theorem~1]{Wall_codim2}.)   

We similarly consider the subspace 
\[
\Emb_{\partial D^m\times D^{n-m}}(D^m\times D^{n-m})^\star\subset \Emb_{\partial D^m\times D^{n-m}}(D^m\times D^{n-m})
\]
of path-components of knots that restrict to knots $D^m\times 0\hookrightarrow D^m\times D^{n-m}$, which are PL trivial. One has again, that 
this inclusion is identity, when $n-m\neq 2$ and $(n,m)\neq (4,3)$. Moreover for $n-m=2$ and $n\neq 4$, such path-components correspond to the invertible elements in $\pi_0$. In other words,
\[
\Emb_{\partial D^m\times D^{2}}(D^m\times D^{2})^\star
=
\Emb_{\partial D^m\times D^{2}}(D^m\times D^{2})^\times,\,\, m\neq 2.
\]

Similarly to the map~\eqref{eq:diff_pd_homeo}, one defines a map
\begin{equation}\label{eq:emb_fr_temb3}
\Emb_{\partial D^m\times D^{n-m}}(D^m\times D^{n-m})\xrightarrow{\,\,\simeq\,\,} \pdEmb_{\partial D^m\times D^{n-m}}^{\underline{O_n}}(D^m\times
D^{n-m}),
\end{equation}
which restricts to
\begin{equation}\label{eq:emb_fr_temb4}
\Emb_{\partial D^m\times D^{n-m}}(D^m\times D^{n-m})^\star\xrightarrow{\,\,\simeq\,\,} \pdEmb_{\partial D^m\times D^{n-m}}^{\underline{O_n}}(D^m\times
D^{n-m})^\star.
\end{equation}
By the same argument as in the proof of Theorem~\ref{th:emb_fr_temb}, the homotopy cartesianity of~\eqref{eq:square_kupers3} implies that~\eqref{eq:emb_fr_temb3} is an equivalence. Thus,~\eqref{eq:emb_fr_temb4}  
is also one. 

\section{Proof of Theorem~\ref{th:b}. TOP part}\label{s:proof_B_TOP}
We will prove the following refinement of Theorem~\ref{th:b}.

\begin{theorem}\label{th:b_top_refine}
For $1\leq m\leq n\neq 4$, $n-m\neq 2$, one has the following equivalences of algebras over $E_m^{\OO_m\times\OO_{n-m}}$ and
$E_{m+1}^{\OO_m\times\OO_{n-m}}$:
\begin{equation}\label{eq:b_top_refine1}
\Emb_\partial(D^m,D^n)\simeq _{E_m^{\OO_m\times\OO_{n-m}}}\Omega^m\hofiber\left(\VV_{n,m}\to\VV_{n,m}^t\right);
\end{equation}
\begin{equation}\label{eq:b_top_refine2}
\Emb^{fr}_\partial(D^m,D^n)\simeq _{E_{m+1}^{\OO_m\times\OO_{n-m}}}\Omega^{m+1}\left(\OO_n\bbslash\TOP_n/
\TOP_{n,m}\right);
\end{equation}
\begin{equation}\label{eq:b_top_refine3}
\Ebar_\partial(D^m,D^n)\simeq _{E_{m+1}^{\OO_m\times\OO_{n-m}}}\Omega^{m+1}\left(\TOP_n/
\TOP_{n,m}\right).
\end{equation}
In case $n=m+2\neq 4$, the equivalences (\ref{eq:b_top_refine1}-\ref{eq:b_top_refine2}-\ref{eq:b_top_refine3}) 
 hold if the left-hand sides are replaced by
$\Emb_\partial(D^m,D^{m+2})^\times$, $\Emb^{fr}_\partial(D^m,D^{m+2})^\times$, $\Ebar_\partial(D^m,D^{m+2})^\times$, respectively.
\end{theorem}

\begin{proof}[Proof of \eqref{eq:b_top_refine1}]
One has the following zigzag of $E_m^{\OO_m\times\OO_{n-m}}$-algebras:
\begin{equation}\label{eq:zigzag_refine1}
\Embmn^\times\xrightarrow{\,\simeq\,}\tEmb_\partial^{\underline{V_{n,m}}}(D^m,D^n)^\times\xleftarrow{\,\simeq\,}
\Omega^m\hofiber(V_{n,m}\to \VV_{n,m}^t).
\end{equation}
The first map is an equivalence from Subsection~\ref{sss:embmn_tembmn}. The second map is the fiber inclusion over 
$(*\colon D^m\subset D^n)\in \tEmbmn^\times$. It is an equivalence since the base space $\tEmbmn^\times=\tEmbmn_*$
is contractible by Proposition~\ref{p:disc_emb}(b) and Lemma~\ref{l:alex}.
\end{proof}

\begin{proof}[Proof of \eqref{eq:b_top_refine2}]
Following Budney~\cite{Budney_cubes} we replace $\Embfrmn$ by the homotopy equivalent space 
$\Emb_{\partial D^m\times D^{n-m}}(D^m\times D^{n-m})$. One has the following zigzag of $E_{m+1}^{\OO_m\times\OO_{n-m}}$-algebras:
\begin{multline}\label{eq:zigzag_refine2}
\Emb_{\partial D^m\times D^{n-m}}(D^m\times D^{n-m})^\times\xrightarrow{\,\simeq\,}
\tEmb_{\partial D^m\times D^{n-m}}^{\underline{O_n}}(D^m\times D^{n-m})^\times   \xleftarrow{\,\simeq\,}\\
 \xleftarrow{\,\simeq\,}
\tEmb_{\partial D^m\times D^{n-m}}^{\underline{O_n}}(D^m\times D^{n-m}\mmod D^m\times 0)\xrightarrow{\,\simeq\,} \\
\xrightarrow{\,\simeq\,} \tImm_{\partial D^m\times D^{n-m}}^{\underline{O_n}}(D^m\times D^{n-m}\mmod D^m\times 0)
\xrightarrow{\,\simeq\,} \Omega^m\left(O_n\times^h_{\TOP_n}\TOP_{n,m}\right).
\end{multline}
The first map is the equivalence~\eqref{eq:emb_dm_dn} from Subsection~\ref{sss:emb_dm_dn}. The second map is the fiber 
inclusion over $(*\colon D^m\times 0\subset D^m\times D^{n-m})\in
\tEmb_\partial(D^m\times 0,D^m\times D^{n-m})^\times$.
It is an equivalence since the base space $\tEmb_\partial(D^m\times 0,D^m\times D^{n-m})^\times=\tEmb_\partial(D^m\times 0,D^m\times D^{n-m})_*$ is contractible. The third map is an equivalence since $\tImm_{\partial D^m\times D^{n-m}}(D^m\times D^{n-m}\mmod D^m\times 0)$ deformation retracts onto  $\tEmb_{\partial D^m\times D^{n-m}}(D^m\times D^{n-m}\mmod D^m\times 0)$. 
 The last (scanning) map is an equivalence by Proposition~\ref{p:top_nm}, equivalence \eqref{eq:top_nm2}.

To finish the proof, we change the category from simplicial sets to topological spaces, using the fact that the natural evaluation map (see Section~\ref{ss:loops}) 
\[
S_*\Omega^m\left(|O_n|\times^h_{|\TOP_n|}|\TOP_{n,m}|\right)
\xrightarrow{\,\simeq\,}\Omega^m\left(O_n\times^h_{\TOP_n}\TOP_{n,m}\right)
\]
is an equivalence of algebras over the singular chains (operad) of $E_{m+1}^{|\OO_m|\times|\OO_{n-m}|}$ (where
$|\OO_m|\times|\OO_{n-m}|$ acts on $E_{m+1}$ by restriction through the group homomorphism 
$|\OO_m|\times|\OO_{n-m}|\to \OO_m\times\OO_{n-m}$ -- in particular, the action of $|\OO_{n-m}|$ is trivial), 
while the latter operad is equivalent to~$E_{m+1}^{\OO_m\times\OO_{n-m}}$. We then apply Proposition~\ref{l:E1equiv}, which implies
\begin{multline*}
|O_n|\times^h_{|\TOP_n|}|\TOP_{n,m}|\simeq_{E_1^{|\OO_m|\times|\OO_{n-m}|}}\Omega\bigl(|\OO_n|\bbslash|\TOP_n|/|\TOP_{n,m}|\bigr)
\simeq_{E_1^{|\OO_m|\times|\OO_{n-m}|}}\\
\simeq_{E_1^{|\OO_m|\times|\OO_{n-m}|}}
\Omega\bigl|\OO_n\bbslash\TOP_n/\TOP_{n,m}\bigr|
\end{multline*}
(the action of $|\OO_m|\times|\OO_{n-m}|$ on $E_1$ being trivial).
\end{proof}

\begin{proof}[Proof of \eqref{eq:b_top_refine3}]
We replace $\Ebar_\partial(D^m,D^n)$ by the equivalent space $\Ebar_{\partial D^m\times D^{n-m}}(D^m\times D^{n-m})$.
We have the following zigzag of $E_{m+1}^{\OO_m\times\OO_{n-m}}$-algebras
\begin{multline}\label{eq:proof_top_refine3}
\Ebar_{\partial D^m\times D^{n-m}}(D^m\times D^{n-m})^\times\xrightarrow{\,\simeq\,} 
\mathrm{t}\Ebar_{\partial D^m\times D^{n-m}}(D^m\times D^{n-m})^\times\xleftarrow{\,\simeq\,}\\
\mathrm{t}\Ebar_{\partial D^m\times D^{n-m}}(D^m\times D^{n-m}\mmod D^m\times 0)
\xrightarrow{\,\simeq\,} 
\mathrm{t}\overline{\Imm}_{\partial D^m\times D^{n-m}}(D^m\times D^{n-m}\mmod D^m\times 0)\\
\xrightarrow{\,\simeq\,} 
\Omega^m\hofiber\left(\TOP_{n,m}\to\TOP_n\right).
\end{multline}
The first map is an equivalence since the map 
\[
\Ebar_{\partial D^m\times D^{n-m}}(D^m\times D^{n-m})\xrightarrow{\,\simeq\,} 
\mathrm{t}\Ebar_{\partial D^m\times D^{n-m}}(D^m\times D^{n-m})
\]
(involving all path-components) is one by Theorem~\ref{th:square_emb}(a). The last thee maps are equivalences by the same argument 
as in the proof of~\eqref{eq:b_top_refine2}. 

To finish the proof we switch to the topological category by replacing the last space by $\Omega^m\hofiber\bigl(
|\TOP_{n,m}|\to|\TOP_{n}|\bigr)$. The latter space is $E_{m+1}^{|\OO_m|\times|\OO_{n-m}|}$-equivalent to
\[
\Omega^{m+1}\bigl(|\TOP_n|/|\TOP_{n,m}|\bigr)=\Omega^{m+1}\bigl|\TOP_n/\TOP_{n,m}\bigr|
\]
by Proposition~\ref{l:E1equiv}.
\end{proof}

\section{Proof of Theorem~\ref{th:b} (PL part) and Theorem~\ref{th:d4} ($n=4$)}\label{s:proof_B_PL}
\begin{proof}[Proof of~\eqref{eq_th_b_3} and~\eqref{eq_th_d4_3}]
One has a zigzag
\begin{multline}\label{eq:ebar_zigzag2}
\Ebar_{\partial D^m\times D^{n-m}}(D^m\times D^{n-m})\xrightarrow{\,\simeq\,}\mathrm{pd}\Ebar_{\partial D^m\times D^{n-m}}(D^m\times D^{n-m})\\
\xleftarrow{\,\simeq\,} \mathrm{pl}\Ebar_{\partial D^m\times D^{n-m}}(D^m\times D^{n-m}),
\end{multline}
where the first map is an equivalence by Theorem~\ref{th:square_emb}(b).
We restrict this zigzag to the unions of path-components, which lie in the preimage of the component $\plEmb_\partial(D^m\times 0,D^m\times D^{n-m})_*$ of the trivial PL knot
$*\colon D^m\times 0\subset D^m\times D^{n-m}$:
\begin{multline}\label{eq:ebar_zigzag3}
\Ebar_{\partial D^m\times D^{n-m}}(D^m\times D^{n-m})^\star\xrightarrow{\,\simeq\,}\mathrm{pd}\Ebar_{\partial D^m\times D^{n-m}}(D^m\times D^{n-m})^\star\\
\xleftarrow{\,\simeq\,} \mathrm{pl}\Ebar_{\partial D^m\times D^{n-m}}(D^m\times D^{n-m})^\star.
\end{multline}
We recall that for $n\neq 4$, these path-components coincide with the invertible elements in $\pi_0$ (i.e., $\star=\times$). Moreover, for $n-m\neq 2$, $(n,m)\neq (4,3)$, the zigzags~\eqref{eq:ebar_zigzag2} and~\eqref{eq:ebar_zigzag3}
are the same.

We then use the fact that the space   $\plEmb_\partial(D^m\times 0,D^m\times D^{n-m})_*$    is contractible (by the Alexander trick Proposition~\ref{p:disc_emb}(c)) and continue the zigzag:
\begin{multline}\label{eq:proof_th_b_pl}
\mathrm{pl}\Ebar_{\partial D^m\times D^{n-m}}(D^m\times D^{n-m})^\star\xleftarrow{\,\simeq\,}
\mathrm{pl}\Ebar_{\partial D^m\times D^{n-m}}(D^m\times D^{n-m}\mmod D^m\times 0)\\
\xrightarrow{\,\simeq\,}  
\mathrm{pl}\overline{\Imm}_{\partial D^m\times D^{n-m}}(D^m\times D^{n-m}\mmod D^m\times 0)
\xrightarrow{\,\simeq\,} 
\Omega^m\hofiber\left(\PL_{n,m}\to\PL_n\right).
\end{multline}
We then similarly as in the proof of~\eqref{eq:b_top_refine3} (at the end of Section~\ref{s:proof_B_TOP}) pass to the 
category of topological spaces and conclude that the last space in~\eqref{eq:proof_th_b_pl} is equivalent
to $\Omega^{m+1}|\PL_n/\PL_{n,m}|$. 

The difference of the PL argument is that the spaces appearing in the zigzags (and maps between them) are not $E_{m+1}$
(not even $E_1$ in fact), but only $H$-spaces with an associative product at the level of $\pi_0$. Note however, that the 
zigzags~\eqref{eq:ebar_zigzag3} and~\eqref{eq:proof_th_b_pl} are parallel to~\eqref{eq:proof_top_refine3} allowing us to
define a morphism between the two. As a conclusion, when $n\neq 4$, the natural maps $\PL_n\to\TOP_n$ and $\PL_{n,m}\to\TOP_{n,m}$
induce an equivalence $\Omega^{m+1}|\PL_n/\PL_{n,m}|\to \Omega^{m+1}|\TOP_n/\TOP_{n,m}|$, which is an equivalence
of $E_{m+1}$-spaces.
\end{proof}

\begin{proof}[Proof of~\eqref{eq_th_b_2} and~\eqref{eq_th_d4_2}]
We need to show that 
\begin{equation}\label{eq:framed_deloop}
\Embfrmn^\star\simeq\Omega^{m+1}\left(\OO_n\bbslash\PD_n/\PL_{n,m}\right)
\end{equation}
and that this equivalence is $E_{m+1}$ for $n\neq 4$. One has a zigzag of equivalences
\begin{multline}\label{eq:zigzag_fr_pd1}
\Emb_{\partial D^{m}\times D^{n-m}}(D^{m}\times D^{n-m})^\star \xrightarrow{\,\simeq\,}
\pdEmb_{\partial D^{m}\times D^{n-m}}^{\underline{O_n}}(D^{m}\times D^{n-m})^\star
\xleftarrow{\,\simeq\,} \\
\xleftarrow{\,\simeq\,}\pdEmb_{\partial D^{m}\times D^{n-m}}^{\underline{O_n}}(D^{m}\times D^{n-m}\mmod D^{m}\times 0)
\xrightarrow{\,\simeq\,}\\
\xrightarrow{\,\simeq\,}\Omega^{m}\left(O_n\times^h_{\PD_n}\PD_{n,m}\right)
\xleftarrow{\,\simeq\,} \Omega^{m}\left(\OO_n\times^h_{\PD_n}\PL_{n,m}\right) ,
\end{multline}
where the first inclusion is the equivalence~\eqref{eq:emb_fr_temb4}. One also has 
\begin{multline}\label{eq:zigzag_fr_pd2}
|\OO_n\times^h_{\PD_n}\PL_{n,m}| 
\simeq\left|\OO_n\right|\times^h_{|\PD_n|}\left|\PL_{n,m}\right|
\xrightarrow{\,\simeq\,}
\hofiber\left(|\OO_n|\to |\PD_n|/|\PL_{n,m}|\right)
\xleftarrow{\,\simeq\,}     \\
\xleftarrow{\,\simeq\,}  \hofiber\Bigl(|\OO_n|\to \bigl(|E\OO_{n+1}\times |\PD_n|\bigr)/|\PL_{n,m}|\Bigr) 
\xrightarrow{\,\simeq\,}
\Omega\left(|\OO_n|\bbslash|\PD_n|/|\PL_{n,m}|\right).
\end{multline}
This together with \eqref{eq:zigzag_fr_pd1} completes the proof of~\eqref{eq:framed_deloop}. As in the proof of~\eqref{eq_th_b_3}, the
$E_{m+1}$ property for $n\neq 4$ follows from the TOP version of~\eqref{eq_th_b_2}.
\end{proof}

\begin{proof}[Proof of~\eqref{eq_th_b_1}]
Recall that there is no natural map $\OO_n\to\PL_n$ or $\OO_{n-m}\to\PL_{n,m}$. Instead one has  roundabout maps
$B\OO_n\to B\PL_n$ and $B\OO_{n-m}\to B\PL_{n,m}$, see  footnote~\ref{foot10}, 
  which by taking loops can then be used to produce the spaces appearing in Theorem~\ref{th:b}.
 To avoid this technical annoyance, one can restate~\eqref{eq_th_b_1} 
and  prove instead that 
\begin{multline}
\Embmn^\times\simeq_{E_{m}} \Omega^m\hofiber\left(\TOP_{n,m}/\OO_{n-m}\to\TOP_n/\OO_{n}\right)\\
\simeq_{E_{m}}
\Omega^m\hofiber\left(\PD_{n,m}/\OO_{n-m}\to\PD_n/\OO_{n}\right),\,\, n\neq 4.\label{eq_th_b_1'}
\end{multline}
(For the same reason we do not consider the PL version of the $\OO_m\times\OO_{n-m}$-equivariant refinement of Theorem~\ref{th:b}.)
 However, a direct proof of the PL part of the latter equivalence is harder 
since in order to compare a space of smooth embeddings to a space of PL ones, one needs an intermediate space of PD embeddings, while such spaces for a positive codimension have not been considered  in the literature.
 Fortunately, in most of the cases the PL part of Theorem~\ref{th:b}\eqref{eq_th_b_1} immediately follows from its TOP part. Indeed, for $n\geq 5$, $n-m\geq 3$, $\PL_n/\PL_{n,m}
\simeq \TOP_n/\TOP_{n,m}$ \cite[Proposition~$(t/pl)$]{Lashof}, which immediately implies the PL part of~\eqref{eq_th_b_1} 
in this range of $(n,m)$.

Similarly, the codimension one case $n=m+1\neq 4$ follows from Lemma~\ref{l:top_n_n-1} ($\TOP_{n,n-1}\simeq \OO_1\simeq\PL_{n,n-1}$) and the fact that 
$$
\TOP_n/\PL_n
\simeq\begin{cases} K(\Z/2\Z,3),& n\geq 5;\\
*,& n\leq 3;
\end{cases}
$$
 see \cite[Essay~V, \S5, Subsection~5.0]{KS_essays}.

Now consider the case $n=m+2\neq 4$.  
In the range $n=m+2>4$,  \eqref{eq_th_b_2} implies \eqref{eq_th_b_1}. Indeed, for $n>4$, one has $\Emb_\partial(D^{n-2},D^n)\simeq
\Emb_\partial^{fr}(D^{n-2},D^n)$ (which implies $\Emb_\partial(D^{n-2},D^n)^\times\simeq
\Emb_\partial^{fr}(D^{n-2},D^n)^\times$). Similarly, the right-hand sides of \eqref{eq_th_b_1} are equivalent to those of~\eqref{eq_th_b_2} since $\pi_i \OO_2=0$ for $i>1$.

We are left to show the PL part of~\eqref{eq_th_b_1} in the classical case $(n,m)=(3,1)$. It is known that $\TOP_3\simeq
\OO_3\simeq\PL_3$, see the Introduction. From the facts that
\[
\Emb_\partial(D^1,D^3)^\times\simeq *\simeq \Omega\,\hofiber\left(\TOP_{3,1}/\OO_2\to \TOP_3/\OO_3\right)
\simeq\Omega(\TOP_{3,1}/\OO_2)
\]
and $\pi_*(\TOP_{3,1}/\OO_2)=0$ for $*\leq 1$ \cite{KS_codim2}, one gets that $\TOP_{3,1}\simeq\OO_2$. 
To finish the proof we only need to show that $\PL_{3,1}\simeq\OO_2$. 
We know that
\[
\Ebar_\partial(D^1,D^3)^\times\simeq\Omega^2S^2\simeq\Omega^2(\TOP_3/\TOP_{3,1})\simeq\Omega^2(\PL_3/\PL_{3,1}).
\]
Thus, $\pi_i(\PL_3/\PL_{3,1})=\pi_i(\TOP_3/\TOP_{3,1})=\pi_i(\TOP_3/\PL_{3,1})$ for $i>1$. From the fiber sequence
\[
\TOP_{3,1}/\PL_{3,1}\to \TOP_3/\PL_{3,1} \to \TOP_3/\TOP_{3,1},
\]
one gets $\pi_i(\TOP_{3,1}/\PL_{3,1})=0$ for $i>1$. As a consequence, $\pi_i\PL_{3,1}=\pi_i\TOP_{3,1}=
\pi_i\OO_2=0$ for $i>1$, and $\pi_1\PL_{3,1}\to\pi_1\TOP_{3,1}=\pi_1\OO_2=\Z$ is injective. 
Because of the natural maps $\OO_2\to\PL_{3,1}\to\TOP_{3,1}$ (or rather $\OO_2\to\PD_{3,1}\to\TOP_{3,1}$),
whose composition is an equivalence,
the map $\pi_1\PL_{3,1}\to\pi_1\TOP_{3,1}$ is an isomorphism. Thus, $\PL_{3,1}\simeq\OO_2$.
\end{proof}

\begin{proof}[Proof of~\eqref{eq_th_d4_1}]
We show that~\eqref{eq_th_d4_1} follows from~\eqref{eq_th_d4_2} and~\eqref{eq_th_d4_3}.

\smallskip

\noindent{Case $(n,m)=(4,3)$.} 
 In this situation
$\Emb_\partial(D^3,D^4)^\star\simeq \Emb_\partial^{fr}(D^3,D^4)^\star.$ 
From~\eqref{eq_th_d4_2},
\[
 \Emb_\partial^{fr}(D^3,D^4)^\star\simeq \Omega^3\hofiber(\PD_{4,3}\to\PD_4/\OO_4)\simeq\Omega^3\hofiber(\PD_{4,3}/\OO_1\to\PD_4/\OO_4),
\]
which implies~\eqref{eq_th_d4_2}.


\smallskip

\noindent{Case $(n,m)=(4,2)$.} 
Since any smooth knot $D^2\hookrightarrow D^4$ has a Seifert surface~\cite{Levine}, its normal bundle is trivial. Thus,
$\Emb_\partial(D^2,D^4)\simeq \Emb_\partial^{fr}(D^2,D^4)$ implying the same for PL-trivial knots:
\begin{equation}\label{eq:42}
\Emb_\partial(D^2,D^4)^\star\simeq \Emb_\partial^{fr}(D^2,D^4)^\star.
\end{equation}
To prove~\eqref{eq_th_d4_1}, we only must show that
\[
\Omega^2\hofiber(\PD_{4,2}\to\PD_4/\OO_4)\simeq\Omega^2\hofiber(\PD_{4,2}/\OO_2\to\PD_4/\OO_4).
\]
One has a fiber sequence
\[
\OO_2\to\hofiber(\PD_{4,2}\to\PD_4/\OO_4)\to\hofiber(\PD_{4,2}/\OO_2\to\PD_4/\OO_4).
\]
Since $\pi_i\OO_2=0$ for $i\geq 2$, the map
\[
\pi_i\hofiber(\PD_{4,2}\to\PD_4/\OO_4)\to\pi_i\hofiber(\PD_{4,2}/\OO_2\to\PD_4/\OO_4)
\]
is bijective for $i\geq 3$ and injective for $i=2$. Thus, we are left to show bijectivity for $i=2$ or, equivalently,
bijectivity of the map $\pi_2\PD_{4,2}\to\pi_2(\PD_{4,2}/\OO_2)$ (since $\PD_4/\OO_4$ is 4-connected by~\cite[Theorem~8.3C]{FQ}). We claim that $\pi_1\OO_2\to\pi_1\PD_{4,2}$ is 
injective, which implies what we need. This holds because $\OO_2$ is a homotopy retract of $\PD_{4,2}$, the composition
$\OO_2\to\PD_{4,2}\to \mathrm{G}_2$ being a homotopy equivalence. 

\smallskip

\noindent{Case $(n,m)=(4,1)$.} 
The space $\Emb_\partial(D^1,D^4)$ is connected and so is $\Ebar_{\partial D^1\times D^3}(D^1\times D^3)\simeq\Ebar_\partial(D^1,D^4)\simeq \Emb_\partial(D^1,D^4)
\times\Omega^2 S^3$,
see~\cite[Proposition~5.17]{Sinha}. 
To prove ~\eqref{eq_th_d4_1}, 
we argue that
\[
\Emb_\partial(D^1,D^4)\simeq\Omega\,\hofiber\left(\PD_{4,1}/\OO_3\to \PD_4/\OO_4\right).
\]
The space $\Ebar_{\partial D^1\times D^3}(D^1\times D^3)$ viewed as a monoid with respect to composition,
acts on itself and on the equivalent space $\mathrm{pd}\Ebar_{\partial D^1\times D^3}(D^1\times D^3)$ by post-composition. This
action can be restricted to the action of its submonoid 
\[
\Ebar_{\partial D^1\times D^3}(D^1\times D^3\mmod D^1\times 0)\simeq \Omega\,\hofiber(\OO_3\to\OO_4)
\simeq\Omega^2 S^3.
\]
The space $\Ebar_{\partial D^1\times D^3}(D^1\times D^3)$ quotiented out by this action is equivalent to
$\Emb_\partial(D^1,D^4)$. 
 On the other hand, the space $\mathrm{pd}\Ebar_{\partial D^1\times D^3}(D^1\times D^3)$
is equivalent to 
\[
\mathrm{pd}\Ebar_{\partial D^1\times D^3}(D^1\times D^3\mmod D^1\times D^0)
\simeq \Omega\,\hofiber(\PD_{4,1}\to\PD_4)
\]
 as an $\bigl( \Omega\,\hofiber(\OO_3\to\OO_4)\bigr)$-module. We claim that
\begin{multline*}
\frac{ \Omega\,\hofiber(\PD_{4,1}\to\PD_4)}{\Omega\,\hofiber(\OO_3\to\OO_4)}
\simeq
\Omega\left(\frac{\hofiber(\PD_{4,1}\to\PD_4)}{\hofiber(\OO_3\to\OO_4)}\right)
\simeq \Omega\,\hofiber\left(\PD_{4,1}/\OO_3\to\PD_4/\OO_4\right).
\end{multline*}
(This finishes the proof of \eqref{eq_th_d4_1}.)
Let $G$ be a group and $X$ be a cofibrant $G$-space. One has $\Omega X/\Omega G\simeq\Omega(X/G)$
if and only if the connecting homomorphism $\pi_1(X/G)\to\pi_0 G$ is zero. In our case 
\[
\pi_0G=\pi_0\hofiber(\OO_3\to\OO_4)=\pi_0\Omega S^3=\pi_1 S^3=0.
\]
Thus, the first equivalence holds. For the second equivalence, we use the fact that $\Omega(X/G)\simeq\hofiber(G\to X)$ as well as the following diagram in which
every row and every column is a fiber sequence:
$$
\xymatrix{
\hofiber(\OO_3\to \OO_4)\ar[d]\ar[r]&\OO_3\ar[r]\ar[d]&\OO_4\ar[d]\\
\hofiber(\PD_{4,1}\to\PD_4)\ar[r]\ar[d]& \PD_{4,1}\ar[r]\ar[d]&\PD_4\ar[d]\\
\hofiber\left(\PD_{4,1}/\OO_3\to\PD_{4}/\OO_4\right)\ar[r]&\PD_{4,1}/\OO_3\ar[r]&\PD_4/\OO_4.
}
$$
\end{proof}

\section{Proof of Theorem~\ref{th:c}}\label{s:proof_c}
We will prove a slight refinement of Theorem~\ref{th:c}.

\begin{theorem}\label{th:c_refined}
For $1\leq m\leq n\neq 4$, one has the following equivariant equivalences:
\begin{multline}
\Embfrmn^\times\simeq_{\OO_m{\times}\OO_{n-m}}\Embfr(S^m,S^n)^\times/\OO_{n+1}\\
\simeq_{\OO_{m+1}{\times}\OO_{n-m}}\Omega^{m+1}\left(\tEmb(S^m,S^n)/\!\!/\OO_{n+1}\right).
\label{eq_c_refined}
\end{multline}
Moreover, in the same range, $\Embfrmn^\times\simeq_{E_{m+1}^{\OO_m{\times}\OO_{n-m}}} \Omega^{m+1}\left(\tEmb(S^m,S^n)/\!\!/\OO_{n+1}\right)$.
\end{theorem}

The group $\OO_{m+1}\times\OO_{n-m}$ acts on $\Embfr(S^m,S^n)^\times$ by pre-compositions, while $\OO_{n+1}$ acts by post-compositions. The two actions commute inducing an 
$\OO_{m+1}\times\OO_{n-m}$ action on the quotient $\Embfr(S^m,S^n)^\times/\OO_{n+1}$. It is not hard to see that $\Embfrmn^\times$ can be $\OO_{m}\times\OO_{n-m}$-equivariantly
included in $\Embfr(S^m,S^n)^\times/\OO_{n+1}$. To describe the $\OO_{m+1}\times\OO_{n-m}$ action on $\Omega^{m+1}\left(\tEmb(S^m,S^n)/\!\!/\OO_{n+1}\right)$, we refer to Subsecton~\ref{ss:O_action_loops},
which explains how  such an action on a space induces an action on its loop space. Since we take loops, it is enough to consider the path-component of $\left(\tEmb(S^m,S^n)/\!\!/\OO_{n+1}\right)$
of the basepoint, which is
\begin{multline}
\tEmb(S^m,S^n)^\times/\!\!/\OO_{n+1} ={}_{\OO_{n+1}}\!\!{\text{\large$\backslash$}}\!\!\left(E\OO_{n+1}\times \tEmb(S^m,S^n)^\times\right)=\\
_{\OO_{n+1}}\!\!{\text{\large$\backslash$}}\!\!\left[E\OO_{n+1}\times\bigl(\Homeo(S^n)/\Homeo(S^n\mmod S^m)\bigr)\right].\label{eq:c_refined2}
\end{multline}
(We denote by $\tEmb(S^m,S^n)^\times$ the path-component of the trivial knot if $n>m$, and the entire space $\Homeo(S^n)$ if $n=m$. In case $n-m\neq 2$, $\tEmb(S^m,S^n)^\times=\tEmb(S^m,S^n)$.
By 
Proposition~\ref{p:spher_emb}, $\tEmb(S^m,S^n)^\times=\Homeo(S^n)/\Homeo(S^n\mmod S^m)$.)

We now describe the $\OO_{m+1}{\times}\OO_{n-m}$ action on~\eqref{eq:c_refined2}.
First, we define the topological $E\OO_{n+1}$ as a countable join of topological $\OO_{n+1}$:
\begin{equation}\label{eq:EO}
E\OO_{n+1}:=\OO_{n+1}*\OO_{n+1}*\OO_{n+1}*\ldots.
\end{equation}
pointed in the unit of the first $\OO_{n+1}$. 
For the simplicial set $E\OO_{n+1}$ we take the singular chains on~\eqref{eq:EO}. Note that $E\OO_{n+1}$ has a left and a right actions by $\OO_{n+1}$. The two actions can be combined to the conjugation action. The latter
is \lq\lq{}more natural\rq\rq{} than the right action. By this we mean the following. If we have a functorial construction of $EG$ (with respect to the left $G$-action) for any simplicial (or topological) group $G$, then the  automorphisms of $G$ act on $EG$, which corresponds to the conjugation action for inner automorphisms. In particular, 
the conjugation action can be extended to a larger group of all automorphisms of~$G$. For $\rho\in\OO_{n+1}$, $e\in E\OO_{n+1}$, we denote by $\rho\cdot e$, $e\cdot\rho$, $\rho(e):=\rho\cdot e\cdot\rho^{-1}$, the left, right,
and conjugation actions, respectively. Note that for the basepoint $*\in E\OO_{n+1}$, the left and right actions agree: $\rho\cdot{*}={*}\cdot\rho$ implying that the conjugation action preserves the basepoint $\rho(*)=*$.
Second, we define the $\OO_{n+1}$- and $\OO_{m+1}\times\OO_{n-m}$-actions on $E\OO_{n+1}\times \tEmb(S^m,S^n)^\times$ as follows. For $\rho\in\OO_{n+1}$, $(\lambda,\mu)\in\OO_{m+1}\times\OO_{n-m}$,
$e\in E\OO_{n+1}$, $f\in \tEmb(S^m,S^n)^\times$, we set $\rho\cdot(e,f):=(\rho\cdot e,\rho\circ f)$ and $(\lambda,\mu)\cdot(e,f):=(e\cdot(\lambda^{-1},\mu^{-1}),f\circ\lambda^{-1})$.
In this definition, the two actions commute and produce an $\OO_{m+1}\times\OO_{n-m}$-action on the quotient~\eqref{eq:c_refined2}. Alternatively, one can choose an $\OO_{m+1}\times\OO_{n-m}$-action on $E\OO_{n+1}\times \tEmb(S^m,S^n)^\times$ to be defined as  $(\lambda,\mu)\cdot(e,f):=\bigl((\lambda,\mu)(e),(\lambda,\mu)\circ f\circ\lambda^{-1})$.
In this definition, the two actions do not commute, but one still gets a well-defined (in fact, exactly the same) $\OO_{m+1}\times\OO_{n-m}$-action on the quotient~\eqref{eq:c_refined2}. The reader may prefer
the second definition as categorically more rounded. In fact, in the Appendix we will consider a more general situation and will use the second definition of the action on the quotient. However, until the end of this section we stick to the first one as it is more convenient  for the proof.

\begin{proof}[Proof of Theorem~\ref{th:c_refined}]
The second statement is equivalent to identity~\eqref{eq:b_top_refine2} of Theorem~\ref{th:b_top_refine}, see Remark~\ref{r:deloop}.
We thus focus on the first statement, which is proved by the following zigzag of $\OO_{m+1}{\times}\OO_{n-m}$-equivariant equivalences:
\begin{multline}
\left.\Embfr(S^m,S^n)^\times\right/\OO_{n+1}\xrightarrow[A]{\simeq}
\left.\tEmb^{\underline{O_n}}(S^m,S^n)^\times\right/\OO_{n+1}\xleftarrow[B]{\simeq}\\
\xleftarrow[B]{\simeq}
 \left.\text{\lq\lq{}}\left(\tEmb^{\underline{O_n}}(S^m,S^n)^\times\right)\text{\rq\rq{}}\right/\OO_{n+1}\cong
\left.{\mathrm{RadMap}}\bigl( (D^{m+1},S^m),(\tEmb(S^m,S^n)^\times,\OO_{n+1})\bigr)\right/\OO_{n+1}   \xrightarrow[C]{\simeq}\\
 \xrightarrow[C]{\simeq} \left.\Map\bigl( (D^{m+1},S^m),(\tEmb(S^m,S^n)^\times,\OO_{n+1})\bigr)\right/\OO_{n+1}  \xleftarrow[D]{\simeq}  \\
\xleftarrow[D]{\simeq} \left. \Map\bigl( (D^{m+1},S^m),(E\OO_{n+1}{\times}\tEmb(S^m,S^n)^\times,\OO_{n+1})\bigr)\right/\OO_{n+1} \xrightarrow[E]{\simeq}\\
 \xrightarrow[E]{\simeq}\Omega^{m+1}\left(   \left.\left(E\OO_{n+1}\times \tEmb(S^m,S^n)^\times\right)\right/\OO_{n+1}\right).
\label{eq:long_zigzag}
\end{multline}
We will gradually explain what are the corresponding simplicial sets, how $\OO_{m+1}\times\OO_{n-m}$ acts on them, what are the maps and why they are equivalences. The maps $A$-$B$-$C$-$D$  
are equivariant equivalences
(with respect to the commuting $\OO_{n+1}$ and $\OO_{m+1}\times\OO_{n-m}$-actions) before taking the quotient by $\OO_{n+1}$. We denote by $\tilde A$-$\tilde B$-$\tilde C$-$\tilde D$  
the corresponding maps before taking quotient.
For simplicity of exposition we describe only the vertices of the corresponding simplicial sets. We refer the reader to Subsection~\ref{ss:spaces} and Section~\ref{s:loops} of conventions for mapping spaces, loop spaces and pullbacks
to recover the general $k$-simplices of the simplicial sets involved.

The space $\Embfr(S^m,S^n)^\times$ consists of pairs $(f,F)$, where $f\colon S^m\hookrightarrow S^n$ is a smooth topologically trivial embedding, while $F$ is a fiberwise isomorphism of vector bundles lifting~$f$:
\begin{equation}\label{eq:embfr_sq}
\xymatrix{
TS^m\times\R^{n-m}\,\ar@{^{(}->}[rr]^-F \ar[d]_{\pi\circ p_1}&& TS^n\ar[d]^\pi\\
S^m\,\,\ar@{^{(}->}[rr]^f&&S^n.
}
\end{equation}
Moreover, it is assumed that $F|_{TS^m\times 0}=D(f)$. Note that $TS^m\times\R^{n-m}=TS^n|_{S^m}$. The simplicial set $\Embfr(S^m,S^n)^\times$ has commuting actions of smooth chains  $d\OO_{n+1}$ and 
$d\OO_{m+1}\times d\OO_{n-m}$. Any $\rho\in \OO_{n+1}$ acts by post-composition with $(\rho,D(\rho))$, while $(\lambda,\mu)\in\OO_{m+1}\times\OO_{n-m}$ acts by pre-composition with $(\lambda^{-1},D(\lambda)^{-1}\times \mu^{-1})$:
\[ 
\xymatrix{
TS^n\,\ar[r]^{D(\rho)} \ar[d]_{\pi}& TS^n\ar[d]^\pi\\
S^n\,\,\ar[r]^{\rho}&S^n,
}
\qquad\qquad
\xymatrix{
TS^m\times\R^{n-m}\,\ar[rr]^-{D(\lambda)^{-1}\times\mu^{-1}} \ar[d]_{\pi\circ p_1}&& TS^m\times\R^{n-m}\ar[d]^{\pi\circ p_1}\\
S^m\,\,\ar[rr]^{\lambda^{-1}}&&S^m.
}
\] 

The space $\tEmb^{\underline{O_n}}(S^m,S^n)^\times$ (defined in Section~\ref{s:fr_embed} and its Subsection~\ref{sss:spherical_fr_emb}) of ${\underline O_n}$-framed topological embeddings consists
of triples $(f,F,H)$, where $f\colon S^m\hookrightarrow S^n$ is a locally flat topological embedding (required to be isotopic to the trivial one when $n-m=2$); $F$ is its lift to a fiberwise isomorphism of vector bundles~\eqref{eq:embfr_sq}, and
$H$ is an isotopy of locally flat monomorphisms of topological microbundles
\begin{equation}\label{eq:embfr_sq2}
\xymatrix{
TS^m\times[0,1]\,\ar@{^{(}->}[rr]^-H \ar[d]_{\pi\times id_{[0,1]}}&& TS^n\times[0,1]\ar[d]^{\pi\times id_{[0,1]}}\\
S^m\times[0,1]\,\,\ar@{^{(}->}[rr]^{f\times id_{[0,1]}}&&S^n\times [0,1]
}
\end{equation}
between $H|_{t=0}=F|_{TS^m\times 0}$ and $H|_{t=1}=exp_{S^n}^{-1}\circ\Delta(f)\circ exp_{S^m}$. The action of $\rho\in\OO_{n+1}$ is the post-composition with $\bigl(\rho,D(\rho),D(\rho)\times id_{[0,1]}\bigr)$, while
the action of $(\lambda,\mu)\in \OO_{m+1}\times\OO_{n-m}$ is the pre-composition with $\bigl(\lambda^{-1}, D(\lambda)^{-1}\times\mu^{-1}, D(\lambda)^{-1}\times id_{[0,1]}\bigr)$.

The map $\tilde A\colon \Embfr(S^m,S^n)^\times\to \tEmb^{\underline{O_n}}(S^m,S^n)^\times$ is the equivalence~\eqref{eq:spherical_fr_emb}. It is $\OO_{m+1}\times\OO_{n-m}$- and $\OO_{n+1}$-equivariant since
$exp_{S^k}\colon TS^k\to S^k\times S^k$ is $\OO_{k+1}$-equivariant, where the right-hand side is endowed with the diagonal action. 

We note that starting from $\left.\tEmb^{\underline{O_n}}(S^m,S^n)^\times\right/\OO_{n+1}$, all spaces in the zigzag~\eqref{eq:long_zigzag} are acted upon by continuous chains $t\OO_{m+1}\times t\OO_{n-m}$ and the quotient is also 
taken by continuous chains~$t\OO_{n+1}$.

The space $\text{\lq\lq{}}\left(\tEmb^{\underline{O_n}}(S^m,S^n)^\times\right)\text{\rq\rq{}}$ is a \lq\lq{}version\rq\rq{} of $\tEmb^{\underline{O_n}}(S^m,S^n)^\times$. It is obtained by replacing each space
of germs of locally flat topological embeddings $T_xS^m\hookrightarrow T_{f(x)}S^n$ by the equivalent space $\tEmb(S^m,S^n\mmod x)^\times$ of (trivial) locally flat topological embeddings $S^m\hookrightarrow S^n$ sending
$x\in S^m$ to $f(x)\in S^n$. Indeed, the first space is the topological Stiefel manifold $\VV_{n,m}^t=\TOP_n/\TOP_{n,m}$, while the second one is equivalent to $\Homeo(S^n\mmod *)/\Homeo(S^n\mmod S^m)$.
One has $\Homeo(S^n\mmod S^m)\simeq\TOP_{n,m}$ by the equivalence~\eqref{eq:top_nm} and the natural surjection $\Homeo(S^n\mmod *)\to\TOP_n$ is an equivalence since its kernel $\Homeo_*(S^n)$ -- 
the group of homeomorphisms of $S^n$, which are identity near $*\in S^n$, is isomorphic to the group of 
compactly supported homeomorphisms of~$\R^n$, the latter being contractible by Remark~\ref{r:alex}. The vertices of $\text{\lq\lq{}}\left(\tEmb^{\underline{O_n}}(S^m,S^n)^\times\right)\text{\rq\rq{}}$ are triples
$(f,F,H)$, where $f\in \tEmb(S^m,S^n)^\times$, $F$ is a map lifting $f$:
\begin{equation}\label{eq:embfr_sq3}
\xymatrix{
S^m\times S^n\,\ar@{^{(}->}[r]^-F \ar[d]_{p_1}& S^n\times S^n\ar[d]^{p_1}\\
S^m\,\,\ar@{^{(}->}[r]^f&S^n,
}
\end{equation}
such that for all $x\in S^m$,
 $F(x,x)=(f(x),f(x))$ 
  and $\bigl(F_2(x,-)\colon S^n\to S^n\bigr)\in\OO_{n+1}$; 
  $H$ is
a topological locally flat embedding lifting $f\times id_{[0,1]}$ (with a stronger condition that $id_{S^m}\times H_2\times id_{[0,1]}\colon S^m\times S^m\times [0,1]\hookrightarrow S^m\times S^n\times [0,1]$ is a topological locally flat embedding\footnote{Compare with footnote~\ref{footnote_imm}.}):
\begin{equation}\label{eq:embfr_sq4}
\xymatrix{
S^m\times S^m\times [0,1]\,\ar@{^{(}->}[r]^-H \ar[d]_{p_1\times p_3}& S^n\times S^n\times [0,1]\ar[d]^{p_1\times p_3}\\
S^m\times [0,1]\,\,\ar@{^{(}->}[r]^{f \times [0,1]}&S^n\times [0,1], 
}
\end{equation}
such that $H(x,x,t)=(f(x),f(x),t)$ 
 for any $x\in S^m$ and $t\in[0,1]$, and $H|_{t=0}=F|_{S^m\times S^m}$, $H|_{t=1}=\Delta(f)$ (this implies that each embedding $H_2(x,-,t)\colon S^m\hookrightarrow S^n$ 
is isotopic to the trivial one). The action of $\rho\in \OO_{n+1}$ and that of  $(\lambda\times\mu)\in\OO_{m+1}\times\OO_{n-m}$ are as before by post- and pre-composition with 
$(\rho,\rho\times\rho,\rho\times\rho\times id_{[0,1]})$ and $\bigl(\lambda^{-1},\lambda^{-1}\times (\lambda^{-1},\mu^{-1}),\lambda^{-1}\times\lambda^{-1}\times id_{[0,1]}\bigr)$, respectively.

The map $\tilde B$ sends
\[
\tilde B\colon (f,F,H)\mapsto \bigl(f, exp^{-1}_{S^n}\circ F\circ\left. exp_{S^n}\right|_{\left(TS^m\times \R^{n-m}=TS^n|_{S^m}\right)}, (exp_{S^n}^{-1}\times id_{[0,1]})\circ H\circ (exp_{S^m}\times id_{[0,1]})\bigr).
\]
and is an equivalence. 
The space $\text{\lq\lq{}}\left(\tEmb^{\underline{O_n}}(S^m,S^n)^\times\right)\text{\rq\rq{}}$ is homeomorphic (i.e., bijectively isomorphic) to the space ${\mathrm{RadMap}}\bigl( (D^{m+1},S^m),(\tEmb(S^m,S^n)^\times,\OO_{n+1})\bigr)$ of {\it radial} maps of pairs 
$$(D^{m+1},S^m)\to (\tEmb(S^m,S^n)^\times,\OO_{n+1}).$$ The latter is a simplicial subset of the simplicial set $\Map\bigl( (D^{m+1},S^m),(\tEmb(S^m,S^n)^\times,\OO_{n+1})\bigr)$
of all maps of pairs defined as  the pullback 
\begin{equation}\label{eq:embfr_sq5}
\xymatrix{
\Map\bigl( (D^{m+1},S^m),(\tEmb(S^m,S^n)^\times,\OO_{n+1})\bigr)\,\ar[r]\ar[d]&\Map(S^m,\OO_{n+1})\ar[d]\\
\Map\bigl(D^{m+1}, \tEmb(S^m,S^n)^\times\bigr)\,\,\ar[r]&\Map\bigl(S^m, \tEmb(S^m,S^n)^\times\bigr).
}
\end{equation}
The $k$-simplices of $\Map(S^m,\OO_{n+1})$ are continuous maps $S^m\times\Delta^k\to\OO_{n+1}$. The $k$-simplices of $\Map\bigl(Y, \tEmb(S^m,S^n)^\times\bigr)$, where $Y=D^{m+1}$ or $S^m$,
are continuous maps $F_k\colon S^m\times Y\times\Delta^k\to S^n$, such that $(F_k,p_2,p_3)\colon S^m\times Y\times \Delta^k\hookrightarrow S^n\times Y\times \Delta^k$ is a topological locally flat embedding
(with additional triviality condition of $F_k(-,y_0,t_0):S^m\hookrightarrow S^n$ for some $(y_0,t_0)\in Y\times \Delta_k$ if $n-m=2$). The right arrow in~\eqref{eq:embfr_sq5} is induced by the composition
\[
\OO_{n+1}\to\Homeo(S^n)\to \Homeo(S^n)/\Homeo(S^n\mmod S^m)=\tEmb(S^m,S^n)^\times.
\]

The vertices of $\Map\bigl( (D^{m+1},S^m),(\tEmb(S^m,S^n)^\times,\OO_{n+1})\bigr)$ are pairs $(\Phi,\Psi)$, where $\Phi\colon D^{m+1}\times S^m\to S^n$ is a continuous map, such that 
$(p_1,\Phi)\colon D^{m+1}\times S^m\hookrightarrow D^{m+1}\times S^n$ is a topological locally flat embedding; and $\Psi\colon S^m\times S^n\to S^n$ is a continuous map such that 
$\bigl(\Psi(x,-)\colon S^n\to S^n\bigr)\in\OO_{n+1}$ for all $x\in S^m$ and $\Phi|_{\partial D^{m+1}\times S^m}=\Psi|_{S^m\times S^m}$.

The {\it radial} condition, which determines the subset $\mathrm{RadMap}(\ldots)$ is $\Phi(tx,x)=\Phi(0,x)$ for all $x\in S^m$, $t\in[0,1]$. The bijection with  
$\text{\lq\lq{}}\left(\tEmb^{\underline{O_n}}(S^m,S^n)^\times\right)\text{\rq\rq{}}$  sends 
$$(\Phi,\Psi)\in{\mathrm{RadMap}}\bigl( (D^{m+1},S^m),(\tEmb(S^m,S^n)^\times,\OO_{n+1})\bigr)$$         
to $(f,F,H)\in \text{\lq\lq{}}\left(\tEmb^{\underline{O_n}}(S^m,S^n)^\times\right)\text{\rq\rq{}}$, where $f(x)=\Phi(0,x)$, $F(x,z)=(\Phi(0,x),\Psi(x,z))$, $H(x,y,t)=\bigl(\Phi(0,x),\Phi((1-t)x,y),t\bigr)$ for 
all $x,y\in S^m$, $z\in S^n$, $t\in[0,1]$.

Elements $\rho\in \OO_{n+1}$ and $(\lambda,\mu)\in\OO_{m+1}\times\OO_{n-m}$ act on 
$$(\Phi,\Psi)\in\Map\bigl( (D^{m+1},S^m),(\tEmb(S^m,S^n)^\times,\OO_{n+1})\bigr)$$ 
as follows:
$(\Phi,\Psi)\mapsto (\rho\circ\Phi,\rho\circ\Psi)$, $(\Phi,\Psi)\mapsto \bigl(\Phi\circ \left(\lambda^{-1}\times\lambda^{-1}\right),\left(\Psi\circ (\lambda^{-1}\times (\lambda^{-1},\mu^{-1})\right)\bigr)$.

The inclusion $\tilde C$ is an equivalence due to the pullback square:
\begin{equation}\label{eq:embfr_sq6}
\xymatrix{
\mathrm{RadMap}\bigl( (D^{m+1},S^m),(\tEmb(S^m,S^n)^\times,\OO_{n+1})\bigr)\,\,\ar@{->>}[r]\ar@{^{(}->}[d]^{\tilde C}&\tEmb(S^m,S^n)^\times\ar@{^{(}->}[d]^\simeq\\
\Map\bigl( (D^{m+1},S^m),(\tEmb(S^m,S^n)^\times,\OO_{n+1})\bigr)\,\,\ar@{->>}[r]&\tEmb(S^m,S^n)^\times\times^h_{\Map(S^m,S^n)}\Map(S^m,S^n),
}
\end{equation}
where the upper arrow sends $(\Phi,\Psi)\mapsto\Phi(0,-)$, the right arrow sends an embedding $f\in \tEmb(S^m,S^n)^\times$ to the constant path~$f$, and the bottom arrow sends $(\Phi,\Psi)$ to the
map $S^m\times [0,1]\to S^n$, $(x,t)\mapsto \Phi(tx,x)$, which is an element in $\tEmb(S^m,S^n)^\times$ when restricted on $S^m\times 0$. Since the right arrow is an equivalence and the bottom one is a Kan fibration,
the left arrow $\tilde C$ is also an equivalence.

To describe the simplicial set 
\begin{equation}\label{eq:embfr7}
\Map\bigl( (D^{m+1},S^m),(E\OO_{n+1}{\times}\tEmb(S^m,S^n)^\times,\OO_{n+1})\bigr),
\end{equation}
 we define $\Map(Y,E\OO_{n+1})$, $Y=D^{m+1}$ or $S^m$, as the simplicial set 
of singular chains on the topological space of continuous maps $Y\to E\OO_{n+1}$, where $E\OO_{n+1}$ is the countable join~\eqref{eq:EO}. Note that $E\OO_{n+1}$ is endowed with both left and right $\OO_{n+1}$-actions.
The map $\OO_{n+1}\xrightarrow{\iota}E\OO_{n+1}$, used in the definition of the pair $(E\OO_{n+1}{\times}\tEmb(S^m,S^n)^\times,\OO_{n+1})$, sends the basepoint to the basepoint and, therefore, it preserves both 
$\OO_{n+1}$-actions. Vertices of~\eqref{eq:embfr7} are triples $(\Phi_1,\Phi,\Psi)$, where $\Phi$ and $\Psi$ are the same as before, while $\Phi_1$ is a continuous map $D^{m+1}\to E\OO_{m+1}$
satisfying $\iota\circ\Psi=\Phi_1|_{\partial D^{m+1}}$.

Elements $\rho\in \OO_{n+1}$ and $(\lambda,\mu)\in\OO_{m+1}\times\OO_{n-m}$ act on the components $\Phi$ and $\Psi$ of $(\Phi_1,\Phi,\Psi)$ in the same way as before, while they change $\Phi_1$ to the compositions:
$D^{m+1}\xrightarrow{\Phi_1} E\OO_{n+1}\xrightarrow{\rho\cdot(-)}E\OO_{n+1}$ and $D^{m+1}\xrightarrow{\lambda^{-1}}D^{m+1}\xrightarrow{\Phi_1} E\OO_{n+1}\xrightarrow{(-)\cdot (\lambda^{-1},\mu^{-1})}E\OO_{n+1}$, respectively.

By construction, the $\OO_{n+1}$-action on~\eqref{eq:embfr7} extends to a free  $\Map(D^{m+1},\OO_{n+1})$-action (where $\OO_{n+1}$ is included in $\Map(D^{m+1},\OO_{n+1})$ as constant maps).
The induced quotient map
\begin{multline*}
\left. \Map\bigl( (D^{m+1},S^m),(E\OO_{n+1}{\times}\tEmb(S^m,S^n)^\times,\OO_{n+1})\bigr)\right/\OO_{n+1}\to\\
\left. \Map\bigl( (D^{m+1},S^m),(E\OO_{n+1}{\times}\tEmb(S^m,S^n)^\times,\OO_{n+1})\bigr)\right/\Map(D^{m+1},\OO_{n+1})
\end{multline*}
is the equivalence $E$, since the target of this map is nothing but $\Omega^{m+1} \left(   \left.\left(E\OO_{n+1}\times \tEmb(S^m,S^n)^\times\right)\right/\OO_{n+1}\right)$. It is an equivalence since the groups $\OO_{n+1}$ and
$\Map(D^{m+1},\OO_{n+1})$ are weakly equivalent, while the action by the latter is cofibrant.
\end{proof}

\appendix

\section{Loop spaces of group quotients}\label{s:appendix}
The last step in the proofs of~\eqref{eq:a_top_1} in Subsection~\ref{ss:a_top}, \eqref{eq:b_top_refine2} and~\eqref{eq:b_top_refine3} in Section~\ref{s:proof_B_TOP}, requires changing the ground category
of simplicial sets to topological spaces. We will assume the latter is the category of compactly generated spaces. The products, subspaces and mapping spaces in this category are defined by taking kelleyfication of the usual topology 
in the corresponding spaces. Note that kelleyfication does not change the weak homotopy type as it induces  a bijection of simplicial sets of singular chains.  The main advantages of this category are that it is cartesian closed and finite products of any CW-complexes in it are CW-complexes~\cite[Appendix]{Lewis}, \cite{Vogt}. In particular the realization of any simplicial group (or an operad in simplicial sets) is a group (or an operad, respectively).
The model structure on this category that we adapt here is considered in~\cite[Section~2.4]{Hovey}.

\begin{proposition}\label{l:E1equiv}
Let $H,K\subset G$ be subgroups of a topological group~$G$, and let $N\subset N(H)\cap N(K)$ be a subgroup of the intersection of their normalizers. We additionally assume that all these (sub)groups are cofibrant as spaces
and all the underlying group inclusions are cofibrations of spaces.\footnote{For example, $G$ is a Lie group, while $H,K,N$ are its closed subgroups. As another example relevant to us, $G,H,K,N$ can be realizations of simplicial (sub)groups.} One then has an equivalence
of $E_1$-algebras in $N$-spaces:
$H\times_G^h K\simeq_{E_1^N}\Omega\left(H\bbslash G/K\right)$.
\end{proposition}

In the above, the framed operad $E_1^N$ is produced from the trivial action of $N$ on $E_1$, see Subsection~\ref{ss:fr_operads}.
The action of $N$ on $H\times_G^h K$ is by pointwise conjugation $n\cdot (h,\alpha(t),k)=(nhn^{-1},n\alpha(t)n^{-1},nkn^{-1})$, $n\in N$, $(h,\alpha,k)\in H\times_G^h K$.
The $E_1$-action on $H\times_G^h K$  is via its associative monoidal structure of a group, the product being $(h_1,\alpha_1(t),k_1)\cdot (h_2,\alpha_2(t),k_2)=(h_1h_2,\alpha_1(t)\alpha_2(t),k_1k_2)$.
For the right-hand side, the $E_1$-structure is the standard one of a loop space, while the $N$-action on $H\backslash(EH\times G)/K$ is defined as $n\cdot[ (e,g)]:=[(n(e),ngn^{-1})]$, where $e\in EH$, $g\in G$, $n\in N$.
Here, $EH$ is defined as the countable join $EH:=H*H*H*\ldots$ pointed in the unit of the first $H$. The action of $N$ is by automorphisms of $H$. (The group $Aut(H)$ of automorphisms of $H$ acts on $EH$ diagonally, while the conjugation 
by $N\subset N(H)$ induces  a map $N\to Aut(H)$.) The action is well-defined. Indeed, if we take a different representative $(h\cdot e,h\cdot g\cdot k)$ of $[(e,g)]\in H\backslash(EH\times G)/K$, then 
\[
\left(n(h\cdot e),n(hgk)n^{-1}\right)=\left(nhn^{-1}\cdot n(e), (nhn^{-1})\cdot (ngn^{-1})\cdot (nkn^{-1})\right) \sim (n(e),ngn^{-1}).
\]
The $N$-action in question preserves the basepoint, and therefore, produces a well-defined pointwise action on the loop space $\Omega\left(H\bbslash G/K\right)$.

\begin{proof}
For the proof it is convenient to use $[0,1]$ for $D^1$  in the definition of the operad $E_1$ and that of loop spaces $\Omega X$. 
The proof is contained in the following zigzag of $E_1^N$-equivalences:
\begin{multline}\label{eq:zigzag_app}
H\times^h_G K\underset{(0)}{\cong} K\times^h_G H\underset{(1)}{\simeq}\left(K\times^h_G H\right)'\xrightarrow[(2)]{\simeq}
\hofiber\left(G/K\leftarrow H\right)\xleftarrow[(3)]{\simeq}\\
\xleftarrow[(3)]{\simeq}\hofiber\left(EH\times G/K\leftarrow H\right)\xrightarrow[(4)]{\simeq}\Omega\bigl(H\backslash (EH\times G)/K\bigr).
\end{multline}
The map (0) here is the group homeomorphism $(h,\alpha(t),k)\mapsto (k,\alpha(1-t),h)$. 
The third space is $K\times^h_G H$ with the same $N$-action, but with a different $E_1$-structure. The action and the zigzag
proving~(1)
are explained in Subsection~\ref{ss:a1}. For a map of spaces $f\colon X\to Y$, where $Y$ is pointed in $*\in Y$, we denote by
$\hofiber(Y\xleftarrow{f} X)$ the space homeomorphic to $\hofiber(X\stackrel{f}{\to} Y)$, but formally defined as the space of pairs $(\alpha,x)$,
where $x\in X$, $\alpha\colon [0,1]\to Y$, such that $\alpha(0)=*$, $\alpha(1)=f(x)$.
The fourth and fifth spaces are defined by means of this construction. The map $H\to EH\times G/K$ is induced by the inclusion
of $H$ in $EH$ as the first summand in the join, and by the inclusion $H\subset G$. Note that both maps are equivariant with respect to the left $H$-action and also with respect to the action of $N$ on $H, G/K, EH$ by conjugation. The $E_1$-action on the fourth and fifth spaces is explained in Subsection~\ref{ss:a2}. The maps (2)-(3)-(4) 
 are obvious quotient maps, which can easily be checked to be equivalences.
\end{proof}

\subsection{Equivalence $(1)$}\label{ss:a1}
There is a slightly larger operad $\tilde E_1^N$ acting on $K\times^h_G H$, where the operad $\tilde E_1:=\Assoc^{E_1(1)}$, 
while the monoid $E_1(1)$ (of strictly increasing linear maps $[0,1]\to[0,1]$) and the group $N$ are acting trivially on $\Assoc$ and on $\tilde E_1$, respectively. For $(L\colon [0,1]\to [0,1])\in E_1(1)$, its left action on $(k,\alpha,h)\in
K\times^h_G H$
is defined as $L\cdot (k,\alpha,h)=(k,L\cdot\alpha,h)$, where 
$$
(L\cdot\alpha)(t)=
\begin{cases}
k,&t\leq L(0);\\
\alpha(L^{-1}(t)),& t\in L([0,1]);\\
h,&t\geq L(1).
\end{cases}
$$
It is easy to see that this $E_1(1)$-action commutes with the $\Assoc^N$-action.

The operads $\Assoc$, $E_1$, and $\tilde E_1$ are equivalent to each other by means of the following operad maps:
\[
\xymatrix{
\Assoc\ar@/^/[rr]^{i_0}&&\tilde E_1\ar@/^/[ll]^{q_0}\\
& E_1\ar[ul]^{p_0}\ar[ur]_{j_0}&.
}
\]
The maps $p_0$, $q_0$,  and $i_0$ are the obvious projections and inclusion, respectively. The inclusion $j_0$
sends $\vec L=(L_1,\ldots,L_k)\in E_1(k)$ to $\left(p_0(\vec L);L_1,\ldots,L_k\right)\in \Assoc(k)\times E_1(1)^{\times k}=
\tilde E_1(k)$.

We denote by $p$, $q$, $i$, $j$ the maps between $\Assoc^N$, $E_1^N$, $\tilde E_1^N$ induced by $p_0$, $q_0$, $i_0$,
and $j_0$, respectively.

For a morphism of operads $\xi\colon P_1\to P_2$, we denote by $\xi_*$ and $\xi^*$ the induced  induction and restriction
functors for the corresponding categories of algebras:
\begin{equation}\label{eq:xi}
\xi_*\colon Alg^{P_1}\rightleftarrows Alg^{P_2}\colon \xi^*.
\end{equation}

We now define $\left(K\times^h_G H\right)'$ from zigzag~\eqref{eq:zigzag_app} as $j^*\left(K\times^h_G H\right)$.
Equivalence~(1) from the zigzag immediately follows from the following lemma applied for $Z=K\times^h_G H$.

\begin{lemma}\label{l:equiv1}
For any $\tilde E_1^N$-algebra $Z$, one has an equivalence of $E_1^N$-algebras: $p^*i^*(Z)\simeq j^*(Z).$
\end{lemma}

\begin{proof}
Since the group $N$ is cofibrant as a space, all operad morphisms $i$, $j$, $p$, $q$ are equivalences of $\Sigma$-cofibrant operads
and by \cite[Theorem~4.4]{BM} they induce Quillen equivalences~\eqref{eq:xi} of the corresponding model categories of algebras.

The statement of the lemma is proved by the following zigzag of equivalences of $E_1^N$-algebras:
\begin{equation}\label{eq:z_a1}
j^*Z\xleftarrow{\simeq} j^*Z^c\xrightarrow{\simeq} 
 j^*q^*q_* Z^c
= p^*i^*q^*q_* Z^c\xleftarrow{\simeq} p^*i^*Z^c\xrightarrow{\simeq}p^*i^*Z.
\end{equation}
All restriction functors, such as $j^*$, $p^*$, $i^*$, preserve weak equivalences.
The first and the last maps are induced by a cofibrant replacement $Z^c\xrightarrow{\simeq} Z$.  The second and the next to the last 
equivalences are induced by the unit of the adjunction~\eqref{eq:xi} for $\xi=q$: $Z^c\xrightarrow{\simeq}q^*q_* Z^c$,
which is an equivalence since the adjunction is a Quillen equivalence. The middle identity is due to $q\circ i\circ p=q\circ j$.
\end{proof}

\subsection{Equivalences (2)-(3)-(4)}\label{ss:a2}
The maps (2)-(3)-(4) 
 in  zigzag~\eqref{eq:zigzag_app} are  $N$-equivariant equivalences. We still have to define $E_1$-action on 
$\hofiber\left( G/K\leftarrow H\right)$ and $\hofiber\left(EH\times G/K\leftarrow H\right)$.

Given a space $X$ with a left $H$-action, a basepoint $*\in X$ induces an $H$-equivariant map $H\to X$, which sends
the group unit to the basepoint. We now define a natural $E_1$-action on the homotopy fiber $\hofiber(X\leftarrow H)$
over~$*$. Let $\vec L=(L_1,L_2,\ldots,L_k)\in E_1(k)$ and $\sigma=p_0(\vec L)\in \Assoc(k)=\Sigma_k$ (it is a permutation that
tells in which order the segments $L_i([0,1])$, $i=1\ldots k$, appear in~$[0,1]$). 
Then for $(\alpha_1,h_1),\ldots,(\alpha_k,h_k)\in \hofiber(X\leftarrow H)$, one defines $\vec L\bigl((\alpha_1,h_1),\ldots
(\alpha_k,h_k)\bigr)$ as $(\beta,h_{\sigma_1}h_{\sigma_2}\cdots h_{\sigma_k})$, where 
\[
\beta(t)=
\begin{cases}
*,& t\leq L_{\sigma_1}(0);\\
h_{\sigma_1}h_{\sigma_2}\cdots h_{\sigma_{i-1}},& L_{\sigma_{i-1}}(1)\leq t\leq L_{\sigma_i}(0);\\
h_{\sigma_1}h_{\sigma_2}\cdots h_{\sigma_k},& t\geq L_{\sigma_k}(1);\\
h_{\sigma_1}h_{\sigma_2}\cdots h_{\sigma_{i-1}}\cdot\alpha_{\sigma_i}(L_{\sigma_i}^{-1}(t)),& t\in L_{\sigma_i}([0,1]).
\end{cases}
\]

With this definition of $E_1$-action on $\hofiber(X\leftarrow H)$ for $X=G/K$ and~$EH\times G/K$, one can easily check that the maps
(2)-(3)-(4) in~\eqref{eq:zigzag_app} are morphisms of $E_1$-algebras.


\begin{thebibliography}{99}

\bibitem{Alex} J. W. Alexander.\
  On the deformation of an n-cell. \
\emph{Proc. Nat. Acad.  Sc. of USA} 9 (12): 406-407, 1923.

\bibitem{AT1} G. Arone, V. Turchin. \
On the rational homology of high dimensional analogues of spaces of long knots.\
 \emph{Geom. Topol.}\
 18 (2014), no.~3, 1261-1322.

%

\bibitem{Berger} C. Berger.\
Un modèle simplicial fibrant de l'espace des chemins.\
A simplicial fiber model of the path space.\
\emph{C.~R.~Acad. Sci. Paris Sér. I Math.} 315 (1992), no. 2, 193-196.

\bibitem{BM}
C. Berger, I. Moerdijk.\
 Axiomatic homotopy theory for operads. \
 \emph{Comment. Math. Helv.} 78, no.~4, 805-831, 2003.




\bibitem{BoardVogt}
J. M. Boardman, R. M. Vogt.\
Homotopy invariant algebraic structures on topological spaces.\
\emph{Lecture Notes in Mathematics}, Vol. 347. Springer-Verlag, Berlin-New York, 1973. x+257~pp.


\bibitem{BW_conf_cat} Pedro Boavida de Brito, Michael Weiss.\
 Spaces of smooth embeddings and configuration categories.\
 \emph{J. Topol.} 11 (2018), no.~1, 65-143. 


\bibitem{BW1}  Pedro Boavida de Brito, Michael Weiss.\
The configuration category of a covering space,\
arXiv:2312.17631.

\bibitem{BW2}  Pedro Boavida de Brito, Michael Weiss.\
Presentations of configuration categories,\
arXiv:2312.17632.

\bibitem{BW3}  Pedro Boavida de Brito, Michael Weiss.\
The torus trick for configuration categories,\
arXiv:2401.00799.

\bibitem{Brakes}
W. R. Brakes.\
Quickly unknotting topological spheres.\
\emph{Proc. Amer. Math. Soc.} 72 (1978), no.~2, 413-416.

\bibitem{Brown} Morton Brown.\
A proof of the generalized Schoenflies theorem,\
\emph{Bull. Amer. Math. Soc.} 66 (1960), 74-76.

\bibitem{Budney_cubes}
Ryan Budney.\
 Little cubes and long knots,\
 \emph{Topology} 46 (2007), 1-27.

\bibitem{Budney_family}
Ryan Budney.\
A family of embedding spaces,\ \emph{Geom. Topol. Monogr.} 13 (2008), 41-83.

\bibitem{Budney_splicing}
Ryan Budney.\ An operad for splicing.\ 
\emph{J. Topol.} 5 (2012), no.~4, 945-976.

\bibitem{Budney_stab}
Ryan Budney.\
Stabilisation, scanning and handle cancellation.\
\emph{Enseign. Math.} 71 (2025), no.~3-4, 411-431.



\bibitem{BG2}
Ryan Budney, David Gabai. \
On the automorphism groups of hyperbolic manifolds,\
\emph{Int. Math. Res. Not. IMRN} 2025, no.~7, Paper No.~rnaf083, 28 pp.

\bibitem{BG1}
Ryan Budney, David Gabai. \
Knotted 3-balls in $S^4$,\
arXiv:1912.09029v2.

\bibitem{BL_diff} Dan
Burghelea, Richard Lashof.\
 The homotopy type of the space of diffeomorphisms.~I,~II.\
 \emph{Trans. Amer. Math. Soc.}\
 196 (1974), 1-36; ibid. 196 (1974), 37-50.


\bibitem{BLR_aut} Dan Burghelea, Richard Lashof, Melvin Rothenberg.\
 Groups of automorphisms of manifolds. With an appendix (\lq\lq{}The topological category\rq\rq{}) by E. Pedersen.\
 \emph{Lecture Notes in Mathematics},\
 Vol.~473. Springer-Verlag, Berlin-New York, 1975. vii+156 pp.

\bibitem{BKK}
Mauricio Bustamante, Manuel Krannich, Alexander Kupers.\
Finiteness properties of automorphism spaces of manifolds with finite fundamental group,\
\emph{Math. Ann.} 388 (2024), no.~4, 3321-3371.


\bibitem{Cerf_lemme}
Jean Cerf.\ Sur les diff\'eomorphismes de la sph\`ere de dimension trois ($\Gamma_4=0$),\ 
\emph{Lecture
Notes in Math.}, 53, Springer, Berlin, 1968.

\bibitem{Cerf_diff} Jean Cerf.\
 La stratification naturelle des espaces des functions différentielles réelles et le théorème de pseudo-isotopie,\
 \emph{Inst. Hautes Études Sci. Publ. Math.} No.~39, 5-173 (1970).

\bibitem{Chernav} Alexey V. Černavskiĭ,\
 Topological embeddings of manifolds,\ \emph{Dokl. Akad. Nauk SSSR} 187 (1969), 1247-1250 (in Russian). 

%

\bibitem{DT} Julien Ducoulombier,  Victor Turchin.\
 Delooping the functor calculus tower.\ 
\emph{Proc. Lond. Math. Soc.} (3) 124 (2022), no.~6, 772-853.

\bibitem{DTW}  Julien Ducoulombier, Victor Turchin, Thomas Willwacher.\
On the delooping of (framed) embedding spaces.\ \emph{Trans. Amer. Math. Soc.} 374 (2021), no. 11, 7657-7677. 

\bibitem{Dunn}
Gerald Dunn.\
 Tensor product of operads and iterated loop spaces.\
 \emph{J. Pure Appl. Algebra} 50 (1988), no.~3, 237-258. 

\bibitem{EdwKirby} Robert
Edwards, Robion C.  Kirby.\
Deformations of spaces of imbeddings.
\emph{Ann. of Math. (2)} 93 (1971), 63-88.

\bibitem{Freedman}
Michael H. Freedman.\
The disk theorem for four-dimensional manifolds.\
\emph{Proceedings of the International Congress of Mathematicians}, Vol.~1,~2 (Warsaw, 1983), 647-663.
PWN---Polish Scientific Publishers, Warsaw, 1984

\bibitem{FQ}
Michael H. Freedman, Frank Quinn.\
Topology of 4-manifolds.\
\emph{Princeton Math. Ser.} 39
Princeton University Press, Princeton, NJ, 1990. viii+259 pp.

\bibitem{FTW}
B. Fresse, V. Turchin,  T. Willwacher.\ 
The rational homotopy of mapping spaces of
$E_n$ operads, arXiv:1703.06123v1.



\bibitem{Gauniyal} N. Gauniyal.\
Haefliger's approach for spherical knots modulo immersions,\
\emph{Homology Homotopy Appl.} 25 (2023), no.~2, 55-73.

\bibitem{Gluck} Herman Gluck.\
Unknotting $S^1$ in $S^4$.\
\emph{Bull. Amer. Math. Soc.} 69 (1963), 91-94.

\bibitem{GW}
 Thomas G. Goodwillie, Michael Weiss.\ Embeddings from the point of view of immersion theory.~II,\ 
\emph{Geom. Topol.} 3 (1999), 103-118.

\bibitem{Haefliger} 
Andr\'e Haefliger.\ Differential embeddings of $S^n$ in $S^{n+q}$
for $q > 2$,\ \emph{Ann. of Math. (2)} 83
(1966), 402-436.

\bibitem{HaeflPoen_imm}
Andr\'e Haefliger, Valentin Po\'enaru.\
 La classification des immersions combinatoires,\ 
\emph{Inst. Hautes Études Sci. Publ. Math.} 23 (1964), 75-91.

\bibitem{Hatcher_diff}
Allen Hatcher.\ A proof of the Smale conjecture, $\Diff(S^3)\simeq O(4)$.\ 
\emph{Ann. of Math.~(2)}  117 (1983), no. 3, 553-607. 

\bibitem{Hatcher_knots} Allen Hatcher.\
 Topological moduli spaces of knots,\\
{ \tt{http://www.math.cornell.edu.er.lib.k-state.edu/$\sim$hatcher/Papers/knotspaces.pdf} }

\bibitem{HW} 
Allen  Hatcher, John Wagoner.\
Pseudo-isotopies of compact manifolds.\
\emph{Astérisque}, No.~6,
Société Mathématique de France, Paris, 1973. i+275 pp.
%
%

\bibitem{Hirsch}
Morris W. Hirsch.\
Immersions of manifolds,\
\emph{Trans. Amer. Math. Soc.} 93 (1959), 242-276.

\bibitem{HirschMazur}
  Morris W. Hirsch, Barry Mazur.\
Smoothings of piecewise linear manifolds.
\emph{Annals of Mathematics Studies}, No. 80. Princeton University Press, Princeton, NJ; University of Tokyo Press, Tokyo, 1974. ix+134 pp.

\bibitem{Hovey}
Mark Hovey. \
Model categories. 
\emph{Mathematical Surveys and Monographs}, \
Vol.~63, American Mathematical Society, Providence, RI, 1999.

\bibitem{Hudson} John F. P. Hudson.\
Extending piecewise-linear isotopies,
\emph{Proc. London Math. Soc.} (3) 16 (1966), 651-668.

\bibitem{Kan} Daniel M. Kan.\
On homotopy theory and c.s.s. groups.\
\emph{Ann. of Math.} (2) 68 (1958), 38-53.

\bibitem{KervaireMilnor}
Michel A. Kervaire, John W. Milnor.\
Groups of homotopy spheres. I,\
\emph{Ann. of Math. (2)} 77 (1963), 504–537.

\bibitem{KS_essays} Robion C. 
 Kirby, Laurence C. Siebenmann.\
 Foundational essays on topological manifolds, smoothings, and triangulations. With notes by John Milnor and Michael Atiyah.\ \emph{Annals of Mathematics Studies}, No.~88.\
Princeton University Press, Princeton, NJ;\
 University of Tokyo Press, Tokyo, 1977. vii+355 pp. 

\bibitem{KS_codim2} Robion C. 
 Kirby, Laurence C. Siebenmann.\
Normal bundles for codimension~2 locally flat imbeddings.\
 Geometric topology (Proc. Conf., Park City, Utah, 1974), pp. 310-324, 
\emph{Lecture Notes in Math.}, Vol.~438, Springer, Berlin-New York, 1975.

\bibitem{Kister} J. M. Kister.\
 Microbundles are fibre bundles.\
 \emph{Ann. of Math.} (2), 80:190-199, 1964.

\bibitem{KraKupDisc}
Manuel Krannich, Alexander Kupers. \
The disc-structure space. \
\emph{Forum Math. Pi} 12 (2024), Paper No.~e26, 98pp.

\bibitem{KraKup}
Manuel Krannich, Alexander Kupers.\
Embedding calculus and Serre classes, II.\
Paper in work.

\bibitem{KraRand}
Manuel Krannich, Oscar Randal-Williams.\
Diffeomorphisms of discs and the second Weiss derivative of $BTop(-)$,
arXiv:2109.03500.

\bibitem{Kreck}
M. Kreck.\ Isotopy classes of diffeomorphisms of $(k - 1)$-connected almost-parallelizable
$2k$-manifolds, in Algebraic topology, Aarhus 1978 (Proc. Sympos., Univ. Aarhus,
Aarhus, 1978), 643-663, \emph{Lecture Notes in Math.}, Vol.~763, Springer, Berlin, 1979.

\bibitem{KuipLash1}
N. H.  Kuiper, R. K.  Lashof.\
Microbundles and bundles. I. Elementary theory,
\emph{Invent. Math.} 1 (1966), 1-17.

\bibitem{KuipLash2}
N. H.  Kuiper, R. K.  Lashof.\
Microbundles and bundles. II. Semisimplicial theory,
\emph{Invent. Math.} 1 (1966), 243-259.

\bibitem{Kupers}
Alexander Kupers.\  Some finiteness results for groups of automorphisms of manifolds.\ 
\emph{Geom. Topol.} 23 (2019), no. 5, 2277-2333.

\bibitem{KupRand}
Alexander Kupers, Oscar Randal-Williams.\
On diffeomorphisms of even-dimensional discs,\
\emph{J. Amer. Math. Soc.} 38 (2025), no.~1, 63-178.



\bibitem{Kurata}
Masahiro Kurata.\
Immersions of topological manifolds.\
\emph{J. Fac. Sci. Hokkaido Univ. Ser. I}\ 21 (1970/71), 237-247.



\bibitem{Lashof}
Richard Lashof.\
 Embedding spaces,\ 
\emph{Illinois J. Math.} 20 (1976), 144-154.

\bibitem{LR}
R. Lashof, M. Rothenberg.\
Microbundles and smoothing.\
\emph{Topology} 3 (1965), 357-388.

\bibitem{Lees}
J. A. Lees.\ Immersions and surgeries of topological manifolds,\
 \emph{Bull. Amer. Math. Soc.} 75 (1969), 529-534. 

\bibitem{Levine}
J. Levine.\ Unknotting spheres in codimension two,\ 
\emph{Topology} 4 (1965), 9-16.

\bibitem{Lewis} Lemoine G. Lewis. \
 The stable category and generalized Thom spectra, Thesis (Ph.D.), The University
of Chicago, 1978.


\bibitem{Lima}
Elon L. Lima.\
On the local triviality of the restriction map for embeddings, \
\emph{Comment. Math. Helv.} 38 (1964), 163-164.



\bibitem{Mather}
Michael Mather.\
Pull-backs in homotopy theory.\
\emph{Canadian J. Math.} 28 (1976), no.~2, 225-263.


\bibitem{May} J. P. May.\
The Geometry of Iterated Loop Spaces, 
\emph{Lecture Notes in Math.} 271, Springer, Berlin, 1972. viii+175~pp.

\bibitem{Mazur}
Barry Mazur.\
On embeddings of spheres,\
\emph{Bull. Amer. Math. Soc.} 65 (1959), 59-65.

\bibitem{Millett}
Keneth C. Millett.\ Piecewise linear embeddings of manifolds,\
 \emph{Illinois J. Math.} 19 (1975), 354-369. 

\bibitem{Milnor}
John Milnor.\
Topological manifolds and smooth manifolds. \
\emph{Proc. Internat. Congr. Mathematicians} (Stockholm, 1962), pp.~132-138, Inst. Mittag-Leffler, Djursholm, 1963.

%

\bibitem{Morlet} Claude Morlet.\
 Isotopie et pseudo-isotopie, \emph{C. R. Acad. Sci. Paris} Sér. A-B 266 (1968), A559-A560.

\bibitem{Morse}
Marston Morse.\
A reduction of the Schoenflies extension problem,\
\emph{Bull. Amer. Math. Soc.} 66 (1960), 113-115.

\bibitem{Munoz}
Samuel Muñoz-Echániz.\
A Weiss--Williams theorem for spaces of embeddings and the homotopy type of spaces of long knots,\
arXiv:2311.05541.




\bibitem{Palais} Richard S. Palais.\
Local triviality of the restriction map for embeddings,
\emph{Comment. Math. Helv.} 34 (1960), 305-312.


\bibitem{Papa} C. D. Papakyriakopoulos.\
 On Dehn's lemma and the asphericity of knots.\
 \emph{Ann. of Math.} (2) 66 (1957), 1-26.

\bibitem{Randal_diff}
Oscar Randal-Williams.\
Diffeomorphisms of discs.\
\emph{ICM -- International Congress of Mathematicians.}\
 Vol.~4. Sections 5-8, 2856-2878.

\bibitem{Randal_S2_S4}
Oscar Randal-Williams.\
Topological embeddings of $S^2$ in $S^4$,\\
{\tt{https://www.dpmms.cam.ac.uk/$\sim$or257/notes/Embeddings.pdf}}


\bibitem{Rourke}
C. P. Rourke.\
On conjectures of Smale and others concerning the diffeomorphism group of $S^n$ and on structure theorems.\
Mathematics Institute, Univ. of Warwick, 1972.



\bibitem{Sakai}
Keiichi Sakai.\
 Deloopings of the spaces of long embeddings.\
 \emph{Fund. Math.} 227 (2014), no.~1, 27-34.

\bibitem{Sakai_BV}
Keiichi Sakai,\
BV-structures on the homology of the framed long knot space.\
\emph{J. Homotopy Relat. Struct.} 11 (2016), no. 3, 425--441.


\bibitem{SakWat}
 K. Sakai, T. Watanabe.\ 1-loop graphs and configuration space integral for embedding
spaces,\ 
\emph{Math. Proc. Cambridge Philos. Soc.} 152 (2012), no. 3, 497-533.

\bibitem{SalvWahl}
Paolo Salvatore, Natalie Wahl.\
 Framed discs operads and Batalin-Vilkovisky algebras,\ \emph{Quart. J. Math.} 54 (2003), 213-231. 

\bibitem{Shane}
Julius L. Shaneson.\
 Embeddings with codimension two of spheres in spheres and $H$-cobordisms of $S^1\times S^3$.\
 \emph{Bull. Amer. Math. Soc.} 74 (1968), 972-974. 


\bibitem{Sieb_deform}
Laurence C. Siebenmann.\
Deformation of homeomorphisms on stratified sets. I, II.\
\emph{Comment. Math. Helv.} 47 (1972), 123-136; ibid. 47 (1972), 137-163.



\bibitem{Sinha}
Dev P. Sinha.\ Operads and knot spaces,\
 \emph{Journal of the American Mathematical Society} 19 (2006), no.~2, 461-486.

\bibitem{Smale_vietoris}
Stephen Smale.\
 A Vietoris mapping theorem for homotopy.\
 \emph{Proc. Amer. Math. Soc.}~8 (1957), 604-610. 

\bibitem{Smale_diff}
Stephen  Smale.\ Diffeomorphisms of the 2-sphere.\ 
\emph{Proc. Amer. Math. Soc.} 10 (1959), 621-626. 

\bibitem{Smale_imm}
Stephen
Smale.\
The classification of immersions of spheres in Euclidean spaces.\
\emph{Ann. of Math.~(2)}  69 (1959), 327–344.


\bibitem{Stallings}
John Stallings.\
On topologically unknotted spheres.\
\emph{Ann. of Math.~(2)} 77(1963), 490-503.

\bibitem{Vogt}
Rainer M. Vogt. \
Convenient categories of topological spaces for homotopy theory,\
 \emph{Arch. Math.} (Basel)
22 (1971), 545-555.

\bibitem{Wall_spheres}
C. T. C. Wall.\ Unknotting tori in codimension one and spheres in codimension two,\
\emph{Proc. Cambridge Philos. Soc.} 61 (1965), 659-664. 


\bibitem{Wall_codim2}
C. T. C. Wall.\ 
Locally flat PL submanifolds with codimension two.\
\emph{Proc. Cambridge Philos. Soc.} 63 (1967), 5-8.

\bibitem{Wall_surgery}
C. T. C. Wall.\ Surgery on compact manifolds, second edition,\ \emph{Math. Surveys Monogr.},
69, American Mathematical Society, Providence, RI, 1999. 


\bibitem{Watanabe}
Tadayuki Watanabe.\
Some exotic nontrivial elements of the rational homotopy groups of $\Diff(S^4)$,\
arXiv:1812.02448.




\bibitem{Weiss}
 M. Weiss. Embeddings from the point of view of immersion theory.~I,\ 
\emph{Geom. Topol.} 3
(1999), 67-101.

\bibitem{WeissErr}
 M. Weiss. 
Erratum to the article Embeddings from the point of view of immersion theory: Part~I.\
\emph{Geom. Topol.} 15 (2011), no. 1, 407-409.


\bibitem{WeissRationalP}
M. Weiss.
Rational Pontryagin classes of Euclidean fiber bundles.\
\emph{Geom. Topol.} 25 (2021), no.~7, 3351-3424.

\bibitem{Whitehead}
J. H. C. Whitehead.\
On $C^1$-complexes.\
\emph{Ann. of Math.}\
(2) 41 (1940), 809-824.

\bibitem{Will}
Thomas Willwacher.\
Little disks operads and Feynman diagrams.\
\emph{Proceedings of the International Congress of Mathematicians}.\
Rio de Janeiro 2018. \
Vol.~II. Invited lectures, 1241-1261.\
World Scientific Publishing Co. Pte. Ltd., Hackensack, NJ, 2018

\bibitem{Wright}
P. Wright.\
Covering isotopies of $M^{n-1}$ in $N^n$.\
\emph{Proc. Amer. Math. Soc.} 29 (1971), 591-598.


\bibitem{Yoshioka} Leo Yoshioka.\
Some non-trivial cycles of the space of long embeddings detected by configuration space integral invariants using $g$-loop ($g =2, 3$) graphs.\
arXiv:2502.12547. 

\bibitem{Zeeman}
E. C. Zeeman.\
Unknotting combinatorial balls,\
\emph{Ann. of Math. (2)} 78 (1963), 501-526.

\end{thebibliography}
\end{document}